\pdfoutput=1 

\RequirePackage{fix-cm}
\documentclass[smallextended]{svjour3}
\usepackage[utf8]{inputenc}
\usepackage{booktabs}
\usepackage{amssymb,amsmath,epsfig,graphics,psfrag,graphicx,color,cite,arydshln,xcolor,grffile}

\renewenvironment{proof}{\noindent{\it Proof.}}{\hfill$\square$}

\definecolor{lightblue}{rgb}{0.22,0.45,0.70}
\definecolor{lightgreen}{rgb}{0.22,0.50,0.25}
\usepackage[right = 1.25cm, left=1.25cm, top = 2.cm, bottom =2.cm]{geometry}
\usepackage[colorlinks=true,breaklinks=true,linkcolor=lightblue,citecolor=lightblue]{hyperref}

\newtheorem{assumption}[theorem]{Assumption}
\numberwithin{equation}{section}
\numberwithin{figure}{section}
\numberwithin{table}{section}
\numberwithin{lemma}{section}
\numberwithin{corollary}{section}
\numberwithin{theorem}{section}
\numberwithin{remark}{section}

\allowdisplaybreaks

\newcommand{\rev}{}

\journalname{Arxiv}

\begin{document}
\title{Optimal error estimates of coupled and divergence-free virtual element methods for the  
Poisson--Nernst--Planck/Navier--Stokes equations}
\titlerunning{VEM for the Poisson--Nernst--Planck/Navier--Stokes equations}

\author{Mehdi Dehghan \and
  Zeinab Gharibi \and
  Ricardo Ruiz-Baier}
\authorrunning{Dehghan, Gharibi, Ruiz-Baier}
\institute{
Mehdi Dehghan (corresponding author),  Zeinab Gharibi \at  
Department of Applied Mathematics, Faculty of Mathematics and Computer Sciences, Amirkabir University of Technology (Tehran Polytechnic), No. 424, Hafez Ave.,
15914 Tehran, Iran\\
\email{mdehghan@aut.ac.ir,z90gharibi@aut.ac.ir}.
\and
Ricardo Ruiz-Baier \at
School of Mathematics, Monash University, 9 Rainforest Walk, Melbourne, VIC 3800, Australia;  World-Class Research Center ``Digital biodesign and personalized healthcare", Sechenov First Moscow State Medical University, Moscow, Russia; and Universidad Adventista de Chile, Casilla 7-D Chillan, Chile\\
\email{ricardo.ruizbaier@monash.edu}.
}

\date{24 January 2022}
\maketitle

\begin{abstract}
In this article, we propose and analyze a fully coupled, nonlinear, and energy-stable  virtual element method (VEM) for solving  the coupled Poisson-Nernst-Planck (PNP) and Navier--Stokes (NS) equations modeling microfluidic and electrochemical systems (diffuse transport of charged species within incompressible fluids coupled through electrostatic forces).  A mixed VEM is employed to discretize the NS equations whereas classical  VEM in primal form is used to discretize the PNP equations. The stability, existence and uniqueness of solution of the associated VEM are proved by fixed point theory. Global mass conservation and electric energy decay of the scheme are also proved. Also, we obtain unconditionally optimal error estimates for both the electrostatic potential and ionic concentrations of PNP equations in the $H^{1}$-norm, as well as for the velocity and pressure of NS equations in the $\mathbf{H}^{1}$- and $L^{2}$-norms, respectively. Finally, several numerical experiments are presented to support the theoretical analysis of convergence and to illustrate the satisfactory  performance of the method in simulating the onset of electrokinetic instabilities in ionic fluids, and studying how they are  influenced by different values of ion concentration and applied voltage. These tests relate to applications in the desalination of water. 
\end{abstract}


\section{Introduction and problem statement}
\subsection{Scope}
The coupled Poisson--Nernst--Planck (PNP)/Navier--Stokes (NS) equations (also known as the electron fluid dynamics equations) serve to describe mathematically the dynamical properties of electrically charged fluids, the motion of ions and/or molecules, and to represent the interaction with electric fields and flow patterns of incompressible fluids within cellular environments and occurring at diverse spatial and temporal scales (see, e.g., \cite{Jerome11}). Ionic concentrations are described by the \rev{Nernst--Planck} equations   (a  convection--diffusion--reaction system), the diffusion of the electrostatic potential is described by a generalized Poisson equation, and the NS equations describe the dynamics of incompressible fluids, neglecting magnetic forces. A large number of dedicated applications are possible with this set of equations as for example semiconductors, electrokinetic flows in electrophysiology, drug delivery into biomembranes, and many others (see, e.g., \cite{Choi05,Cioffi06,Dreyer13,Jerome08,Hu05,Lu10,Mauri15,wang17} and the references therein). 

The mathematical analysis  (in particular,  existence and uniqueness of   solutions) for the coupled PNP/NS equations is a challenging task, due to the coupling of different mechanisms and multiphysics (internal/external charges, convection--diffusion, electro--osmosis, hydrodynamics, and so on) interacting closely. Starting from the early works \cite{Jerome85,Park97}, where one finds the  well-posedness analysis and the study of other properties of steady-state PNP equations, a number of contributions have addressed the existence, uniqueness, and regularity of different variants of the coupled PNP/NS equations. See, for instance,  \cite{Jerome02,Ryham06,Schmuck09} and the references therein. 

Reliable computational results are also challenging to obtain, again due to the nonlinearities involved, the presence of solution singularities owing to some types of charges, as well as the multiscale nature of the underlying phenomena. Double layers in the electrical fields  near the \rev{liquid--solid} interface are key to capturing the onset of instabilities and  fine spatio-temporal resolution is required, whereas the patterns of ionic transport are on a much larger scale \cite{Kim21}. Although numerical methods of different types have been used by computational physicists  and biophysicists and other  practitioners  over many decades, the rigorous analysis of numerical schemes is somewhat more recent. In such a context, the analysis of standard finite element methods (FEMs)  as well as of mixed, conservative, discontinuous Galerkin, stabilized, weak Galerkin, and other variants have been established for PNP and coupled PNP/NS equations  \cite{Huadong17,Huadong18,Gharibi20,he_numpde17,he_jcam18,Kim21,Linga20,Andreas09,prohl10,xie20}.  Since the formulation of FEMs  requires  explicit knowledge of the basis functions, such  methods might be often limited (at least in their classical setting) to meshes with simple-geometrical shaped elements, e.g., triangles or quadrilaterals. This constraint is overcome by polytopal element methods such as the VEM, which are designed for providing arbitrary order of accuracy on polygonal/polytopal elements.  In the VEM setting,   the   explicit knowledge of the basis functions is not required, while its practical implementation is based on suitable projection operators  which are computable by their degrees of freedom. 

One of the main purposes of this paper is to develop efficient numerical schemes, in the framework of VEM to solve the coupled PNP/NS model. By design, the proposed  schemes provide the following three desired properties, i.e., (i) accuracy (first order in time); (ii) stability (in the sense that the unconditional energy dissipation law holds); and (iii) simplicity and flexibility to be implemented on general meshes. For this \rev{purpose} we combine a space discretization by  mixed VEM   for the NS equations with the usual primal  VEM formulation for the PNP system, whereas for the discretization in time we use a classical backward Euler implicit method.

As an extension of  FEMs onto polygonal/polyhedral meshes, VEMs were  introduced in \cite{Brezzi13}. In the VEM, the local discrete space on each mesh element consists of polynomials up to a given degree and some additional non-polynomial functions. In order to discretize continuous problems, the VEM only requires the knowledge of the degrees of freedom of the shape functions, such as values at mesh vertices, the moments on mesh edges/faces, or the moments on mesh polygons/polyhedrons, instead of knowing the shape functions explicitly. Moreover, the discrete space can be extended to high order in a straightforward way.  VEMs for general second-order eltic problems were presented in \cite{Cangiani17}. We also mention that VEMs for the building blocks of the coupled system are already available from the literature. In particular, we employ here the VEM for NS equations introduced in \cite{Beir18N}. Other formulations (of mixed, discontinuous, nonconforming, and other types) for NS  include \cite{Beir19,Gatica18,liu19,verma21,wang21}, whereas for the PNP system a VEM scheme has been recently proposed in  \cite{Liu21}. The present method also follows other VEM formulations for Stokes flows from \cite{Beir17,Cangiani16,Gatica17,Wei21}. For a more thorough survey, we refer to \cite{Brezzi14,Marini14} and the references therein.

\subsection{Outline} The remainder of the paper has been organized in the  following manner. In what is left of this \rev{section},  we recall the coupled PNP/NS equations in non-dimensional form, we provide notational preliminaries, and introduce the corresponding variational formulation for the system. In Section \ref{s3}, we present the VE discretization, introducing the mesh entities, the degrees of freedom, the construction of VE spaces, and establishing properties of the discrete multilinear forms.
In Section \ref{s4}, we obtain two conservative properties global mass conservation and electric (and kinetic) energy conservation of the proposed scheme.
 In Section \ref{s5}, under the assumption of small data, the existence and uniqueness of the discrete problem are proved. In Section \ref{s6}, we establish error estimates for the velocity, pressure, concentrations and electrostatic potential. A set of  numerical tests are reported in Section \ref{s7}. They  allow us to assess the accuracy properties of the method  by confirming  the experimental rates of convergence predicted by the theory. Examples of applicative interest in the process of water desalination are also included.  

\subsection{The model problem in non-dimensional form}
Consider a spatial bounded domain  $\Omega\subset\mathbb{R}^{d}$ ($d=2,3$)  with a Lipschitz continuous boundary $\partial\Omega$ with outward-pointing unit normal $\textbf{n}$, and consider the time interval $t \in [0,t_{F}]$, with $t_F>0$ a given final time. We focus on  the  electro-hydrodynamic model described by the coupled  PNP/NS equations  following the non-dimensionalization and problem setup from, e.g.,  \cite{Mani13,Gross19}, and cast in the following strong form  (including  transport of a  dilute 2-component electrolyte, electrostatic equilibrium, momentum balance with body force exerted by the electric field, mass conservation, no-flux and no-boundary conditions,  and appropriate initial conditions)
\rev{
\begin{subequations}\label{Eq_1}
\begin{align}
\partial_{t}c_{i}-\operatorname{div}\left(\kappa_{i}(\nabla c_{i}+e_{i}\nabla\phi) \right)+\operatorname{div}(\textbf{u}c_{i})&= 0 &  \text { in } \Omega \times(0, t_{F}], \\
-\operatorname{div}(\epsilon \nabla \phi)&=c_{1}-c_{2} &  \text { in } \Omega \times(0, t_{F}], \\
\partial_{t}\textbf{u}-\Delta \textbf{u} + (\textbf{u}\cdot\nabla)\textbf{u}+\nabla p&=-(c_{1}-c_{2})\nabla \phi &  \text { in } \Omega \times(0, t_{F}], \\
\operatorname{div}(\textbf{u})&=0 &  \text { in } \Omega \times(0, t_{F}], \\
\nabla c_{i}\cdot\textbf{n}=\nabla\phi\cdot\textbf{n}=0, ~~~\textbf{u}&=\textbf{0}&  \text { on } \partial \Omega \times(0, t_{F}], \\
    c_{i}(\mathbf{x}, 0)=c_{i,0}(\mathbf{x}), \quad \phi(\mathbf{x}, 0)=\phi_{0}(\mathbf{x}), \quad \textbf{u}(\mathbf{x}, 0)&=\textbf{u}_{0}(\mathbf{x}) &  \text { in } \Omega ,\end{align}
\end{subequations}
where $i\in\{1,2\}$, $c_{1},c_{2}$} are the concentrations of positively and negatively charged ions \rev{with valences $e_{1}=1$ and $e_{2}=-1$}, respectively; 
$\phi$ is the electrostatic potential, 
$\textbf{u}$ and $p$ are the velocity and  pressure of the incompressible fluid, respectively; 
$\epsilon$ represents the dielectric coefficient (\rev{assumed a positive constant}) 
and \rev{$\kappa_{1}$ and $\kappa_{2}$} are diffusion/mobility coefficients (\rev{assumed also constant and positive}). 
The boundary conditions considered in \eqref{Eq_1} could be extended to more general scenarios. They are taken as they are for sake of simplicity in the presentation of the analysis.

\subsection{Notation and weak formulation}\label{s2}
Throughout the paper, let $\mathcal{D}$ be any given open subset of $\Omega$. By $(\cdot, \cdot)$ and $\Vert\cdot\Vert_{\mathcal{D}}$ we denote the usual integral inner product and the corresponding norm of $L^{2}(\mathcal{D})$. For a non-negative integer $m$, we shall use the common notation for the Sobolev spaces $W^{m,r}(\mathcal{D})$ with the corresponding norm and semi-norm $\Vert\cdot\Vert_{m,r,\mathcal{D}}$ and $\vert\cdot\vert_{m,r,\mathcal{D}}$, respectively; and if $r=2$, we set $H^{m}(\mathcal{D}):=W^{m,2}(\mathcal{D})$, $\Vert\cdot\Vert_{m,\mathcal{D}}:=\Vert\cdot\Vert_{m,2,\mathcal{D}}$ and $\vert\cdot\vert_{m,\mathcal{D}}:=\vert\cdot\vert_{m,2,\mathcal{D}}$.
If $\mathcal{D}=\Omega$, the subscript will be omitted. 

Let us introduce the following functional spaces for velocity, pressure, and concentrations and electrostatic potential 
\[
\rev{\boldsymbol{X} := \textbf{H}_{0}^{1}(\Omega),\qquad Q:=L_{0}^{2}(\Omega),\qquad\textbf{Z} :=Z\times Z,\qquad Y:=\{v\in Z:~~~(v,1)_{0}=0\},}
\]
respectively, \rev{with $Z:=H^{1}(\Omega)$}. We endow $\textbf{X}$, $Q$, $\textbf{Z}$ and $Y$ with the following norms
\[
\rev{\|\boldsymbol{\tau}\|_{\boldsymbol{X}}^{2}:=\Vert \boldsymbol{\tau}\Vert_{1}^{2}, \quad\quad\|q\|_{Q}^{2}:=\|q\|_{0}^{2},\quad\quad\|(z_{1},z_{2})\|_{\textbf{Z}}^{2}:=\|z_{1}\|_{Z}^{2}+\|z_{2}\|_{Z}^{2},\quad\quad \|v\|_{Y}^{2}:=\|v\|_{Z}^{2},}
\]
respectively, \rev{with $\Vert \cdot\Vert_{Z}^{2}:=\|\cdot\|_{1}^{2}$}. For functions of both spatial \rev{$\textbf{x} \in\Omega$} and temporal variables $t\in J:=[0,t_{F}]$, we will also use the standard function space $L^{2}(J ; V)$ whose norms are defined by:
$$
\|\boldsymbol{v}\|_{L^{2}(V)}:=\left(\int_{0}^{t_{F}}\|v(t)\|_{V}^{2} \mathrm{~d} \rev{t}\right)^{\frac{1}{2}}, \quad\|v\|_{L^{\infty}( V)}:=\rev{\operatorname{ess\,sup}_{t \in J}}\|v(t)\|_{V},
$$
particularly, $V$ can represent $\textbf{X}$,~$Q$ and $\textbf{Z},~Y$. Next, and in order to write the variational formulation of problem \eqref{Eq_1}, we introduce the following bilinear \rev{(and trilinear) forms
\begin{align*}
 \mathcal{M}_{1}(\rho, \zeta)&:=(\rho, \zeta)_{0}, \quad    \mathcal{A}_{i}(\rho, \zeta):=(\kappa_{i}\nabla \rho, \nabla \zeta)_{0} ,    \quad \mathcal{A}_{3}(\rho, \zeta):=(\epsilon \nabla \rho, \nabla \zeta)_{0},\quad    \mathcal{C}(\xi; \rho, \zeta):=(\xi\nabla \rho, \nabla \zeta)_{0},  \\[1mm]
  \quad  
\mathcal{M}_{2}(\textbf{u}, \textbf{v})&:=(\textbf{u}, \textbf{v})_{0}, \quad  \mathcal{K}(\textbf{u}, \textbf{v}):=(\nabla \textbf{u}, \nabla \textbf{v})_{0} , \quad \mathcal{B}(q, \textbf{v}):=(q, \operatorname{div}(\textbf{u}))_{0}.
\end{align*}
for all $\rho,\zeta,\xi\in H^{1}(\Omega)$, $\textbf{u},\textbf{v}\in \textbf{X}$ and $q\in Q$}. As usual for convective problems, for $\textbf{u},\textbf{v},\textbf{w}\in \boldsymbol{X}$ and using  that $\operatorname{div}(\textbf{u})=0$, we utilize the following equivalent skew--symmetric {forms for the terms} $(\operatorname{div}(\textbf{u}\rho), \zeta)_0$ and $((\textbf{u}\cdot\nabla)\textbf{w}, \textbf{v})_0$, respectively 
\[
\mathcal{D}(\textbf{u}; \rho, \zeta):=\frac{1}{2}\big[(\textbf{u} \rho, \nabla \zeta)_{0}-(\textbf{u}\cdot \nabla \rho,  \zeta)_{0}  \big],\qquad 
\mathcal{E}(\textbf{u};\textbf{w}, \textbf{v}):=\frac{1}{2}\big[(\textbf{u}\cdot\nabla \textbf{w}, \textbf{v})_{0}-(\textbf{u}\cdot\nabla \textbf{v}, \textbf{w})_{0} \big].
\]
The weak formulation of \eqref{Eq_1}  consists in finding, for almost all $t\in J$, \rev{$\{(c_{1}(t),c_{2}(t)),\phi(t) \}\in \textbf{Z}\times Y$ and $\{\textbf{u}(t), p(t)\}\in \textbf{X}\times Q$} such that $\partial_{t}c_{i} \in L^{2}(J;\rev{H^{-1}(\Omega)})$,~$\partial_{t}\textbf{u}\in L^{2}(J;\rev{\textbf{H}^{-1}(\Omega)})$ and such that for $i\in\{1,2\}$ the following relations hold
\rev{\begin{subequations}\label{eq6}
\begin{align}
\mathcal{M}_{1}(\partial_{t}c_{i}, z_{i})+\mathcal{A}_{i}(c_{i}, z_{i})+e_{i}\mathcal{C}(c_{i};\phi, z_{i})-\mathcal{D}(\textbf{u}; c_{i}, z_{i})&=0 & \forall z_{i}\in Z,\\
\mathcal{A}_{3}(\phi, z_{3})&=\mathcal{M}_{1}(c_{1}, z_{3})-\mathcal{M}_{1}(c_{2}, z_{3}) & \forall z_{3}\in Y,\\
\mathcal{M}_{2}(\partial_{t}\textbf{u}, \textbf{v})+\mathcal{K}(\textbf{u}, \textbf{v})+\mathcal{E}(\textbf{u};\textbf{u}, \textbf{v})-\mathcal{B}(p, \textbf{v})&=-\left((c_{1}-c_{2})\nabla \phi, \textbf{v} \right) & \forall \textbf{v}\in \textbf{X},\\
\mathcal{B}(q, \textbf{u})&=0  & \forall q\in Q,
\end{align}
\end{subequations}
endowed with initial conditions $c_{i}(\cdot,0)=c_{i,0}$} and $\textbf{u}(\cdot,0)=\textbf{u}^{0}$. The existence and uniqueness of \rev{a} weak solution to \eqref{eq6}   has been proved in \cite{Schmuck09}, for the 2D case.

\section{Virtual element approximation}\label{s3}
The chief target of this section is to present the  VE spaces and required discrete bilinear (and trilinear) forms. \rev{The presentation is restricted to the 2D case, for which the well-posedness of the continuous problem is available.}

\subsection{Mesh notation and mesh regularity}
By $\{\mathcal{T}_{h}\}_{h}$ we will denote a sequence of partitions of $\Omega$ into general polygons $E$ (open and simply connected sets whose boundary $\partial E$ is a non-intersecting poly-line consisting of a finite number of straight line segments) having diameter $h_{E}$. Let $\mathcal{E}_{h}$ be the set of edges $e$ of $\{\mathcal{T}_{h}\}_{h}$, and let $\mathcal{E}_{h}^{I}=\mathcal{E}_{h}\backslash \partial\Omega$ ($\mathcal{E}_{h}^{B}=\mathcal{E}_{h}\cap \partial\Omega$) be the set of all interior edges. By $\textbf{n}_{E}^{e}$, we denote the unit normal (pointing outwards) vector $E$ for any edge $e\in \partial E \cap \mathcal{E}_{h}$. Following, \rev{for example,} \cite{Brezzi13,Lovadina17,Brenner17,Chen8}, we adopt the following regularity assumption 
\begin{assumption}\label{assu:mesh} 
\label{as_1}
There exist constants $\rho_{1},\rho_{2}>0$ such that:
\begin{itemize}
\item[\textbullet] Every element $E$ is shaped like a star with respect to a ball with radius $\geq \rho_{1}h_{E}$,
\item[\textbullet] In $E$, the distance between every two vertices is $\geq \rho_{2}h_{E}$.
\end{itemize}
\end{assumption}

\subsection{Construction of a virtual element space for \rev{$\textbf{Z}$}}
This subsection is devoted to introducing  the VE subspace $Z_{h}^{k}\subset Z$. In order to do that, we recall the definition of \rev{some} useful spaces.  Given $k\in \mathbb{N}$, $E\in \mathcal{T}_{h}$ and $e\in \mathcal{E}_{h}$, we define
\begin{itemize}
\item[\textbullet]
$\mathbb{P}_{k}(E) $ the set of polynomials of degree at most $k$ on $E$ (with extended notation $\mathbb{P}_{-1}(E):=\emptyset$).
\item[\textbullet]
$\mathbb{P}_{k}(e) $ the set of polynomials of degree at most $k$ on $e$ (with the extended notation $\mathbb{P}_{-1}(e):=\emptyset$).
\item[\textbullet] $\mathbb{B}_{k}(\partial E):=\{v\in C^{0}(\partial E):~~~v|_{e}\in \mathbb{P}_{k}(e)~~~for~all~edges~e\subset \partial E\}$.
\item[\textbullet]
$\widetilde{Z}_{k}^{E}:=\left\{z_{h} \in C^{0}(E)\cap H^{1}(E) : \quad z_{h}|_{\partial E}\in \mathbb{B}_{k}(E),\quad \Delta z_{h} \in \mathbb{P}_{k}(E) \right\}.$
\end{itemize}
For $\mathcal{O}\subset\mathbb{R}^{2}$,  we denote by $|O|$ its area, $h_{O}$ its diameter, and $\mathbf{x}_{O}$ its barycenter. Given any integer $r\geq 1$, we denote by $\mathcal{M}_{r}(O)$ the set of scaled monomials
\begin{equation*}
\mathcal{M}_{r}(\mathcal{O}):=\left\{m: m=\left(\frac{\mathbf{x}-\mathbf{x}_{\mathcal{O}}}{h_{\mathcal{O}}}\right)^{\mathbf{s}} \text { for } \quad \mathbf{s} \in \mathbb{N}^{2} \quad \text { with } \quad|\mathbf{s}| \leq r\right\},
\end{equation*}
where $\mathbf{s}=(s_{1}, s_{2})$, $|s|=s_{1}+s_{2}$ and $\mathbf{x}^{s}=x_{1}^{s_{1}}x_{2}^{s_{2}}$. Besides, we need another set which is as follows
\begin{equation*}
\mathcal{M}_{r}^{*}(\mathcal{O}):=\left\{m: m=\left(\frac{\mathbf{x}-\mathbf{x}_{\mathcal{O}}}{h_{\mathcal{O}}}\right)^{\mathbf{s}} \text { for } \quad \mathbf{s} \in \mathbb{N}^{2} \quad \text { with } \quad|\mathbf{s}|=r\right\}.
\end{equation*}
Further, we recall the helpful polynomial projections $\Pi_{k}^{0,E}$ and $\Pi_{k}^{\nabla,E}$  associated with   $E\in\mathcal{T}_{h}$ as follows:
\begin{itemize}
\item[\textbullet]
the $L^{2}$-projection $\Pi_{k}^{0,E}: \widetilde{Z}_{k}^{E}\rightarrow \mathbb{P}_{k}(E)$, given by
\begin{equation*}
\int_{E} q_{k}(z-\Pi_{k}^{0, E} z) ~\mathrm{d}\textbf{x}=0 \quad \forall z \in L^{2}(E) \quad \text { and } \quad \forall q_{k} \in \mathbb{P}_{k}(E),
\end{equation*}
\item[\textbullet]
 the $H^{1}$-projection $\Pi_{k}^{\nabla,E}: \widetilde{Z}_{k}^{E}\rightarrow \mathbb{P}_{k}(E)$, defined by
 \begin{equation*}
 \left\{
\begin{array}{ll}
\int_{E} \nabla q_{k}\cdot \nabla (z-\Pi_{k}^{\nabla, E} z) ~\mathrm{d} \textbf{x}=0, & \forall z\in H^{1}(E) \quad \text { and } \forall q_{k} \in \mathbb{P}_{k}(E) ,\\[1mm]
 \int_{\partial E}(z-\Pi_{k}^{\nabla, E} z)~ \mathrm{d} s=0, & \text { if } \quad k=1, \\[1mm]
  \int_{E} (z-\Pi_{k}^{\nabla, E} z)~ \mathrm{d} \textbf{x}=0, & \text { if } \quad k \geq 2.
\end{array}\right.
\end{equation*}
\end{itemize}
\rev{Finally, let $k$ be a fixed positive integer and consider the following local VE space on each  $E\in\mathcal{T}_{h}$ (cf. \cite{Ahmad})} 
\begin{equation*}
Z_{k}^{E}:=\left\{z_{h} \in \widetilde{Z}_{k}^{E}:~~ (q_{h}^{*} , z_{h})_{E}=(q_{h}^{*} , \Pi_{k}^{\nabla, E} z_{h})_{E} \quad \forall q_{h}^{*} \in \mathcal{M}_{k-1}^{*}(E) \cup \mathcal{M}_{k}^{*}(E)\right\}.
\end{equation*}
\rev{And its degrees of freedom (guaranteeing unisolvency) are as follows
(see, e.g.,  \cite{Cangiani17}):}
\begin{itemize}
\item[\textbullet] $(\textbf{D1})$ The value of $z$ at the $i$-th vertex of the element $E$.

\item[\textbullet] $(\textbf{D2})$  The values of $z$ at $k-1$ distinct points in $e$, for all $e\subset\partial E$, and for $k\geq 2$.

\item[\textbullet] \rev{$(\textbf{D3})$} The internal moment $\int_{E}z~q$, for all $q\in \mathcal{M}_{k-2}(E)$, and $k\geq 2$.
\end{itemize}
It is noteworthy that $\Pi_{k}^{0, E}$ and $\Pi_{k}^{\nabla, E}$ are computable from \rev{knowing $(\textbf{D1})-(\textbf{D3})$ (see, e.g., \cite{Cangiani17}). Similarly to} the finite element case, the global VE space can be assembled as:
\[
Z_{h}:=\left\{z_{h} \in \rev{Z}:\quad z_{h}|_{E} \in Z_{k}^{E}\quad \forall E\in \mathcal{T}_{h}\right\}.
\]
Finally, we define a VE space on  $\mathcal{T}_{h}$ \rev{for the concentrations and electrostatic potential} as follows:
\[
\rev{\textbf{Z}_{h}:=Z_{h}\times Z_{h},\quad \text{and}\quad Y_{h}:=\left\{z_{h}\in Z_{h}: \quad(z_{h},1)_{0,\mathcal{T}_{h}}=0 \right\}}.
\]

\noindent\textbf{Approximation properties in the \rev{local space $Z_{k}^{E}$}.} The following estimates (established using \rev{Assumption \ref{assu:mesh}}) can be obtained for the projection and interpolation operators \cite{Brezzi13}.
\begin{itemize}
\item[\textbullet]  there exists a $z_{\pi} \in \mathbb{P}_{k}(E)$ such that for $s\in [1,k+1]$ and  $z \in H^{s}(E)$, there holds  
\begin{equation}\label{eqPi}
\left|z-z_{\pi}\right|_{0, E}+h_{E}\left|z-z_{\pi}\right|_{1, E} \leq C h_{E}^{s}|z|_{s, E}.
\end{equation}
\item[\textbullet]  there exists a $z_{I} \in Z_{k}^{E}$ such that for $s\in [2,k+1]$ and  $z \in H^{s}(E)$, there holds  
\begin{equation}\label{eqIn}
\left|z-z_{I}\right|_{0, E}+h_{E}\left|z-z_{I}\right|_{1, E} \leq C h_{E}^{s}|z|_{s, E}.
\end{equation}
\end{itemize}
\subsection{Construction of a virtual element space approximating $\textbf{X}$}\label{subSA}
Following \cite{Beir17}, for $k\geq 2$ let us introduce the spaces 
\begin{align*} 
\mathcal{G}_{k}(E) &:=\nabla\mathbb{P}_{k+1}(E)\subset [\mathbb{P}_{k}(E)]^{2},\qquad 
\mathcal{G}_{k}(E)^{\perp}  :=\textbf{x}^{\perp}[\mathbb{P}_{k-1}(E)]\subset [\mathbb{P}_{k}(E)]^{2} \text{ with }\textbf{x}^{\perp}:=(x_{2}, -x_{1}),\\
 \mathbf{\widetilde{X}}_{k}^{E} &:=\bigg\{\begin{array}{ll}\mathbf{v} \in \textbf{X} & \text { s.t }  \mathbf{v}\mid_{ \partial E} \in\left[\mathbb{B}_{k}(\partial E)\right]^{2},\left\{\begin{array}{l}- \boldsymbol{\Delta} \mathbf{v}-\nabla w \in \rev{\mathcal{G}_{k}}(E)^{\perp}, ~\forall w \in L^{2}(E)\setminus\mathbb{R} \\ \operatorname{div} \mathbf{v} \in \mathbb{P}_{k-1}(E),\end{array}\right.\end{array}\bigg\}.\end{align*}
The definition of scaled monomials can be extended to the vectorial case.  Let $\boldsymbol{\alpha}:=(\alpha_{1}, \alpha_{2})$ and $\boldsymbol{\beta}:=(\beta_{1}, \beta_{2})$ be two
multi-indexes, then we define a vectorial scaled monomial as
\begin{equation*}
\rev{\boldsymbol{m}_{\boldsymbol{\alpha}, \boldsymbol{\beta}}:=\left(\begin{array}{l}
m_{\boldsymbol{\alpha}} \\
m_{\boldsymbol{\beta}}
\end{array}\right).}
\end{equation*}
Also in this case, it is easy to show that the set
\begin{equation*}
[\mathcal{M}_{r}(\mathcal{O})]^{2}:=\left\{\boldsymbol{m}_{\boldsymbol{\alpha}, \emptyset}: 0 \leq|\boldsymbol{\alpha}| \leq r\right\} \cup\left\{\boldsymbol{m}_{\emptyset, \boldsymbol{\beta}}: 0 \leq|\boldsymbol{\beta}| \leq r\right\}:=\left\{\boldsymbol{m}_{i}: 1 \leq i \leq 2 \pi_{r}\right\},
\end{equation*}
is a basis for the vectorial polynomial space $[\mathbb{P}_{r}(E)]^{2}$, where we implicitly use the natural correspondence between one-dimensional indices and double multi-indices. 

One core idea in the VEM construction is to define suitable (computable) polynomial projections. 
\rev{Polynomial projections can be extended to the vector case (see, e.g., \cite{Gatica17}): the $\textbf{L}^{2}$-projection $\boldsymbol{\Pi}_{k}^{0,E}$ and the $\textbf{H}^{1}$-projection  $\boldsymbol{\Pi}_{k}^{\nabla,E}$ similarly as in the scalar case. }
And a VE subspace of  $\mathbf{\widetilde{X}}_{k}^{E}$ is given by 
\begin{equation*}
\mathbf{X}_{k}^{E}:=\left\{\mathbf{v}_{h} \in \mathbf{\widetilde{X}}_{k}^{E}: \quad  \quad\left(\boldsymbol{\Pi}_{k}^{\nabla, E} \mathbf{v}_{h}-\mathbf{v}_{h}, \mathbf{g}_{k}^{\perp}\right)=0,\quad \quad \forall \mathbf{g}_{k}^{\perp} \in \mathcal{G}_{k}^{\perp} (E) / \mathcal{G}_{k-2}^{\perp} (E)\right\}.
\end{equation*}
We recall the following properties of the space $\mathbf{X}_{k}^{E}$. Also, \rev{the corresponding unisolvent degrees of freedom} in $\mathbf{X}_{k}^{E}$ can be divided into the following four types (see \cite{Beir17,Beir18N}) 
\begin{itemize}
\item[\textbullet] $(\textbf{D1}_{\textbf{v}})$: the values of $\textbf{v}$ at the vertexes of the element $E$,
\item[\textbullet] $(\textbf{D2}_{\textbf{v}})$: the values of $\textbf{v}$ at $k-1$ distinct points of any edge $e\subset\partial E$,
\item[\textbullet] $(\textbf{D3}_{\textbf{v}})$: the moments 
\begin{equation*}
\rev{\int_{E} \mathbf{v} \cdot \textbf{g}^{\perp} ~\mathrm{d} E, \quad\quad \forall \textbf{g}^{\perp} \in \mathcal{G}_{k-2}^{\perp} (E),}
\end{equation*}
\item[\textbullet] $(\textbf{D4}_{\textbf{v}})$: the moments 
\begin{equation*}
\int_{E}(\operatorname{div} \mathbf{v}) m_{\alpha} \mathrm{~d} E ,\quad\quad \forall m_{\alpha} \in \mathcal{M}_{k-1}(E) / \mathbb{R}.
\end{equation*}
\end{itemize}
We observe that the projectors $\boldsymbol{\Pi}_{k}^{\nabla,E}$ and $\boldsymbol{\Pi}_{k}^{0,E}$ can be computed using only the degrees of freedom $(\textbf{D1}_{\textbf{v}})$--$(\textbf{D4}_{\textbf{v}})$.

Finally, the global finite dimensional space $\textbf{X}_{h}$, associated with the partition $\mathcal{T}_{h}$, is defined such that the restriction of every VE function $\textbf{v}$ to the mesh element $E$ belongs to $\textbf{X}_{k}^{E}$. On the other hand, the discrete pressure spaces are simply given by piecewise polynomials of degree \rev{up to $k-1$}:
\begin{equation*}
\rev{Q_{h}:=\bigg\{q_{h} \in Q:\left.\quad q_{h}\right|_{E} \in \mathbb{P}_{k-1}(K), \quad \forall E \in \mathcal{T}_{h}\bigg\},}
\end{equation*}
and we also remark that
\begin{equation}\label{propA}
\rev{\operatorname{div}\textbf{X}_{h} \subseteq Q_{h}}.
\end{equation}

\noindent\textbf{Approximation properties associated with the space $\textbf{X}_{k}^{E}$.} The following estimates \rev{can be obtained using Assumption \ref{assu:mesh} (see, e.g., \cite{Beir17})}:
\begin{itemize}
\item[\textbullet]  there exists a $\textbf{z}_{\pi} \in [\mathbb{P}_{k}(E)]^{2}$ such that for $s\in [1,k+1]$,  $\textbf{z} \in [H^{s}(E)]^{2}$ and \rev{$r\in [1,\infty]$} we have 
\begin{equation}\label{eqPi_A}
\rev{\left|\textbf{z}-\textbf{z}_{\pi}\right|_{r,E}+h_{E}\left|\textbf{z}-\textbf{z}_{\pi}\right|_{1,r,E} \leq C h_{E}^{s}|\textbf{z}|_{s,r,E}.}
\end{equation}
\item[\textbullet]  there exists a $\textbf{z}_{I} \in \textbf{X}_{k}^{E}$ such that for $s\in [2,k+1]$, $\textbf{z} \in [H^{s}(E)]^{2}$ and \rev{$r\in [1,\infty]$}, we have 
\begin{equation}\label{eqIn_A}
\rev{\left|\textbf{z}-\textbf{z}_{I}\right|_{r,E}+h_{E}\left|\textbf{z}-\textbf{z}_{I}\right|_{1,r,E} \leq C h_{E}^{s}|\textbf{z}|_{s,r,E}.}
\end{equation}
\end{itemize}
\subsection{The discrete forms and their properties}\label{ss4}
As usual in the  VE literature \cite{Brezzi13,Brezzi14} we define computable discrete  forms that approximate the continuous bilinear and trilinear forms in  \eqref{eq6} using projections. \rev{Similarly} to the finite element case, we only need to construct the computable local discrete forms, which can be summed up element by element to obtain the corresponding global discrete forms.

\medskip 
Firstly, we define \rev{$\mathcal{M}_{1,h}^{E}: Z_{k}^{E}\times Z_{k}^{E}\rightarrow \mathbb{R}$ and $\mathcal{A}_{j,h}^{E}: Z_{k}^{E}\times Z_{k}^{E}\rightarrow \mathbb{R}$ for $j=1,2,3$} as
\begin{equation}\label{dis_F1}
\rev{\mathcal{M}_{1,h}^{E}(\rho_{h}, \zeta_{h}):=\mathcal{M}_{1}^{E}(\Pi_{k}^{0,E}(\rho_{h}), \Pi_{k}^{0,E}(\zeta_{h}))+|E|S_{\mathtt{m}}^{E}((I-\Pi_{k}^{0,E})\rho_{h}, (I-\Pi_{k}^{0,E})\zeta_{h}),}
\end{equation}
and
\begin{equation*}
\rev{\mathcal{A}_{j,h}^{E}(\rho_{h}, \zeta_{h}):=\mathcal{A}_{j}^{E}(\Pi_{k}^{\nabla,E}\rho_{h}, \Pi_{k}^{\nabla,E}\zeta_{h})+|\lambda_{j}|S_{\mathtt{a}}^{E}((I-\Pi_{k}^{\nabla,E})\rho_{h}, (I-\Pi_{k}^{\nabla,E})\zeta_{h}),\quad \lambda_{j}\in\{\kappa_{1},\kappa_{2},\epsilon\},}
\end{equation*}
respectively, where the stabilizations  $S_{m}^{E}: Z_{k}^{E}\times Z_{k}^{E}\rightarrow \mathbb{R}$ and $S_{a}^{E}: Z_{k}^{E}\times Z_{k}^{E}\rightarrow \mathbb{R}$ are a symmetric, positive definite, bilinear \rev{forms such that
\begin{subequations}
\begin{align}\label{eq28_S}
c_{0,\mathtt{m}}\Vert \rho_{h}\Vert _{0,E}^{2}\leq S_{\mathtt{m}}^{E}(\rho_{h},\rho_{h}) \leq c_{1,\mathtt{m}}\Vert\rho_{h}\Vert _{0,E}^{2},\quad \forall \rho_{h}\in Z_{k}^{E},\quad \text{with}~\Pi_{k}^{0,E}(\rho_{h})&=0,\\
\label{eq29_S}
c_{0,\mathtt{a}}\rvert \rho_{h}\rvert _{1,E}^{2}\leq S_{\mathtt{a}}^{E}(\rho_{h},\rho_{h}) \leq c_{1,\mathtt{a}}\rvert \rho_{h}\rvert _{1,E}^{2},\quad \forall \rho_{h}\in Z_{k}^{E},\quad \text{with}~\Pi_{k}^{\nabla,E}(\rho_{h})&=0,
\end{align}
\end{subequations}
for positive constants $c_{0,\mathtt{m}}$, $c_{1,\mathtt{m}}$, $c_{0,\mathtt{a}}$, $c_{1,\mathtt{a}}$} that are independent of $h$.

Moreover, the term $\mathcal{C}(\xi;\rho, \zeta)$ is replaced by
\begin{equation}\label{defiDisCPNP}
\rev{\mathcal{C}_{h}^{E}(\xi_{h};\rho_{h}, \zeta_{h}):=(\Pi_{k-1}^{0,E}\xi_{h}\boldsymbol{\Pi}_{k-1}^{0,E}\nabla \rho_{h}, \boldsymbol{\Pi}_{k-1}^{0,E}\nabla \zeta_{h} )_{0,E}.}
\end{equation}
Also, the discrete local forms \rev{$\mathcal{M}_{2,h}: \textbf{X}_{k}^{E}\times \textbf{X}_{k}^{E}\rightarrow \mathbb{R}$ and $\mathcal{K}_{h}: \textbf{X}_{k}^{E}\times \textbf{X}_{k}^{E}\rightarrow \mathbb{R}$ are} defined as
\begin{align*}
\mathcal{M}_{2,h}^{E}(\textbf{u}_{h}, \textbf{v}_{h})&:=\mathcal{M}_{2}^{E}(\boldsymbol{\Pi}_{k}^{0,E}(\textbf{u}_{h}), \boldsymbol{\Pi}_{k}^{0,E}(\textbf{v}_{h}))+|E|\tilde{\mathcal{S}}_{\mathtt{m}}^{E}((I-\boldsymbol{\Pi}_{k}^{0,E})\textbf{u}_{h}, (I-\boldsymbol{\Pi}_{k}^{0,E})\textbf{v}_{h}),\\
\mathcal{K}_{h}^{E}(\textbf{u}_{h}, \textbf{v}_{h})&:=\mathcal{K}^{E}(\boldsymbol{\Pi}_{k}^{\nabla,E}(\textbf{u}_{h}), \boldsymbol{\Pi}_{k}^{\nabla,E}(\textbf{v}_{h}))+\tilde{\mathcal{S}}_{\mathtt{a}}^{E}((I-\boldsymbol{\Pi}_{k}^{0,E})\textbf{u}_{h}, (I-\boldsymbol{\Pi}_{k}^{0,E})\textbf{v}_{h}),
\end{align*}
respectively, where the stabilizers $\tilde{\mathcal{S}}_{\mathtt{m}}^{E}: \textbf{X}_{k}^{E}\times \textbf{X}_{k}^{E}\rightarrow \mathbb{R}$ and $\tilde{\mathcal{S}}_{\mathtt{a}}^{E}: \textbf{X}_{k}^{E}\times \textbf{X}_{k}^{E}\rightarrow \mathbb{R}$ are symmetric, positive definite bilinear forms \rev{satisfying
\begin{subequations}
  \begin{align}
    \label{eq30_S}
\tilde{c}_{0,\mathtt{m}}\Vert \textbf{z}_{h}\Vert _{0,E}^{2}\leq \tilde{\mathcal{S}}_{\mathtt{m}}^{E}(\textbf{z}_{h},\textbf{z}_{h}) \leq\tilde{c}_{1,\mathtt{m}}\Vert \textbf{z}_{h}\Vert _{0,E}^{2},\quad\forall \textbf{z}_{h}\in \textbf{X}_{k}^{E},\quad\text{with}~\boldsymbol{\Pi}_{k}^{0,E}(\textbf{z}_{h})&=\textbf{0},\\
\label{eq31_S}
\tilde{c}_{0,\mathtt{a}}\rvert \textbf{z}_{h}\rvert _{1,E}^{2}\leq \tilde{\mathcal{S}}_{\mathtt{a}}^{E}(\textbf{z}_{h},\textbf{z}_{h}) \leq \tilde{c}_{1,\mathtt{a}}\rvert \textbf{z}_{h}\rvert _{1,E}^{2},\quad\forall \textbf{z}_{h}\in \textbf{X}_{k}^{E},\quad\text{with}~\boldsymbol{\Pi}_{k}^{\nabla,E}(\textbf{z}_{h})&=\textbf{0},
\end{align}
\end{subequations}
for positive constants $\tilde{c}_{0,\mathtt{m}}, \tilde{c}_{1,\mathtt{m}}, \tilde{c}_{0,\mathtt{a}}, \tilde{c}_{1,\mathtt{a}}$} that are independent of $h$.
Finally, the skew-symmetric trilinear forms $\mathcal{D}(\textbf{w}; \rho ,\zeta)$ and \rev{$\mathcal{E}(\textbf{w};\textbf{u},\textbf{v})$} are replaced, respectively, by
\rev{
\begin{align}\nonumber 
\mathcal{D}_{h}^{E}(\textbf{u}_{h}; \rho_{h}, \zeta_{h})&:=\dfrac{1}{2}\big[(\boldsymbol{\Pi}_{k}^{0,E}\textbf{u}_h\cdot \Pi_{k}^{0,E} \rho_{h},  \Pi_{k-1}^{0,E}\nabla \zeta_{h})_{0}-(\boldsymbol{\Pi}_{k}^{0,E}\textbf{u}_{h}\cdot  \Pi_{k-1}^{0,E}\nabla \rho_{h},  \Pi_{k}^{0,E}\zeta_{h})_{0}  \big],\\
\label{dis_F6}
\mathcal{E}_{h}(\textbf{u}_{h};\textbf{w}_{h}, \textbf{v}_{h})&:=\dfrac{1}{2}\big[(\boldsymbol{\Pi}_{k}^{0,E}\textbf{u}_{h}\cdot\boldsymbol{\Pi}_{k-1}^{0,E}\nabla \textbf{w}_{h}, \boldsymbol{\Pi}_{k}^{0,E}\textbf{v}_{h})_{0}-(\boldsymbol{\Pi}_{k}^{0,E}\textbf{u}_{h}\cdot\boldsymbol{\Pi}_{k-1}^{0,E}\nabla \textbf{v}_{h}, \boldsymbol{\Pi}_{k}^{0,E}\textbf{w}_{h})_{0} \big].
\end{align}
}
\medskip 
\rev{The aforementioned bilinear forms are continuous thanks to Cauchy--Schwarz inequality, the continuity of the projections with respect to the $L^{2}$-norm, and the stabilities  \eqref{eq28_S}--\eqref{eq31_S}: 
  \begin{subequations}
    \begin{align}      \label{BUND1}
\mathcal{M}_{1,h}(\rho_{h},\zeta_{h})& \leq \alpha_{1}\Vert \rho_{h}\Vert_{Z}\Vert \zeta_{h}\Vert_{Z},\quad\quad \mathcal{A}_{j,h}(\rho_{h},\zeta_{h})\leq \alpha_{j+1}\Vert \rho_{h}\Vert_{Z}\Vert \zeta_{h}\Vert_{Z},\\
\label{BUND2}
\mathcal{M}_{2,h}(\textbf{u}_{h},\textbf{v}_{h})&\leq \tilde{\alpha}_{1}\Vert \textbf{u}_{h}\Vert_{\textbf{X}}\Vert \textbf{v}_{h}\Vert_{\textbf{X}},\quad\quad\mathcal{K}_{h}(\textbf{u}_{h},\textbf{v}_{h})\leq \tilde{\alpha}_{2}\Vert \textbf{u}_{h}\Vert_{\textbf{X}}\Vert \textbf{v}_{h}\Vert_{\textbf{X}},
\end{align}
\end{subequations}
for all $\rho_{h},\zeta_{h}\in Z_{h}$, $\textbf{u}_{h},\textbf{v}_{h}\in \textbf{X}_{h}$ and $j=1,2,3$.}

\rev{The bilinear forms $\mathcal{M}_{1}, \mathcal{M}_{2}$ and $\mathcal{A}_{j}$, $j=1,2,3$, turn out to be coercive owing to the stability properties of stabilizers (cf. \eqref{eq28_S}-\eqref{eq31_S}) together with Young and triangle inequalities
  \begin{subequations}
    \begin{gather}\label{CORC1}
\mathcal{M}_{1,h}(\zeta_{h},\zeta_{h})\geq \beta_{1}\Vert \zeta_{h}\Vert_{0}^{2},\quad\quad \mathcal{A}_{j,h}(\zeta_{h},\zeta_{h})\geq \beta_{j+1}\Vert \zeta_{h}\Vert_{Z}^{2},\\
\label{CORC2}
\mathcal{M}_{2,h}(\textbf{v}_{h},\textbf{v}_{h})\geq \tilde{\beta}_{1}\Vert \textbf{v}_{h}\Vert_{0}^{2},
\end{gather}\end{subequations}
for all $\zeta_{h}\in Z_{h}, \textbf{v}_{h}\in\textbf{X}_{h}$.}

\rev{On the other hand, $\mathcal{K}$ is coercive on the discrete kernel $\textbf{X}_{h}$ of the bilinear form $\mathcal{B}$
\begin{equation}
\mathcal{K}_{h}(\textbf{v}_{h},\textbf{v}_{h})\geq \tilde{\beta}_{2}\Vert \textbf{v}_{h}\Vert_{\textbf{X}}^{2},\quad\quad\forall \textbf{v}_{h}\in \widetilde{\textbf{X}}_{h},
\end{equation}
where
\[
\widetilde{\textbf{X}}_{h}=\left\{\textbf{v}_{h}\in \textbf{X}_{h}:~~~\mathcal{B}(q_{h},\textbf{v}_{h})=0,~~~\forall q_{h}\in Q_{h}  \right\}.
\]}

\rev{The continuity of $\mathcal{D}_{h}(\cdot;\cdot,\cdot)$ and $\mathcal{E}_{h}(\cdot;\cdot,\cdot)$ on $Z_{h}$ and $\textbf{X}_{h}$, respectively, is stated in the following result.}
\begin{lemma}
\label{l_N}
\rev{The trilinear forms $\mathcal{D}_{h}(\cdot;\cdot,\cdot)$ and $\mathcal{E}_{h}(\cdot;\cdot,\cdot)$ are continuous, with respective continuity constants}
\begin{equation*}
\rev{\gamma_{1}:=\sup _{\mathbf{u}_{h}\in \mathbf{X}_h, \rho_{h}, \zeta_{h} \in Z_{h}} \frac{\left|\mathcal{D}_{h}(\mathbf{u}_{h} ; \rho_{h}, \zeta_{h})\right|}{\|\mathbf{u}_{h}\|_{\mathbf{X}}\|\rho_{h}\|_{Z}\|\zeta_{h}\|_{Z}},\quad \quad \gamma_{2}:=\sup _{\mathbf{u}_{h}, \mathbf{w}_{h}, \mathbf{v}_{h} \in \mathbf{X}_{h}} \frac{\left|\mathcal{E}_{h}(\mathbf{u}_{h} ; \mathbf{w}_{h}, \mathbf{v}_{h})\right|}{\|\mathbf{u}_{h}\|_{\mathbf{X}}\|\mathbf{w}_{h}\|_{\mathbf{X}}\|\mathbf{v}_{h}\|_{\mathbf{X}}}.}
\end{equation*}
\end{lemma}
\begin{proof}
Using \rev{the} definition of the discrete form $\mathcal{D}_{h}$ and the H\"older inequality, we have
\begin{align}\label{eq_V1}
\mathcal{D}_{h}(\textbf{u}_{h}; \rho_{h}, \zeta_{h})&=\dfrac{1}{2}\big[(\boldsymbol{\Pi}_{k}^{0}\textbf{u}_{h}~\Pi_{k}^{0} \rho_{h},  \Pi_{k-1}^{0}\nabla \zeta_{h})_{0}-(\boldsymbol{\Pi}_{k}^{0}\textbf{u}_{h}\cdot  \Pi_{k-1}^{0}\nabla \rho_{h},  \Pi_{k}^{0} \zeta_{h})_{0}  \big]\nonumber\\[1mm]
&\leq \Vert \boldsymbol{\Pi}_{k}^{0}\textbf{u}_{h}\Vert_{0,4}\Vert \Pi_{k}^{0} \rho_{h}\Vert_{0,4}\Vert  \Pi_{k-1}^{0}\nabla \zeta_{h}\Vert_{0}+\Vert \boldsymbol{\Pi}_{k}^{0}\textbf{u}_{h}\Vert_{0,4}\Vert \Pi_{k-1}^{0} \nabla\rho_{h}\Vert_{0}\Vert  \Pi_{k}^{0} \zeta_{h}\Vert_{0,4}.
\end{align}
Applying the inverse inequality in conjunction with the continuity of the projectors $\boldsymbol{\Pi}_{k}^{0}$ and $\Pi_{k}^{0}$ (with respect to the $L^{2}$-norm), gives the following upper bound for the terms $\Vert \boldsymbol{\Pi}_{k}^{0}\textbf{u}_{h}\Vert_{0,4}$, $\Vert \Pi_{k}^{0} \rho_{h}\Vert_{0,4}$ and  $\Vert  \Pi_{k}^{0} \zeta_{h}\Vert_{0,4}$ on the right-hand side of the above inequality, and for $E\in\mathcal{T}_{h}$:
\begin{equation}\label{Linfity_L4}
\Vert \boldsymbol{\Pi}_{k}^{0,E}\textbf{u}_{h}\Vert_{0,4}\leq h_{E}^{-1/2}\Vert \boldsymbol{\Pi}_{k}^{0,E}\textbf{u}_{h}\Vert_{0,E}\leq h_{E}^{-1/2}\Vert\textbf{u}_{h}\Vert_{0,E}\leq C_{1}\Vert\textbf{u}_{h}\Vert_{0,4},
\end{equation}
and similarly
\begin{equation}
\Vert \Pi_{k}^{0} \rho_{h}\Vert_{0,4}\leq C_{2}\Vert \rho_{h}\Vert_{0,4},\quad \Vert \Pi_{k}^{0} \zeta_{h}\Vert_{0,4}\leq C_{3}\Vert \zeta_{h}\Vert_{0,4}.
\end{equation}
Combining the above estimates with Eq. \eqref{eq_V1}, leads to 
\begin{equation*}
\mathcal{D}_{h}(\textbf{u}_{h}; \rho_{h}, \zeta_{h})\leq \dfrac{1}{2}\left( C_{1}C_{2}+C_{1}C_{3}\right)\|\mathbf{u}_{h}\|_{\mathbf{X}}\|\rho_{h}\|_{Z}\|\zeta_{h}\|_{Z},
 \end{equation*} 
which \rev{confirms the continuity of} $\mathcal{D}_{h}(\cdot;\cdot,\cdot)$. The proof of continuity of \rev{$\mathcal{E}_{h}(\cdot;\cdot,\cdot)$ can be found} in \cite{Beir18N}.
\end{proof}
\begin{lemma}\label{l_Con_C}
 \rev{There exist constants $\gamma_{3}$ and $\gamma_{4}$ (independent of $E$ and $h$) verifying
  \begin{subequations}
\begin{align}
\mathcal{C}_{h}(\xi_{h};\rho_{h},\zeta_{h})&\leq \gamma_{3}\Vert\xi_{h}\Vert_{\infty}\Vert\rho_{h}\Vert_{Z}\Vert\zeta_{h}\Vert_{Z},\quad\quad\quad\quad \forall \xi_{h},\rho_{h},\zeta_{h}\in Z_{h},\label{Boun_C}\\
\left(\xi_{h}\nabla \rho_{h}, \textbf{v}_{h}\right)_{h}&\leq \gamma_{4}\Vert\xi_{h}\Vert_{Z}\Vert \rho_{h}\Vert_{Z}\Vert\textbf{v}_{h}\Vert_{\textbf{X}},\quad\quad\quad\quad \forall \xi_{h},\rho_{h}\in Z_{h},~\textbf{v}_{h}\in \textbf{X}_{h}.\label{Boun_C1}
\end{align}\end{subequations}}
\end{lemma}
\begin{proof}
\rev{The proof is a direct consequence of the H\"older inequality, the
continuity of $\Pi_{k}^{0}$ with respect to the $L^{\infty}$-norm, and the stability property of $S_{\mathtt{a}}$ (cf. \eqref{eq29_S}).}
\end{proof}
\begin{lemma}[Discrete inf-sup condition]\cite{Beir18N}\label{l_INFD} 
Given the VE spaces $\textbf{X}_{h}$ and $Q_{h}$ defined in Section \ref{subSA}, there exists
a positive \rev{constant} $\widehat{\beta}$, independent of $h$, such that:
\begin{equation*}
\rev{\sup _{\mathbf{v}_{h} \in \mathbf{X}_{h} \mathbf{v}_{h} \neq \mathbf{0}} \frac{\mathcal{B}\left(\mathbf{v}_{h}, q_{h}\right)}{\left\|\mathbf{v}_{h}\right\|_{\mathbf{X}}} \geq \widehat{\beta}\left\|q_{h}\right\|_{Q} \quad \text { for all } q_{h} \in Q_{h}.}
\end{equation*}
\end{lemma}
\rev{Such a discrete inf-sup property together with   \eqref{propA},} indicate that 
\[
\operatorname{div}\textbf{X}_{h} = Q_{h}.
\]
\rev{The following result compares $\mathcal{M}_{1}^{E}, \mathcal{M}_{2}^{E}$, $\mathcal{A}_{j}^{E}$ and $\mathcal{K}^{E}$ against their  computable counterparts.}
\begin{lemma}[\cite{Gharibi21}]\label{l_12}
Let $\alpha_{l}, \tilde{\alpha}_{r}$, $l=1,\cdots,4,~r=1,2$ be the constants from \eqref{BUND1} and \eqref{BUND2}. Then for each $\rho,\zeta\in Z$ and $\textbf{u},\textbf{v}\in \textbf{X}$, \rev{there hold
\begin{align*}
|\mathcal{M}_{1}(\rho, \zeta)-\mathcal{M}_{1,h}(\rho, \zeta)| & \leq \alpha_{1} \|\rho-\Pi_{k}^{0}(\rho)\|_{0}\|\zeta\|_{0},\\[1mm]
|\mathcal{M}_{2}(\textbf{u}, \textbf{v})-\mathcal{M}_{2,h}(\textbf{u}, \textbf{v})| & \leq \tilde{\alpha}_{1} \|\textbf{u}-\boldsymbol{\Pi}_{k}^{0}(\textbf{u})\|_{0}\|\textbf{v}\|_{0},\\[1mm]
|\mathcal{A}_{i}(\rho, \zeta)-\mathcal{A}_{i,h}(\rho, \zeta)|&\leq \alpha_{l+1} \vert\rho-\Pi_{k}^{\nabla}(\rho)\vert_{1}\vert\zeta\vert_{1},\\[1mm]
|\mathcal{K}(\textbf{u}, \textbf{v})-\mathcal{K}_{h}(\textbf{u}, \textbf{v})|& \leq \tilde{\alpha}_{2}\vert\textbf{u}-\boldsymbol{\Pi}_{k}^{\nabla}(\textbf{u})\vert_{1}\vert\textbf{v}\vert_{1}.
\end{align*}}
\end{lemma}
\begin{lemma}[\cite{Beir18N}]\label{l_C}
Assume that $\textbf{w}\in \textbf{X}\cap [H^{s+1}(\Omega)]^{2}$ and $s\in [0,k]$. Then, it holds
\begin{equation*}
\rev{|\mathcal{E}^{E}(\textbf{w};\textbf{w},\textbf{v})-\mathcal{E}_{h}^{E}(\textbf{w};\textbf{w},\textbf{v})|\leq Ch^{s}\left(\Vert \textbf{w}\Vert_{\textbf{X}}+\Vert \textbf{w}\Vert_{s}+\Vert \textbf{w}\Vert_{s+1} \right)\Vert \textbf{w}\Vert_{s+1}\Vert \textbf{v}\Vert_{\textbf{X}},\qquad \forall \textbf{v}\in \textbf{X}.}
\end{equation*}
\end{lemma}
\begin{lemma}\label{l_ComD}
\rev{Assume that $\textbf{w}\in \textbf{X}\cap [H^{s+1}(\Omega)]^{2}$, $\rho\in Z\cap H^{s+1}(\Omega)$ and $s\in [0,k]$. Then} 
\begin{equation*}
\rev{|\mathcal{D}(\textbf{w};\rho,\zeta)-\mathcal{D}_{h}(\textbf{w};\rho,\zeta) |\leq Ch^{s}\left( \Vert \textbf{w}\Vert_{s+1}(\Vert \rho\Vert_{s+1}+\Vert \rho\Vert_{Z})+\Vert \rho\Vert_{s+1}(\Vert \textbf{w}\Vert_{s}+\Vert \textbf{w}\Vert_{\textbf{X}}) \right)\Vert \zeta\Vert_{Z}, \qquad \forall \zeta \in Z.}
\end{equation*}
\end{lemma}
\begin{proof}
First, the definitions of the trilinear continuous and discrete forms $\mathcal{D}(\cdot;\cdot,\cdot)$ and $\mathcal{D}_{h}(\cdot;\cdot,\cdot)$ give
\begin{align}\label{MA1}
\mathcal{D}(\textbf{w};\rho,\zeta)-\mathcal{D}_{h}(\textbf{w};\rho,\zeta)&=\dfrac{1}{2}\big[(\textbf{w} \rho, \nabla \zeta)_{0}-(\boldsymbol{\Pi}_{k}^{0}\textbf{w}\cdot \Pi_{k}^{0} \rho,  \Pi_{k-1}^{0}\nabla \zeta)_{0} \big]\nonumber\\
&\quad -\dfrac{1}{2}\big[(\textbf{w}\cdot\nabla \rho, \zeta)_{0}-(\boldsymbol{\Pi}_{k}^{0}\textbf{w}\cdot  \Pi_{k-1}^{0}\nabla \rho,  \Pi_{k}^{0}\zeta)_{0}  \big]\nonumber\\
&=\rev{\dfrac{1}{2}(\eta_{1}-\eta_{2})}.
\end{align}
We now bound the terms $\eta_{1}$ and $\eta_{2}$. For the first term, elementary calculations show that 
\begin{align}\label{EqQ0}
\eta_{1}|_{E}&=\int_{E}\bigg( (\textbf{w} \rho)\cdot \nabla \zeta -(\boldsymbol{\Pi}_{k}^{0,E}\textbf{w} \Pi_{k}^{0,E} \rho)\cdot \Pi_{k-1}^{0,E}\nabla \zeta \bigg)dE\nonumber\\
&=\int_{E}\bigg(\textbf{w}~\rho~(I-\Pi_{k-1}^{0,E})\nabla \zeta+(I-\boldsymbol{\Pi}_{k}^{0,E})\textbf{w}~\rho~\Pi_{k-1}^{0,E}\nabla \zeta+\boldsymbol{\Pi}_{k}^{0,E}\textbf{w}(I-\Pi_{k}^{0,E})~\rho~\Pi_{k-1}^{0,E}\nabla \zeta\bigg)dE.
\end{align}
Then, using the H\"older inequality, the continuity of $\boldsymbol{\Pi}_{k}^{0,E}$ (with respect to the $L^{4}$-norm), as well as Sobolev embeddings, we can control the terms on the right-hand side of \eqref{MA1} as follows
\begin{subequations}
\begin{align}\label{EqQ1}
\eta_{1}&\leq Ch^{s}\left( \Vert \textbf{w}\Vert_{s+1}(\Vert \rho\Vert_{s+1}+\Vert \rho\Vert_{Z})+\Vert \rho\Vert_{s+1}\Vert \textbf{w}\Vert_{\textbf{X}}\right)\Vert \zeta\Vert_{Z},\\
\label{EqQ2}
\eta_{2}&\leq Ch^{s}\left(\Vert \rho\Vert_{s+1}(\Vert \textbf{w}\Vert_{s}+\Vert \textbf{w}\Vert_{\textbf{X}})+\Vert\textbf{w}\Vert_{s+1}\Vert \rho\Vert_{Z} \right)\Vert \zeta\Vert_{Z}.
\end{align}\end{subequations}
Consequently, the proof follows after putting together \eqref{EqQ1} and \eqref{EqQ2} into \eqref{EqQ0}.
\end{proof}
\begin{lemma}\label{l_ComCPNP}
Assume that $\xi \in H^{s}(\Omega)\cap L^{\infty}(\Omega)$, $\rho\in H^{s+1}(\Omega)\cap W^{1,\infty}(\Omega)$ and $s\in [0,k]$. Then, it holds
\begin{equation*}
\rev{\big| \mathcal{C}(\xi;\rho,\zeta)-\mathcal{C}_{h}(\xi;\rho,\zeta)\big|\leq C\bigg(h^{s}\Vert \xi\nabla \rho\Vert_{s}+\Vert \xi\Vert_{\infty}h^{s}\Vert \rho\Vert_{s+1}+\Vert \rho\Vert_{1,\infty}h^{s}\Vert \xi\Vert_{s} \bigg)\Vert \zeta\Vert_{Z},\quad\forall \zeta\in Z.}
\end{equation*}
\end{lemma}
\begin{proof}
We first write on each element $E\in \mathcal{T}_{h}$ the following relation 
\begin{align*}
\mathcal{C}^{E}(\xi;\rho,\zeta)-\mathcal{C}_{h}^{E}(\xi;\rho,\zeta)&=(\xi\nabla \rho, \nabla \zeta)_{0,E}-\left(\Pi_{k-1}^{0,E}\xi\boldsymbol{\Pi}_{k-1}^{0,E}\nabla \rho, \boldsymbol{\Pi}_{k-1}^{0,E}\nabla \zeta\right)_{0,E}\nonumber\\ 
&=\big[(\xi\nabla \rho, \nabla \zeta)_{0,E}-(\boldsymbol{\Pi}_{k-1}^{0,E}(\xi\nabla \rho), \nabla \zeta)_{0,E} \big]\nonumber\\ 
&\quad +
\big[(\boldsymbol{\Pi}_{k-1}^{0,E}(\xi\nabla \rho), \nabla \zeta)_{0,E}-(\boldsymbol{\Pi}_{k-1}^{0,E}(\xi\boldsymbol{\Pi}_{k-1}^{0,E}\nabla \rho), \nabla \zeta)_{0,E} \big]\nonumber\\ 
&\quad +\big[(\boldsymbol{\Pi}_{k-1}^{0,E}(\xi\boldsymbol{\Pi}_{k-1}^{0,E}\nabla \rho), \nabla \zeta)_{0,E}-(\Pi_{k-1}^{0,E}\xi\boldsymbol{\Pi}_{k-1}^{0,E}\nabla \rho, \boldsymbol{\Pi}_{k-1}^{0,E}\nabla \zeta)_{0,E} \big]\nonumber\\ 
&\leq C\bigg(h^{k}\Vert \xi\nabla \rho\Vert_{k}+\Vert \xi\Vert_{\infty}h^{k}\Vert \rho\Vert_{k+1}+\Vert \boldsymbol{\Pi}_{k-1}^{0,E}\nabla \rho\Vert_{\infty}h^{k}\Vert \xi\Vert_{k} \bigg)\rev{\Vert \zeta\Vert_{Z}},
\end{align*}
where in the last step the approximation properties of the $L^{2}$-projectors are used, and the term $\Vert \boldsymbol{\Pi}_{k-1}^{0,E}\nabla \rho\Vert_{\infty}$ is estimated \rev{by applying the inverse inequality and the continuity   of the projector $\boldsymbol{\Pi}_{k-1}^{0,E}$ as 
\begin{equation}\label{Linfity_Pi}
\Vert \boldsymbol{\Pi}_{k-1}^{0,E}\nabla \rho\Vert_{\infty,E}\leq h_{E}^{-1}\Vert \boldsymbol{\Pi}_{k-1}^{0,E}\nabla \rho\Vert_{0,E}\leq h_{E}^{-1}\Vert \nabla \rho\Vert_{0,E}\leq \Vert \rho\Vert_{1,\infty},
\end{equation}
which completes the proof.}
\end{proof}
\subsection{Semi-discrete and fully-discrete schemes}
With the aid of the discrete forms   \eqref{dis_F1}-\eqref{dis_F6} we can state the semi-discrete VE scheme as: 
Find $\{(c_{1,h}(\cdot,t),c_{2,h}(\cdot,t)),\phi_{h}(\cdot,t) \}\in \textbf{Z}_{h}\times Y_{h}$ and $\{\textbf{u}_{h}(\cdot,t),p_{h}(\cdot,t)\}\in \textbf{X}_{h}\times Q_{h}$ such that for almost all $t\in [0,t_{F}]$
\begin{align*}
  \mathcal{M}_{1,h}(\partial_{t}c_{i,h}, z_{i,h})+\mathcal{A}_{i,h}(c_{i,h}, z_{i,h})+e_{i}
  \mathcal{C}_{h}(c_{i,h};\phi_{h}, z_{i,h})-\mathcal{D}_{h}(\textbf{u}_{h}; c_{i,h}, z_{i,h})&=0 & \forall z_{i,h}\in Z_{h},\\
\mathcal{A}_{3,h}(\phi_{h}, z_{3,h})-\mathcal{M}_{1,h}(c_{1,h}, z_{3,h})+\mathcal{M}_{1,h}(c_{2,h}, z_{3,h})&=0 & \forall z_{3,h}\in Y_{h},\\
\mathcal{M}_{2,h}(\partial_{t}\textbf{u}_{h}, \textbf{v}_{h})+\mathcal{K}_{h}(\textbf{u}_{h}, \textbf{v}_{h})+\mathcal{E}_{h}(\textbf{u}_{h};\textbf{u}_{h}, \textbf{v}_{h})-\mathcal{B}_{h}(p_{h}, \textbf{v}_{h})+\left((c_{1,h}-c_{2,h})\nabla \phi_{h}, \textbf{v}_{h} \right)_{h}&=0 & \forall \textbf{v}_{h}\in \textbf{X}_{h},\\
\mathcal{B}_{h}(q_{h}, \textbf{u}_{h})&=0  & \forall q_{h}\in Q_{h},
\end{align*}
with initial conditions \rev{$c_{i,h}(\cdot,0)=\Pi_{k}^{0}c_{i,0}$} and $\textbf{u}_{h}(\cdot,0)=\boldsymbol{\Pi}_{k}^{0}\textbf{u}_{0}$, and where 
\begin{align*}
\left((c_{1,h}-c_{2,h})\nabla \phi_{h}, \textbf{v} \right)_{h}:=\sum_{E\in\mathcal{T}_{h}}\left((\Pi_{k}^{0,E}c_{1,h}-\Pi_{k}^{0,E}c_{2,h})\boldsymbol{\Pi}_{k-1}^{0,E}\nabla \phi_{h}, \boldsymbol{\Pi}_{k}^{0,E}\textbf{v} \right)_{0,E}.
\end{align*}
Next, we \rev{discretize in time using the backward Euler method with constant step-size $\tau=\frac{t_{F}}{N}$ and
for a generic function $f$, denote $f^{n}=f(\cdot, t_{n})$, $\delta_{t}f^{n}=\frac{f^{n}-f^{n-1}}{\tau}$}.
The fully discrete system reads: for $n=1,\cdots , N$ find $\{ (c_{1,h}^{n}, c_{2,h}^{n}), \rev{\phi_h^{n}}\}\in \textbf{Z}_{h}\times Y_{h}$,~$\{\textbf{u}_{h}^{n},p_{h}^{n}\}\in \textbf{X}_{h}\times Q_{h}$ such that  
\begin{subequations}\label{ful_1}
\begin{align}
\mathcal{M}_{1,h}(\delta_{t}c_{i,h}^{n}, z_{i,h})+\mathcal{A}_{i,h}(c_{i,h}^{n}, z_{i,h})+e_{i}\mathcal{C}_{h}(c_{i,h}^{n};\phi_{h}^{n}, z_{i,h})-\mathcal{D}_{h}(\textbf{u}_{h}^{n}; c_{i,h}^{n}, z_{i,h})&=0,\\
\mathcal{A}_{3,h}(\phi_{h}^{n}, z_{3,h})-\mathcal{M}_{1,h}(c_{1,h}^{n}, z_{3,h})+\mathcal{M}_{1,h}(c_{2,h}^{n}, z_{3,h})&=0,\\
\mathcal{M}_{2,h}(\delta_{t}\textbf{u}_{h}^{n}, \textbf{v}_{h})+\mathcal{K}_{h}(\textbf{u}_{h}^{n}, \textbf{v}_{h})+\mathcal{E}_{h}(\textbf{u}_{h}^{n};\textbf{u}_{h}^{n}, \textbf{v}_{h})-\mathcal{B}_{h}(p_{h}^{n}, \textbf{v}_{h})+\left((c_{1,h}^{n}-c_{2,h}^{n})\nabla \phi_{h}^{n}, \textbf{v}_{h} \right)_{h}&=0,\\
\mathcal{B}_{h}(q_{h}, \textbf{u}_{h}^{n})&=0,\label{eq:full-last}
\end{align}
\end{subequations}
\rev{for all $\{ (z_{1}, z_{2}), z_{3}\}\in \textbf{Z}_{h}\times Y_{h}$ and $\{\textbf{v},q\}\in \textbf{X}_{h}\times Q_{h}$, where $c_{i,h}^{0}=c_{i,h}(\cdot,0)$, $\textbf{u}_{h}^{0}=\textbf{u}_{h}(\cdot,0)$}.

\begin{remark}
\rev{Equation \eqref{eq:full-last}}  along with the property \eqref{propA}, implies that the discrete
velocity $\textbf{u}_{h}^{n}\in \textbf{X}_{k}^{h}$ is exactly divergence-free. More generally, introducing the continuous and discrete kernels:
\[\rev{
\widetilde{\textbf{X}}=\left\{\textbf{v}\in \textbf{X}:~~~\mathcal{B}(q,\textbf{v})=0,~~~\forall q\in Q \right\},\qquad 
\widetilde{\textbf{X}}_{h}=\left\{\textbf{v}_{h}\in \textbf{X}_{h}:~~~\mathcal{B}(q_{h},\textbf{v}_{h})=0,~~~\forall q_{h}\in Q_{h}  \right\}},\]
we can readily check  that $\widetilde{\textbf{X}}_{h}\subseteq \widetilde{\textbf{X}}$. Therefore we \rev{consider the following reduced problem (equivalent to \eqref{ful_1}): Find $\{ (c_{1,h}^{n}, c_{2,h}^{n}), \phi_{h}^{n}\}\in \textbf{Z}_{h}\times Y_{h}$,~$\textbf{u}_{h}^{n}\in \widetilde{\textbf{X}}_{h}$ and $n=1,\cdots , N$ such that
\begin{subequations}\label{ful_3}
\begin{align}\label{ful_3:a}
\mathcal{M}_{1,h}(\delta_{t}c_{i,h}^{n}, z_{i,h})+\mathcal{A}_{i,h}(c_{i,h}^{n}, z_{i,h})+e_{i}\mathcal{C}_{h}(c_{i,h}^{n};\phi_{h}^{n}, z_{i,h})-\mathcal{D}_{h}(\textbf{u}_{h}^{n}; c_{i,h}^{n}, z_{i,h})&=0,\\
\mathcal{A}_{3,h}(\phi_{h}^{n}, z_{3,h})-\mathcal{M}_{1,h}(c_{1,h}^{n}, z_{3,h})+\mathcal{M}_{1,h}(c_{2,h}^{n}, z_{3,h})&=0,\\
\mathcal{M}_{2,h}(\delta_{t}\textbf{u}_{h}^{n}, \textbf{v}_{h})+\mathcal{K}_{h}(\textbf{u}_{h}^{n}, \textbf{v}_{h})+\mathcal{E}_{h}(\textbf{u}_{h}^{n};\textbf{u}_{h}^{n}, \textbf{v}_{h})+\left((c_{1,h}^{n}-c_{2,h}^{n})\nabla \phi_{h}^{n}, \textbf{v}_{h} \right)_{h}&=0,\label{ful_3:c}
\end{align}
\end{subequations} 
for all $\{ (z_{1}, z_{2}), z_{3}\}\in \textbf{Z}_{h}\times Y_{h}$ and $\textbf{v}_{h}\in \widetilde{\textbf{X}}_{h}$}.
\end{remark}
\section{Discrete mass conservativon and discrete thermal energy decay}\label{s4}
This section is devoted to investigate discrete mass conservative and discrete energy decaying properties of \eqref{ful_3}. \rev{To that end, first we recall the $k-$consistency of $\mathcal{M}_{1,h}$ and $\mathcal{A}_{i,h}$.}
\begin{lemma}[\cite{Cangiani17}] \label{l_6}
For every polynomial $q_{k}\in \mathbb{P}_{k}(E)$ and every VE function $z_{h}\in Z_{k}^{E}$ it holds that
\begin{align*}
\mathcal{M}_{h}^{E}(q_{k}, z_{h})=\mathcal{M}^{E}(q_{k}, z_{h}), \quad \text{and}\quad
\mathcal{A}_{i,h}^{E}(q_{k}, z_{h})=\mathcal{A}_{i}^{E}(q_{k}, z_{h}), \qquad \text{for $i=1,2$}.
\end{align*}
\end{lemma}
\begin{theorem}[Discrete mass conservation]\label{T:cons}
Let $\{ (c_{1,h}^{n}, c_{2,h}^{n}, \phi^{n}),\textbf{u}_{h}^{n}\}_{n=1}^{N}$ be a solution of the VE scheme \eqref{ful_3}. Then the \rev{approximate concentrations satisfy}
\begin{equation*}
\rev{\sum_{E\in \mathcal{T}_{h}}\int_{E}c_{i,h}^{n}~dE=\sum_{E\in \mathcal{T}_{h}}\int_{E}c_{i,h}^{0}~dE, \qquad i = 1,2.}
\end{equation*}
\end{theorem}
\begin{proof}
\rev{The proof follows after testing \eqref{ful_3:a} against $z_{i}=1$,~$i=1,2$, and applying Lemma \ref{l_6}.}
\end{proof}

\rev{We now establish a discrete energy decay, independently of the discretization parameters  $h$, $\tau$. We define the total free energy as follows (see \cite{prohl10}):} 
\begin{equation}\label{EH}
\rev{E_{h}(\phi_{h}^{n},\textbf{u}_{h}^{n}):=\dfrac{1}{2}\big[\Vert \phi_{h}^{n}\Vert_{1}^{2}+\Vert \textbf{u}_{h}^{n}\Vert_{0}^{2} \big].}
\end{equation}
\begin{theorem}[Discrete energy decay]\label{T:En}
\rev{Let $\{ (c_{1,h}^{n}, c_{2,h}^{n}, \phi^{n}),\textbf{u}_{h}^{n}\}_{n=1}^{N}$ be a solution of \eqref{ful_3}. Then}
\begin{equation}\label{DisEn}
\rev{E_{h}(\phi_{h}^{n},\textbf{u}_{h}^{n})+\tau\sum_{j=0}^{n}\big[\beta_{1}\Vert c_{1,h}^{j}-c_{2,h}^{j}\Vert_{0}^{2}+\tilde{\beta}_{1} \Vert \textbf{u}_{h}^{j}\Vert_{0}^{2} \big]+\dfrac{\tau^{2}}{2}\big[\beta_{4} \Vert\delta_{t}\phi_{h}^{n}\Vert_{1}^{2}+\tilde{\beta}_{1}\Vert \delta_{t}\textbf{u}_{h}^{n}\Vert_{0}^{2} \big]
\leq E_{h}(\phi_{h}^{0},\textbf{u}_{h}^{0}).}
\end{equation}
\end{theorem}
\begin{proof}
Using \rev{as} test functions \rev{$(z_{i},z_{3})=(\tau\phi_{h}^{n},\tau (c_{1,h}^{n}-c_{2,h}^{n}))$ and $\textbf{v}=\tau\textbf{u}_{h}^{n}$ in \eqref{ful_3}, gives
\begin{subequations}
\begin{align}
\mathcal{M}_{1,h}(\delta_{t}c_{i,h}^{n}, \tau\phi_{h}^{n})+\mathcal{A}_{i,h}(c_{i,h}^{n}, \tau\phi_{h}^{n})+e_{i}\mathcal{C}_{h}(c_{i,h}^{n};\phi_{h}^{n}, \tau\phi_{h}^{n})-\mathcal{D}_{h}(\textbf{u}_{h}^{n}; c_{i,h}^{n}, \tau\phi_{h}^{n})&=0,\label{Sub1}\\[1mm] 
\mathcal{A}_{3,h}(\tau\phi_{h}^{n}, c_{1,h}^{n}-c_{2,h}^{n})-\tau\mathcal{M}_{1,h}(c_{1,h}^{n}-c_{2,h}^{n}, c_{1,h}^{n}-c_{2,h}^{n})&=0,\label{Sub3} \\[1mm]
\mathcal{M}_{2,h}(\delta_{t}\textbf{u}_{h}^{n}, \tau\textbf{u}_{h}^{n})+\mathcal{K}_{h}(\textbf{u}_{h}^{n}, \tau\textbf{u}_{h}^{n})+\mathcal{E}_{h}(\textbf{u}_{h}^{n};\textbf{u}_{h}^{n}, \tau\textbf{u}_{h}^{n})+\left((c_{1,h}^{n}-c_{2,h}^{n})\nabla \phi_{h}^{n}, \tau\textbf{u}_{h}^{n} \right)_{h}&=0.\label{NS}
\end{align}
\end{subequations}
Next we proceed to differentiate \eqref{ful_3:c} with respect to $t$, leading to} 
\begin{equation*}
\rev{\mathcal{A}_{3,h}(\partial_{t}\phi_{h}^{n}, z_{3,h})=\mathcal{M}_{1,h}(\partial_{t}c_{1,h}^{n}, z_{3,h})-\mathcal{M}_{1,h}(\partial_{t}c_{2,h}^{n}, z_{3,h}),\quad\quad \forall z_{3,h}\in Z_{h}.}
\end{equation*}
Using the backward Euler method to approximate the time derivative in the above equation yields
\begin{equation*}
\rev{\mathcal{A}_{3,h}(\delta_{t}\phi_{h}^{n}, z_{3,h})=\mathcal{M}_{1,h}(\delta_{t}c_{1,h}^{n}, z_{3,h})-\mathcal{M}_{1,h}(\delta_{t}c_{2,h}^{n}, z_{3,h}),\quad\quad \forall z_{3,h}\in Z_{h},}
\end{equation*}
and then taking \rev{$z_{3,h}=\tau\phi_{h}^{n}$ implies that}
\begin{equation}\label{EQ_I}
\mathcal{A}_{3,h}(\delta_{t}\phi_{h}^{n}, \tau\phi_{h}^{n})=\mathcal{M}_{1,h}(\delta_{t}c_{1,h}^{n}, \tau\phi_{h}^{n})-\mathcal{M}_{1,h}(\delta_{t}c_{2,h}^{n}, \tau\phi_{h}^{n}).
\end{equation}
Combining \eqref{Sub1}-\eqref{Sub3} and \eqref{EQ_I}, and \rev{using the chain of identities}
\begin{align}\label{Iden1}
\mathcal{A}_{3,h}(\delta_{t}\phi_{h}^{n}, \tau\phi_{h}^{n})&=\mathcal{A}_{3,h}(\phi_{h}^{n}-\phi_{h}^{n-1}, \phi_{h}^{n})=\dfrac{1}{2}\mathcal{A}_{3,h}(\phi_{h}^{n}-\phi_{h}^{n-1}, (\phi_{h}^{n}-\phi_{h}^{n-1})+(\phi_{h}^{n}+\phi_{h}^{n-1}))\nonumber\\
&=\dfrac{\tau^{2}}{2}\mathcal{A}_{3,h}(\delta_{t}\phi_{h}^{n},\delta_{t}\phi_{h}^{n})+\dfrac{\tau}{2}\delta_{t}\mathcal{A}_{3,h}(\phi_{h}^{n},\phi_{h}^{n}),
\end{align}
we can readily \rev{conclude that 
\begin{align*}
\dfrac{\tau^{2}}{2}\mathcal{A}_{3,h}(\delta_{t}\phi_{h}^{n},\delta_{t}\phi_{h}^{n})+\dfrac{\tau}{2}\delta_{t}\mathcal{A}_{3,h}(\phi_{h}^{n},\phi_{h}^{n})&=-\tau\mathcal{M}_{1,h}(c_{1,h}^{n}-c_{2,h}^{n}, c_{1,h}^{n}-c_{2,h}^{n})
-\mathcal{C}_{h}(c_{1,h}^{n}-c_{2,h}^{n};\phi_{h}^{n}, \tau\phi_{h}^{n})\nonumber\\
&\quad +\mathcal{D}_{h}(\textbf{u}_{h}^{n}; c_{1,h}^{n}-c_{2,h}^{n}, \tau\phi_{h}^{n}).
\end{align*}
Also, after applying the fact that} $\mathcal{E}_{h}(\textbf{u}_{h}^{n};\textbf{u}_{h}^{n}, \tau\textbf{u}_{h}^{n})=0$ and its analogous identity \eqref{Iden1} in \eqref{NS}, we obtain
\begin{equation*}
\rev{\dfrac{\tau^{2}}{2}\mathcal{M}_{2,h}(\delta_{t}\textbf{u}_{h}^{n},\delta_{t}\textbf{u}_{h}^{n})+\dfrac{\tau}{2}\delta_{t}\mathcal{M}_{2,h}(\textbf{u}_{h}^{n},\textbf{u}_{h}^{n})+\tau \mathcal{K}_{h}(\textbf{u}_{h}^{n}, \textbf{u}_{h}^{n})=-\tau\left((c_{1,h}^{n}-c_{2,h}^{n})\nabla \phi_{h}^{n}, \textbf{u}_{h}^{n} \right)_{h}.}
\end{equation*}
Summing the two obtained inequalities and employing the \rev{coercivity of} $\mathcal{A}_{3,h}, \mathcal{M}_{1,h}, \mathcal{M}_{2,h}$ stated in \eqref{CORC1} and \eqref{CORC2}, \rev{allows us to assert that
\begin{align*}
\dfrac{\tau^{2}}{2}\big[ \beta_{4}\Vert \delta_{t}\phi_{h}^{n}\Vert_{1}^{2}+\tilde{\beta}_{1}\Vert \delta_{t}\textbf{u}_{h}^{n}\Vert_{0}^{2}\big]&+\dfrac{1}{2}\big[\beta_{4} \Vert\phi_{h}^{n}\Vert_{1}^{2}+\tilde{\beta}_{1}\Vert \textbf{u}_{h}^{n}\Vert_{0}^{2} \big]+\tau \beta_{1}\Vert c_{1,h}^{n}-c_{2,h}^{n}\Vert_{0}^{2}\nonumber\\[1mm]
&+ \tau \tilde{\beta}_{1} \Vert \textbf{u}_{h}^{n}\Vert_{0}^{2}+ \Vert (\Pi_{k-1}^{0}c_{1,h}^{n}+\Pi_{k-1}^{0}c_{2,h}^{n})^{1/2}\boldsymbol{\Pi}_{k-1}^{0}(\nabla\phi_{h}^{n})\Vert_{0}^{2}\leq \dfrac{1}{2}\big[\beta_{4} \Vert\phi_{h}^{n-1}\Vert_{1}^{2}+\tilde{\beta}_{1}\Vert \textbf{u}_{h}^{n-1}\Vert_{0}^{2} \big].
\end{align*}
And}\rev{ summing up the above inequality on $n$ ($1\leq n\leq N$), leads to \eqref{DisEn}. }
\end{proof}

\section{Well-posedness analysis}\label{s5}
\rev{We begin by introducing a fixed-point operator
  \[\textbf{T}: \textbf{Z}_{h} \times \textbf{X}_{h} \to \textbf{Z}_{h} \times \textbf{X}_{h}, \quad
  (\boldsymbol{\xi}_{h}^{n},\textbf{w}_{h}^{n})\mapsto \textbf{T}(\boldsymbol{\xi}_{h}^{n},\textbf{w}_{h}^{n}):=(\widehat{\textbf{c}}_{h}^{n},\widehat{\textbf{u}}_{h}^{n}),\]
with $\boldsymbol{\xi}_{h}^{n}:=(\xi_{1,h}^{n}, \xi_{2,h}^{n})\in \textbf{Z}_{h}$, $\widehat{\textbf{c}}_{h}^{n}:=(\widehat{c}_{1,h}^{n},\widehat{c}_{2,h}^{n})$,  and where $(\widehat{\textbf{c}}_{h}^{n},\widehat{\textbf{u}}_{h}^{n})$ are the first and third components of the solution of the linearized version of problem \eqref{ful_3}: Given $\widehat{\textbf{c}}_{h}^{0}=(\Pi_{k}^{0}c_{1,0},\Pi_{k}^{0}c_{2,0}),~\widehat{\textbf{u}}_{h}^{0}=\boldsymbol{\Pi}_{k}^{0}\textbf{u}_{0}$, find $\{(\widehat{\textbf{c}}_{h}^{n},\widehat{\phi}_{h}^{n}),\widehat{\textbf{u}}_{h}^{n}\}$ for $i=1,2$,~$n=1,\cdots,N$ such \rev{that
 \begin{equation}\label{EQ_EU1}
 \begin{split}
\mathcal{M}_{1,h}(\delta_{t}\widehat{c}_{i,h}^{n}, z_{i,h})+\mathcal{A}_{i,h}(\widehat{c}_{i,h}^{n}, z_{i,h})+e_{i}\mathcal{C}_{h}(\xi_{i,h}^{n};\widehat{\phi}_{h}^{n}, z_{i,h})-\mathcal{D}_{h}(\textbf{w}_{h}^{n};\widehat{c}_{i,h}^{n}, z_{i,h})&=0,\\
\mathcal{A}_{3,h}(\widehat{\phi}_{h}^{n}, z_{3,h})&=\mathcal{M}_{1,h}(\widehat{c}_{1,h}^{n}, z_{3,h})-\mathcal{M}_{1,h}(\widehat{c}_{2,h}^{n}, z_{3,h}),\\
\mathcal{M}_{2,h}(\delta_{t}\widehat{\textbf{u}}_{h}^{n}, \textbf{v}_{h})+\mathcal{K}_{h}(\widehat{\textbf{u}}_{h}^{n}, \textbf{v}_{h})+\mathcal{E}_{h}(\textbf{w}_{h}^{n};\widehat{\textbf{u}}_{h}^{n}, \textbf{v}_{h})&=-\left((\xi_{1,h}^{n}-\xi_{2,h}^{n})\nabla \widehat{\phi}_{h}^{n}, \textbf{v}_{h} \right)_{h}.
\end{split}
 \end{equation}
System} \eqref{EQ_EU1} can be reformulated as follows:
\begin{equation}\label{EEQ_1}
\widehat{\mathbf{A}}_{\boldsymbol{\xi}_{h}^{n},\textbf{w}_{h}^{n}}\left((\widehat{\textbf{c}}_{h}^{n},\widehat{\phi}_{h}^{n}, \widehat{\textbf{u}}_{h}^{n}), (\textbf{z}_{h},z_{3,h},\textbf{v}_{h}) \right)=\mathcal{M}_{1,h}(\widehat{c}_{1,h}^{n-1},z_{1,h})+\mathcal{M}_{1,h}(\widehat{c}_{2,h}^{n-1},z_{2,h})+\mathcal{M}_{2,h}(\widehat{\textbf{u}}_{h}^{n-1},\textbf{v}_{h}),
\end{equation}
where
\begin{equation}\label{defi_Ah}
\widehat{\mathbf{A}}_{\boldsymbol{\xi}_{h}^{n},\textbf{w}_{h}^{n}}\left((\boldsymbol{\rho}_{h}^{n},\zeta_{h}^{n},\textbf{z}_{h}^{n}), (\textbf{z}_{h},z_{3,h},\textbf{v}_{h})\right):= \mathbf{A}\left((\boldsymbol{\rho}_{h}^{n},\zeta_{h}^{n},\textbf{z}_{h}^{n}), (\textbf{z}_{h},z_{3,h},\textbf{v}_{h}) \right)+\tau \mathbf{B}_{\boldsymbol{\xi}_{h}^{n},\textbf{w}_{h}^{n}}\left((\boldsymbol{\rho}_{h}^{n},\zeta_{h}^{n},\textbf{z}_{h}^{n}), (\textbf{z}_{h},z_{3,h},\textbf{v}_{h})\right),
\end{equation}
with 
 \begin{align*}
 &\mathbf{A}\left((\boldsymbol{\rho}_{h}^{n},\zeta_{h}^{n},\textbf{z}_{h}^{n}), (\textbf{z}_{h},z_{3,h},\textbf{v}_{h}) \right):=\mathcal{M}_{1,h}(\rho_{1,h}^{n},z_{1,h})+\tau\mathcal{A}_{1,h}(\rho_{1,h}^{n},z_{1,h})
 +\mathcal{M}_{1,h}(\rho_{2,h}^{n},z_{2,h})+\tau\mathcal{A}_{2,h}(\rho_{2,h}^{n},z_{2,h})\nonumber\\[1mm]
 &\quad +\tau\mathcal{A}_{3,h}(\zeta_{h}^{n},z_{3,h})-\tau\mathcal{M}_{1,h}(\rho_{1,h}^{n},z_{3,h})+\tau\mathcal{M}_{1,h}(\rho_{2,h}^{n},z_{3,h})+\mathcal{M}_{2,h}(\textbf{z}_{h}^{n},\textbf{v}_{h})+\tau\mathcal{K}_{h}(\textbf{z}_{h}^{n},\textbf{v}_{h}),\\
& \mathbf{B}_{\boldsymbol{\xi}_{h}^{n},\textbf{w}_{h}^{n}}\left((\boldsymbol{\rho}_{h}^{n},\zeta_{h}^{n},\textbf{z}_{h}^{n}), (\textbf{z}_{h},z_{3,h},\textbf{v}_{h})\right) :=\mathcal{C}_{h}(\xi_{1,h}^{n};\zeta_{h}^{n},z_{1,h})-\mathcal{C}_{h}(\xi_{2,h}^{n};\zeta_{h}^{n},z_{2,h})-\mathcal{D}_{h}(\textbf{w}_{h}^{n};\rho_{1,h}^{n},z_{1,h})\nonumber\\[1mm]
&\quad -\mathcal{D}_{h}(\textbf{w}_{h}^{n};\rho_{2,h}^{n},z_{2,h})+\mathcal{E}_{h}(\textbf{w}_{h}^{n};\textbf{z}_{h}^{n}, \textbf{v}_{h})+\left((\xi_{1,h}^{n}-\xi_{2,h}^{n})\nabla \zeta_{h}^{n}, \textbf{v}_{h}\right)_{h},
\end{align*}
 \rev{for all $(\textbf{z}_{h},z_{3,h})\in \textbf{Z}_{h}\times Y_{h}$ and $\textbf{v}_{h}\in \textbf{X}_{h}$. 

\begin{lemma}[Discrete global inf-sup condition]\label{l_INFSUP}
For each $\boldsymbol{\xi}_{h}^{n}\in \textbf{Z}_{h}$ and $\textbf{w}_{h}^{n}\in \textbf{X}_{h}$ such that  $\Vert\textbf{w}_{h}^{n}\Vert_{\textbf{X}}\leq \frac{\widehat{\alpha}}{6(\gamma_{1}+\gamma_{2})}$, $\Vert\boldsymbol{\xi}_{h}^{n}\Vert_{\infty}\leq \frac{\widehat{\alpha}}{6\gamma_{3}}$ and $\Vert\boldsymbol{\xi}_{h}^{n}\Vert_{\textbf{Z}}\leq \frac{\widehat{\alpha}}{6\gamma_{4}}$, there exists a positive constant $\widehat{\alpha}$ satisfying
\begin{align}\label{INF_Ah}
\sup _{\left(\textbf{z}_{h},z_{3,h},\textbf{v}_{h}\right) \in \textbf{Z}_{h} \times Y_{h}\times \textbf{X}_{h}\atop
(\textbf{z}_{h},z_{3,h},\textbf{v}_{h}) \neq \mathbf{0}} \frac{\widehat{\mathbf{A}}_{\boldsymbol{\xi}_{h}^{n},\textbf{w}_{h}^{n}}\left((\boldsymbol{\rho}_{h}^{n},\zeta_{h}^{n},\textbf{z}_{h}^{n}), (\textbf{z}_{h},z_{3,h},\textbf{v}_{h})\right)}{\left\|(\textbf{z}_{h},z_{3,h},\textbf{v}_{h})\right\|_{\textbf{Z}\times Y\times \textbf{X}}} &\geq \beta_{1}\Vert \boldsymbol{\rho}_{h}^{n}\Vert_{0}+\tilde{\beta}_{1}\Vert\textbf{z}_{h}^{n}\Vert_{0}+\tau\dfrac{\widehat{\alpha}}{2}
\left\|(\boldsymbol{\rho}_{h}^{n},\zeta_{h}^{n},\textbf{z}_{h}^{n})\right\|_{\textbf{Z} \times Y\times \textbf{X}}.
\end{align}
\end{lemma}
 \begin{proof}
First, note that the ellipticity of $\mathcal{M}_{1,h}, \mathcal{A}_{i,h}$ and $\mathcal{M}_{2,h}, \mathcal{K}_{h}$, will imply an inf-sup condition for $\mathbf{A}$. That is, there exists $\widehat{\alpha}>0$, such that
\begin{equation}\label{INF_SUP}
\begin{aligned}
&\sup_{\left(\textbf{z}_{h},z_{3,h},\textbf{v}_{h}\right) \in \textbf{Z}_{h} \times Y_{h}\times \textbf{X}_{h}\atop
(\textbf{z}_{h},z_{3,h},\textbf{v}_{h}) \neq \mathbf{0}} \frac{\mathbf{A}\left((\boldsymbol{\rho}_{h}^{n},\zeta_{h}^{n},\textbf{z}_{h}^{n}), (\textbf{z}_{h},z_{3,h},\textbf{v}_{h})\right)}{\left\|(\textbf{z}_{h},z_{3,h},\textbf{v}_{h})\right\|_{\textbf{Z}\times Y\times \textbf{X}}}  \geq \beta_{1}\Vert \boldsymbol{\rho}_{h}^{n}\Vert_{0}+\tilde{\beta}_{1}\Vert\textbf{z}_{h}^{n}\Vert_{0}+\tau\widehat{\alpha}
\left\|(\boldsymbol{\rho}_{h}^{n},\zeta_{h}^{n},\textbf{z}_{h}^{n})\right\|_{\textbf{Z} \times Y\times \textbf{X}}, 
\end{aligned}
\end{equation}
for all $\left(\textbf{z}_{h}, z_{3,h},\textbf{v}_{h}\right) \in \textbf{Z}_{h}\times Y_{h}\times\textbf{X}_{h}$. Employing \eqref{INF_SUP} and the boundedness for $\mathcal{C}_{h}, \mathcal{D}_{h}, \mathcal{E}_{h}, (\cdot,\cdot)_{h}$ stated in Lemmas \ref{l_N} and \ref{l_Con_C}, we readily obtain 
\begin{align*}
&\sup_{\left(\textbf{z}_{h},z_{3,h},\textbf{v}_{h}\right) \in \textbf{Z}_{h} \times Y_{h}\times \textbf{X}_{h}\atop
  (\textbf{z}_{h},z_{3,h},\textbf{v}_{h}) \neq \mathbf{0}} \frac{\widehat{\mathbf{A}}_{\boldsymbol{\xi}_{h}^{n},\textbf{w}_{h}^{n}}\left((\boldsymbol{\rho}_{h}^{n},\zeta_{h}^{n},\textbf{z}_{h}^{n}), (\textbf{z}_{h},z_{3,h},\textbf{v}_{h})\right)}{\left\|(\textbf{z}_{h},z_{3,h},\textbf{v}_{h})\right\|_{\textbf{Z}\times Y\times \textbf{X}}} \nonumber\\
&\qquad \geq \beta_{1}\Vert \boldsymbol{\rho}_{h}^{n}\Vert_{0}+\tilde{\beta}_{1}\Vert\textbf{z}_{h}^{n}\Vert_{0}
+\tau(\widehat{\alpha}-(\gamma_{1}+\gamma_{2})\Vert\textbf{w}_{h}^{n}\Vert_{\textbf{X}}-\gamma_{3}\Vert\boldsymbol{\xi}_{h}^{n}\Vert_{\infty}-\gamma_{4}\Vert\boldsymbol{\xi}_{h}^{n}\Vert_{\textbf{Z}})
\left\|(\boldsymbol{\rho}_{h}^{n},\zeta_{h}^{n},\textbf{z}_{h}^{n})\right\|_{\textbf{Z} \times Y\times \textbf{X}},
\end{align*}
and  \eqref{INF_Ah} follows as a consequence of the assumptions $\Vert\textbf{w}_{h}^{n}\Vert_{\textbf{X}}\leq \frac{\widehat{\alpha}}{6(\gamma_{1}+\gamma_{2})}$, $\Vert\boldsymbol{\xi}_{h}^{n}\Vert_{\infty}\leq \frac{\widehat{\alpha}}{6\gamma_{3}}$ and $\Vert\boldsymbol{\xi}_{h}^{n}\Vert_{Z}\leq \frac{\widehat{\alpha}}{6\gamma_{4}}$.
 \end{proof}
 Now we are ready to show that  $\textbf{T}$ is well-defined, or equivalently, that problem \eqref{EEQ_1} is uniquely solvable.}
\begin{lemma}[Well-definedness of $\textbf{T}$]\label{l_S1}
Let the assumptions of Lemma \ref{l_INFSUP} be satisfied. Then, there exists a unique $\{(\widehat{\textbf{c}}_{h}^{n},\widehat{\phi}_{h}^{n}),\widehat{\textbf{u}}_{h}^{n}\}$ solution to \eqref{EQ_EU1}. In addition, \rev{for any $1\leq n\leq N$}, there holds
\begin{align}\label{E_Stab}
\Vert \textbf{T}(\boldsymbol{\xi}_{h}^{n},\textbf{w}_{h}^{n})\Vert_{0}+\tau\sum_{j=0}^{n}\Vert \textbf{T}(\boldsymbol{\xi}_{h}^{j},\textbf{w}_{h}^{j})\Vert_{\textbf{Z}\times\textbf{X}}&=\Vert \widehat{\textbf{c}}_{h}^{n}\Vert_{0}+\Vert\widehat{\textbf{u}}_{h}^{n}\Vert_{0}+\tau\sum_{j=0}^{n}(\Vert\widehat{\textbf{c}}_{h}^{j}\Vert_{\textbf{Z}}+\Vert\widehat{\textbf{u}}_{h}^{j}\Vert_{\textbf{X}})\nonumber\\[1mm]
&\leq \max\{ c_{1},c_{2}\}\max\{ c_{1},c_{2},\frac{\widehat{\alpha}}{2}\}\left(\Vert \textbf{c}_{0}\Vert_{0}+\Vert\textbf{u}_{0}\Vert_{0}\right).
\end{align} 
\end{lemma}
 \begin{proof}
 \rev{A straightforward application of the classical Babu\v{s}ka--Brezzi theory and Lemma \ref{l_INFSUP} implies that problem \eqref{EQ_EU1} is well-posed. The continuous dependence on data then gives}
\begin{equation}
\beta_{1}\Vert \widehat{\textbf{c}}_{h}^{n}\Vert_{0}+\tilde{\beta}_{1}\Vert\widehat{\textbf{u}}_{h}^{n}\Vert_{0}+\tau\frac{\widehat{\alpha}}{2}
\left\|(\widehat{\textbf{c}}_{h}^{n},\widehat{\phi}_{h}^{n},\widehat{\textbf{u}}_{h}^{n})\right\|_{\textbf{Z} \times Y\times \textbf{X}}\leq \alpha_{1}\Vert \widehat{\textbf{c}}_{h}^{n-1}\Vert_{0}+\tilde{\alpha}_{1}\Vert\widehat{\textbf{u}}_{h}^{n-1}\Vert_{0},
\end{equation}
which, after a simple manipulation of the terms, leads to 
\begin{equation}
\underbrace{\min\{\beta_{1},\alpha_{1} \}}_{c_{1}}\big[\Vert \widehat{\textbf{c}}_{h}^{n}\Vert_{0}-\Vert \widehat{\textbf{c}}_{h}^{n-1}\Vert_{0} \big]+\underbrace{\min\{\tilde{\beta}_{1},\tilde{\alpha}_{1}\}}_{c_{2}}\big[ \Vert\widehat{\textbf{u}}_{h}^{n}\Vert_{0}-\Vert\widehat{\textbf{u}}_{h}^{n-1}\Vert_{0}\big]+\tau\frac{\widehat{\alpha}}{2}
\left\|(\widehat{\textbf{c}}_{h}^{n},\widehat{\phi}_{h}^{n},\widehat{\textbf{u}}_{h}^{n})\right\|_{\textbf{Z} \times Y\times \textbf{X}}\leq 0.
\end{equation}
\rev{Summing up the above inequality for $n$}, produces
\begin{equation}
\Vert \widehat{\textbf{c}}_{h}^{n}\Vert_{0}+\Vert\widehat{\textbf{u}}_{h}^{n}\Vert_{0}+\tau\sum_{j=0}^{n}\left\|(\widehat{\textbf{c}}_{h}^{j},\widehat{\phi}_{h}^{j},\widehat{\textbf{u}}_{h}^{j})\right\|_{\textbf{Z} \times Y\times \textbf{X}}\leq \max\{ c_{1},c_{2}\}\max\{ c_{1},c_{2},\frac{\widehat{\alpha}}{2}\}\left(\Vert \textbf{c}_{0}\Vert_{0}+\Vert\textbf{u}_{0}\Vert_{0}\right),
\end{equation}
which, yields \eqref{E_Stab}.
 \end{proof}
 
The next step is to show that $\textbf{T}$ maps a closed ball in $\textbf{Z}_{h}\times \textbf{X}_{h}$ into itself. Let us define the set
 \[
V_{h}:=\bigg\{ (\boldsymbol{\xi}_{h}^{n},\textbf{w}_{h}^{n})\in \textbf{Z}_{h}\times\textbf{X}_{h}:\quad \Vert\textbf{w}_{h}^{n}\Vert_{\textbf{X}}\leq \frac{\widehat{\alpha}}{6(\gamma_{1}+\gamma_{2})},\quad \Vert\boldsymbol{\xi}_{h}^{n}\Vert_{\infty}\leq \frac{\widehat{\alpha}}{6\gamma_{3}}, \quad\Vert\boldsymbol{\xi}_{h}^{n}\Vert_{\textbf{Z}}\leq \frac{\widehat{\alpha}}{6\gamma_{4}} \bigg\}.
\]

\begin{lemma}\label{l_S2}
Let
\[
C_{\mathtt{stab}}:=\max\{ c_{1},c_{2}\}\max\{ c_{1},c_{2},\frac{\widehat{\alpha}}{2}\},\quad\text{and}\quad \tilde{C}_{\mathtt{stab}}:=t_{F}^{-1}\min\{\frac{\widehat{\alpha}}{6\alpha_{4}},\frac{\widehat{\alpha}}{6(\gamma_{1}+\gamma_{2})}\}.
\]
Suppose that the data satisfy
\begin{equation}\label{AS_Dat}
C_{\mathtt{stab}}\tilde{C}_{\mathtt{stab}}\left(\Vert c_{1,0}\Vert_{0} +\Vert c_{2,0}\Vert_{0}+\Vert\textbf{u}_{0}\Vert_{0}\right)\leq 1.
\end{equation}
Then, $\textbf{T}(V_{h})\subset V_{h}$.
\end{lemma}
\begin{proof}
It is deduced straightforwardly from Lemma \ref{l_S1} and the a priori estimate stated by \eqref{E_Stab}.
\end{proof}

\begin{lemma}[Lipschitz-continuity of $\mathbf{T}$]\label{l_S2a}
For any $1\leq n\leq N$, there holds
\begin{align}\label{LS_T}
\Vert \textbf{T}(\boldsymbol{\xi}_{h}^{n},\textbf{w}_{h}^{n})-\textbf{T}(\boldsymbol{\rho}_{h}^{n},\textbf{z}_{h}^{n})\Vert_{\textbf{Z}\times\textbf{X}}&\leq \hat{C}_{\mathtt{stab}}\left(\Vert\boldsymbol{\xi}_{h}^{n}-\boldsymbol{\rho}_{h}^{n}\Vert_{\textbf{Z}}+\Vert\textbf{w}_{h}^{n}-\textbf{z}_{h}^{n}\Vert_{\textbf{X}}\right),
\end{align}
in which 
\[
\hat{C}_{\mathtt{stab}}:=2c_{\mathtt{p}}^{2}\max\{\frac{\alpha_{1}}{3\beta_{4}},\frac{\gamma_{1}}{6\gamma_{4}}+\frac{\gamma_{2}}{6(\gamma_{1}+\gamma_{2})}\}.
\]
\end{lemma}
\begin{proof}
Given $(\boldsymbol{\xi}_{h}^{n},\textbf{w}_{h}^{n})\in V_{h}$ and $(\boldsymbol{\rho}_{h}^{n},\textbf{z}_{h}^{n})\in V_{h}$, we let $\textbf{T}(\boldsymbol{\xi}_{h}^{n},\textbf{w}_{h}^{n})=(\widehat{\textbf{c}}_{h}^{n},\widehat{\textbf{u}}_{h}^{n})\in V_{h}$ and $\textbf{T}(\boldsymbol{\rho}_{h}^{n},\textbf{z}_{h}^{n})=(\tilde{\textbf{c}}_{h}^{n},\tilde{\textbf{u}}_{h}^{n})\in V_{h}$, where $\{(\widehat{\textbf{c}}_{h}^{n},\widehat{\phi}_{h}^{n}),\widehat{\textbf{u}}_{h}^{n})\}$ and $\{(\tilde{\textbf{c}}_{h}^{n},\tilde{\phi}_{h}^{n}),\tilde{\textbf{u}}_{h}^{n}\}$  are the unique solutions of \eqref{EQ_EU1} (equivalently \eqref{EEQ_1}) with $(\boldsymbol{\xi}_{h}^{n},\textbf{w}_{h}^{n})=(\boldsymbol{\rho}_{h}^{n},\textbf{z}_{h}^{n})$. It follows from \eqref{EEQ_1} that
\begin{equation}
\widehat{\mathbf{A}}_{\boldsymbol{\xi}_{h}^{n},\textbf{w}_{h}^{n}}\left((\widehat{\textbf{c}}_{h}^{n},\widehat{\phi}_{h}^{n}, \widehat{\textbf{u}}_{h}^{n}), (\textbf{z}_{h},z_{3,h},\textbf{v}_{h}) \right)=
\widehat{\mathbf{A}}_{\boldsymbol{\rho}_{h}^{n},\textbf{z}_{h}^{n}}\left((\tilde{\textbf{c}}_{h}^{n},\tilde{\phi}_{h}^{n}, \tilde{\textbf{u}}_{h}^{n}), (\textbf{z}_{h},z_{3,h},\textbf{v}_{h}) \right),
\end{equation}
for all $(\textbf{z}_{h}, z_{3,h},\textbf{v}_{h})\in \textbf{Z}_{h}\times Y_{h}\times \textbf{X}_{h}$,  which, according to the definition of $\widehat{\mathbf{A}}_{\boldsymbol{\xi}_{h}^{n},\textbf{w}_{h}^{n}}$ (cf.
\eqref{defi_Ah}), becomes
\begin{align*}
\mathbf{A}\left((\widehat{\textbf{c}}_{h}^{n},\widehat{\phi}_{h}^{n}, \widehat{\textbf{u}}_{h}^{n})-(\tilde{\textbf{c}}_{h}^{n},\tilde{\phi}_{h}^{n}, \tilde{\textbf{u}}_{h}^{n}), (\textbf{z}_{h},z_{3,h},\textbf{v}_{h}) \right)&=\tau\big[\textbf{B}_{\boldsymbol{\rho}_{h}^{n},\textbf{z}_{h}^{n}}((\tilde{\textbf{c}}_{h}^{n},\tilde{\phi}_{h}^{n}, \tilde{\textbf{u}}_{h}^{n}), (\textbf{z}_{h},z_{3,h},\textbf{v}_{h}))-\textbf{B}_{\boldsymbol{\xi}_{h}^{n},\textbf{w}_{h}^{n}}((\widehat{\textbf{c}}_{h}^{n},\widehat{\phi}_{h}^{n}, \widehat{\textbf{u}}_{h}^{n}), (\textbf{z}_{h},z_{3,h},\textbf{v}_{h}))\big].
\end{align*}
This result, combined with
\eqref{defi_Ah} by setting $(\boldsymbol{\rho}_{h}^{n},\zeta_{h}^{n},\textbf{z}_{h}^{n})=(\widehat{\textbf{c}}_{h}^{n}-\tilde{\textbf{c}}_{h}^{n},\widehat{\phi}_{h}^{n}-\tilde{\phi}_{h}^{n},\widehat{\textbf{u}}_{h}^{n}-\tilde{\textbf{u}}_{h}^{n})$ yields
\begin{align*}
\widehat{\mathbf{A}}_{\boldsymbol{\xi}_{h}^{n},\textbf{w}_{h}^{n}}\left((\widehat{\textbf{c}}_{h}^{n},\widehat{\phi}_{h}^{n}, \widehat{\textbf{u}}_{h}^{n})-(\tilde{\textbf{c}}_{h}^{n},\tilde{\phi}_{h}^{n}, \tilde{\textbf{u}}_{h}^{n}), (\textbf{z}_{h},z_{3,h},\textbf{v}_{h}) \right)&=\mathbf{A}\left((\widehat{\textbf{c}}_{h}^{n},\widehat{\phi}_{h}^{n}, \widehat{\textbf{u}}_{h}^{n})-(\tilde{\textbf{c}}_{h}^{n},\tilde{\phi}_{h}^{n}, \tilde{\textbf{u}}_{h}^{n}), (\textbf{z}_{h},z_{3,h},\textbf{v}_{h}) \right)\nonumber\\[1mm]
&\quad +\tau\textbf{B}_{\boldsymbol{\xi}_{h}^{n},\textbf{w}_{h}^{n}}\left((\widehat{\textbf{c}}_{h}^{n},\widehat{\phi}_{h}^{n}, \widehat{\textbf{u}}_{h}^{n})-(\tilde{\textbf{c}}_{h}^{n},\tilde{\phi}_{h}^{n}, \tilde{\textbf{u}}_{h}^{n}), (\textbf{z}_{h},z_{3,h},\textbf{v}_{h})\right)\nonumber\\[1mm]
&=\tau\textbf{B}_{\boldsymbol{\rho}_{h}^{n}-\boldsymbol{\xi}_{h}^{n},\textbf{z}_{h}^{n}-\textbf{w}_{h}^{n}}\left((\tilde{\textbf{c}}_{h}^{n},\tilde{\phi}_{h}^{n}, \tilde{\textbf{u}}_{h}^{n}), (\textbf{z}_{h},z_{3,h},\textbf{v}_{h}) \right).
\end{align*}
Hence, we apply the inf-sup condition stated in Lemma \ref{l_INFSUP} in the left-hand side of the above equation and utilize the estimates given in Lemmas \ref{l_N} and \ref{l_Con_C} for $\mathcal{C}_{h}, \mathcal{D}_{h}, \mathcal{E}_{h}, (\cdot,\cdot)_{h}$, to get
\begin{align}\label{OP_2}
&\tau\dfrac{\widehat{\alpha}}{2}\left\|(\widehat{\textbf{c}}_{h}^{n},\widehat{\phi}_{h}^{n}, \widehat{\textbf{u}}_{h}^{n})-(\tilde{\textbf{c}}_{h}^{n},\tilde{\phi}_{h}^{n}, \tilde{\textbf{u}}_{h}^{n})\right\|_{\textbf{Z}\times Y\times \textbf{X}}\leq \hspace{-.3cm}\sup _{
(\textbf{z}_{h},z_{3,h},\textbf{v}_{h}) \neq \mathbf{0}} \frac{\tau\textbf{B}_{\boldsymbol{\rho}_{h}^{n}-\boldsymbol{\xi}_{h}^{n},\textbf{z}_{h}^{n}-\textbf{w}_{h}^{n}}\left((\tilde{\textbf{c}}_{h}^{n},\tilde{\phi}_{h}^{n}, \tilde{\textbf{u}}_{h}^{n}), (\textbf{z}_{h},z_{3,h},\textbf{v}_{h}) \right)}{\left\|(\textbf{z}_{h},z_{3,h},\textbf{v}_{h})\right\|_{\textbf{Z}\times Y\times \textbf{X}}}\nonumber\\[1mm]
&=\tau\hspace{-.3cm}\sup _{
(\textbf{z}_{h},z_{3,h},\textbf{v}_{h}) \neq \mathbf{0}}\frac{\mathcal{C}_{h}(\boldsymbol{\rho}_{h}^{n}-\boldsymbol{\xi}_{h}^{n};\tilde{\phi}_{h}^{n},\textbf{z}_{h})-\mathcal{D}_{h}(\textbf{z}_{h}^{n}-\textbf{w}_{h}^{n};\tilde{\textbf{c}}_{h}^{n},\textbf{z}_{h})+\mathcal{E}_{h}(\textbf{z}_{h}^{n}-\textbf{w}_{h}^{n};\tilde{\textbf{u}}_{h}^{n},\textbf{v}_{h})+((\boldsymbol{\rho}_{h}^{n}-\boldsymbol{\xi}_{h}^{n})\nabla\tilde{\phi}_{h}^{n},\textbf{v}_{h})_{h}}{\left\|(\textbf{z}_{h},z_{3,h},\textbf{v}_{h})\right\|_{\textbf{Z}\times Y\times \textbf{X}}}\nonumber\\[1mm]
&\leq \tau\bigg\{\gamma_{3}\Vert\boldsymbol{\rho}_{h}^{n}-\boldsymbol{\xi}_{h}^{n}\Vert_{0}\Vert\tilde{\phi}_{h}^{n}\Vert_{1,\infty}+\gamma_{1}\Vert\textbf{z}_{h}^{n}-\textbf{w}_{h}^{n}\Vert_{\textbf{X}}\Vert\tilde{\textbf{c}}_{h}^{n}\Vert_{\textbf{Z}}+\gamma_{2}\Vert\textbf{z}_{h}^{n}-\textbf{w}_{h}^{n}\Vert_{\textbf{X}}\Vert\tilde{\textbf{u}}_{h}^{n}\Vert_{\textbf{X}}+\gamma_{4}\Vert\boldsymbol{\rho}_{h}^{n}-\boldsymbol{\xi}_{h}^{n}\Vert_{\textbf{Z}}\Vert\tilde{\phi}_{h}^{n}\Vert_{Y}\bigg\}\nonumber\\[1mm]
&\leq \tau(\gamma_{3}c_{\mathtt{p}}\Vert\tilde{\phi}_{h}^{n}\Vert_{1,\infty}+\gamma_{4}\Vert\tilde{\phi}_{h}^{n}\Vert_{Y})\Vert\boldsymbol{\rho}_{h}^{n}-\boldsymbol{\xi}_{h}^{n}\Vert_{\textbf{Z}}+\tau(\gamma_{1}\frac{\widehat{\alpha}}{6\gamma_{4}}+\gamma_{2}\frac{\widehat{\alpha}}{6(\gamma_{1}+\gamma_{2})})\Vert\textbf{z}_{h}^{n}-\textbf{w}_{h}^{n}\Vert_{\textbf{X}},
\end{align}
where in the last inequality we have used the Poincar\'e inequality and the fact that $(\tilde{\textbf{c}}_{h}^{n},\tilde{\textbf{u}}_{h}^{n})\in V_{h}$. In addition, a bound for the terms $\Vert\tilde{\phi}_{h}^{n}\Vert_{Y}$ and $\Vert\tilde{\phi}_{h}^{n}\Vert_{1,\infty}$ can be derived using that $\{(\tilde{\textbf{c}}_{h}^{n},\tilde{\phi}_{h}^{n}),\tilde{\textbf{u}}_{h}^{n}\}$ is actually a solution to \eqref{EQ_EU1}, that is 
\begin{equation*}
\mathcal{A}_{3,h}(\tilde{\phi}_{h}^{n}, z_{3,h})=\mathcal{M}_{1,h}(\tilde{c}_{1,h}^{n}, z_{3,h})-\mathcal{M}_{1,h}(\tilde{c}_{2,h}^{n}, z_{3,h}).
\end{equation*}
Letting $z_{3,h}=\tilde{\phi}_{h}^{n}$ in the above equation and invoking Eqs. \eqref{BUND1} and \eqref{CORC1}, we readily get
\begin{equation*}
\beta_{4}\Vert\tilde{\phi}_{h}^{n}\Vert_{Y}^{2}\leq \alpha_{1}\left(\Vert\tilde{c}_{1,h}^{n}\Vert_{0}+\Vert\tilde{c}_{2,h}^{n}\Vert_{0} \right)\Vert\tilde{\phi}_{h}^{n}\Vert_{0}.
\end{equation*}
\rev{Appealing to Poincar\'e and inverse inequalities}, implies that
\[
\Vert\tilde{\phi}_{h}^{n}\Vert_{Y}\leq \frac{\alpha_{1}}{\beta_{4}}c_{\mathtt{p}}^{2}\Vert\tilde{\textbf{c}}_{h}^{n}\Vert_{\textbf{Z}},\quad\quad\text{and}\quad\quad \Vert\tilde{\phi}_{h}^{n}\Vert_{1,\infty}\leq \frac{\alpha_{1}}{\beta_{4}}c_{\mathtt{p}}\Vert\tilde{\textbf{c}}_{h}^{n}\Vert_{\infty},
\]
which together with the fact that $(\tilde{\textbf{c}}_{h}^{n},\tilde{\textbf{u}}_{h}^{n})\in V_{h}$, leads us to
\begin{equation}\label{OP_1}
\Vert\tilde{\phi}_{h}^{n}\Vert_{Y}\leq \frac{\alpha_{1}\widehat{\alpha}}{6\gamma_{4}\beta_{4}}c_{\mathtt{p}}^{2},\quad\quad\text{and}\quad\quad\Vert\tilde{\phi}_{h}^{n}\Vert_{1,\infty}\leq \frac{\widehat{\alpha}\alpha_{1}}{6\gamma_{3}\beta_{4}}c_{\mathtt{p}}.
\end{equation} 
Finally, combining \eqref{OP_1}, \eqref{OP_2} and observing that  
\[\Vert \textbf{T}(\boldsymbol{\xi}_{h}^{n},\textbf{w}_{h}^{n})-\textbf{T}(\boldsymbol{\rho}_{h}^{n},\textbf{z}_{h}^{n})\Vert_{\textbf{Z}\times \textbf{X}}\leq 
\Vert (\widehat{\textbf{c}}_{h}^{n},\widehat{\phi}_{h}^{n},\widehat{\textbf{u}}_{h}^{n})-(\tilde{\textbf{c}}_{h}^{n},\tilde{\phi}_{h}^{n},\tilde{\textbf{u}}_{h}^{n})\Vert_{\textbf{Z}\times Y\times \textbf{X}},\]
the desired continuity follows.
\end{proof}

The main result of this section is summarized in the following theorem.
\begin{theorem}\label{T_1}
Assume that the data satisfy
\begin{equation}
C_{\mathtt{stab}}\tilde{C}_{\mathtt{stab}}\left(\Vert c_{1,0}\Vert_{0} +\Vert c_{2,0}\Vert_{0}+\Vert\textbf{u}_{0}\Vert_{0}\right)\leq 1.
\end{equation}
Then, there exists a unique solution $\{(\textbf{c}_{h}^{n},\phi_{h}^{n}),\textbf{u}_{h}^{n}\}\in \textbf{Z}_{h}\times Y_{h}\times \textbf{X}_{h}$ with $(\textbf{c}_{h}^{n},\textbf{u}_{h}^{n})\in V_{h}$ for the fully discrete problem \eqref{ful_3}, \rev{and for any $1\leq n\leq N$}, there holds
\begin{equation}\label{STB}
\rev{\Vert \textbf{c}_{h}^{n}\Vert_{0}+\Vert\textbf{u}_{h}^{n}\Vert_{0}+\tau\sum_{j=0}^{n}\left\|(\textbf{c}_{h}^{j},\phi_{h}^{j},\textbf{u}_{h}^{j})\right\|_{\textbf{Z} \times Y\times \textbf{X}}\leq C_{\mathtt{stab}}\left(\Vert c_{1,0}\Vert_{0}+\Vert c_{2,0}\Vert_{0}+\Vert\textbf{u}_{0}\Vert_{0} \right).}
\end{equation}
\end{theorem}
\begin{proof}
Firstly we realize that solving \eqref{ful_3} (or equivalently \eqref{EEQ_1}) is equivalent to finding  $\textbf{c}_{h}^{n}, \textbf{u}_{h}^{n}$ such that
 \begin{equation}\label{FP_1}
 \left\{\begin{array}{ll}
 \textbf{T}(\textbf{c}_{h}^{n}, \textbf{u}_{h}^{n})=(\textbf{c}_{h}^{n},\textbf{u}_{h}^{n}),~~~~~~~~~~~~n=1,\cdots,N,\\[2mm]
\textbf{c}_{h}^{0}=(\Pi_{k}^{0}c_{1,0},\Pi_{k}^{0}c_{2,0}),~\textbf{u}_{h}^{0}=\boldsymbol{\Pi}_{k}^{0}\textbf{u}_{0}.
  \end{array}\right.
 \end{equation} 
Finally, the compactness of $\textbf{T}$ (on the ball $V_{h}$) and its Lipschitz-continuity are guaranteed by Lemmas 
\ref{l_S2} and \ref{l_S2a}. Hence, it suffices to apply Banach's fixed-point theorem to the fully discrete VE scheme \eqref{ful_3} to conclude the existence and uniqueness of solution. Furthermore, the stability result \eqref{STB} is derived directly from \eqref{E_Stab}, provided in Lemma \ref{l_S1}.
\end{proof}
}

\section{Convergence analysis}\label{s6}
We split the error analysis in two steps. First one estimates the velocity and pressure discretization errors, $\Vert \textbf{u}^{n}-\textbf{u}_{h}^{n}\Vert_{0}$ and $\Vert p^{n}-p_{h}^{n}\Vert_{0}$, respectively; and  the second stage corresponds to establishing bounds for the concentrations error 
\rev{$\Vert c_{i}^{n}-c_{i,h}^{n}\Vert_{0}$} and electrostatic potential error $\Vert \phi^{n}-\phi_{h}^{n}\Vert_{0}$. \rev{
For this purpose, we recall the following estimate as corollary from Theorem \ref{T_1} under assumption \eqref{AS_Dat}
\begin{equation}\label{RR1}
\Vert c_{1,h}^{n}\Vert_{Z}+\Vert c_{2,h}^{n}\Vert_{Z}\leq \dfrac{\widehat{\alpha}}{6\gamma_{4}},\quad \Vert c_{1,h}^{n}\Vert_{\infty}+\Vert c_{2,h}^{n}\Vert_{\infty}\leq \dfrac{\widehat{\alpha}}{6\gamma_{3}},\quad 
\Vert\tilde{\phi}_{h}^{n}\Vert_{1,\infty}\leq \frac{\widehat{\alpha}\alpha_{1}}{6\gamma_{3}\beta_{4}}c_{\mathtt{p}}.
\end{equation}
Also, we notice the following a priori estimate, which  can be derived similarly to the proof of Theorem \ref{T_1} under a similar assumption with \eqref{AS_Dat}
\begin{equation}\label{RR2}
\Vert c_{1}^{n}\Vert_{Z}+\Vert c_{2}^{n}\Vert_{Z}\leq \dfrac{\widehat{\beta}}{6\tilde{\gamma}_{4}},\quad \Vert c_{1}^{n}\Vert_{\infty}+\Vert c_{2}^{n}\Vert_{\infty}\leq \dfrac{\widehat{\beta}}{6\tilde{\gamma}_{3}},\quad 
\Vert\tilde{\phi}^{n}\Vert_{1,\infty}\leq \frac{\widehat{\beta}\tilde{\alpha}_{1}}{6\tilde{\gamma}_{3}\tilde{\beta}_{4}}c_{\mathtt{p}}.
\end{equation}
where constants $\tilde{\gamma}_{1}, \tilde{\gamma}_{2}, \tilde{\gamma}_{4}$ are the bounds of continuous forms $\mathcal{D}, \mathcal{E}, (\cdot;\cdot,\cdot)_{0,\Omega}$, respectively.
}
\subsection{Error bounds: velocity and pressure}
We consider the following problem:
\begin{equation}\label{part1}
\rev{\mathcal{M}_{2,h}(\delta_{t}\textbf{u}_{h}^{n}, \textbf{v})+\mathcal{K}_{h}(\textbf{u}_{h}^{n}, \textbf{v})+\mathcal{E}_{h}(\textbf{u}_{h}^{n};\textbf{u}_{h}^{n}, \textbf{v})=-\left((c_{1,h}^{n}-c_{2,h}^{n})\nabla \phi_{h}^{n}, \textbf{v} \right)_{h},\quad \forall \textbf{v}\in \widetilde{\textbf{X}}_{h},}
\end{equation}
where \rev{$\{c_{1,h}^{n},c_{2,h}^{n},\phi_{h}^{n}\}$} is an approximate solution of the PNP system \eqref{ful_1} and $\textbf{u}_{h}^{0}=\textbf{u}_{h,0}$ with $n=1,\cdots , N$.  The aim is to obtain an error bound for $\Vert \textbf{u}^{n}-\textbf{u}_{h}^{n}\Vert_{0}$ and $\Vert p^{n}-p_{h}^{n}\Vert_{0}$ dependent on \rev{$\Vert c_{i}^{n}-c_{i,h}^{n}\Vert_{0}$} and $\Vert \phi^{n}-\phi_{h}^{n}\Vert_{0}$.
\begin{theorem}\label{T2}
\rev{Suppose that the data satisfy
\begin{equation}\label{AS_T2}
\tilde{\beta}_{2}^{-1}C_{\mathtt{stab}}\left(\Vert c_{1,0}\Vert_{0} +\Vert c_{2,0}\Vert_{0}+\Vert\textbf{u}_{0}\Vert_{0}\right)\leq \frac{1}{4}.
\end{equation}
Given $\{\textbf{c}_{h}^{n},\phi_{h}^{n}\}\in \textbf{Z}_{h}\times Y_{h}$}, let $\textbf{u}_{h}^{n}\in \widetilde{\textbf{X}}_{h}$ be the solution to \eqref{part1} and 
$\{ \textbf{c}^{n}, \phi^{n}\}$, $\{\textbf{u}^{n},p^{n}\}$ be the solution of \eqref{Eq_1} satisfying the following regularity conditions
\[
\left\|\partial_{t} \textbf{u}\right\|_{L^{\infty}(H^{k+1})}+\Vert \textbf{u}\Vert_{L^{\infty}(H^{k+1})}+\left\|\partial_{tt} \textbf{u}^{n}\right\|_{L^{2}\left( L^{2}\right)}+ \left\|\partial_{t} \textbf{u}^{n}\right\|_{L^{2}\left( H^{k+1}\right)}+\left(\Vert \textbf{u}^{n}\Vert_{k}+\Vert \textbf{u}^{n}\Vert_{1}+\Vert \textbf{u}^{n}\Vert_{k+1}\! +\!1\right)\Vert \textbf{u}^{n}\Vert_{k+1}\leq C,
\]
\[(\Vert c_{1}^{n}\Vert_{1}+\Vert c_{2}^{n}\Vert_{1})(\Vert\phi^{n}\Vert_{2}+\Vert\phi^{n}\Vert_{s+1})+(\Vert c_{1}^{n}\Vert_{k+1}+\Vert c_{2}^{n}\Vert_{k+1})\Vert\phi^{n}\Vert_{1}\leq C.
\]
Then, for all $k\in\mathbb{N}_{0}$, the following estimate holds
\begin{equation*}
\rev{\Vert \textbf{u}^{n}-\textbf{u}_{h}^{n}\Vert_{0}^{2}+\tau\sum_{j=1}^{n} \Vert  \textbf{u}^{j}-\textbf{u}_{h}^{j}\Vert_{\textbf{X}}^{2} \leq C(\tau^{2}+h^{2k})+C\tau\sum_{j=1}^{n}\big[\Vert c_{1}^{j}-c_{1,h}^{j}\Vert_{0}^{2}+\Vert c_{2}^{j}-c_{2,h}^{j}\Vert_{0}^{2}+\Vert \phi^{j}-\phi_{h}^{j}\Vert_{Y}^{2}\big].}
\end{equation*}
\end{theorem}
\begin{proof}
The proof is conducted in three steps:

\medskip 
\noindent\textbf{Step 1:  evolution equation for the error.} 
Setting $\boldsymbol{\vartheta}_{\textbf{u}}^{n}:=\textbf{u}_{h}^{n}-\textbf{u}_{I}^{n}$, it holds that $\boldsymbol{\vartheta}_{\textbf{u}}^{n}\in \widetilde{\textbf{X}}_{h}$. Using Eq. \eqref{part1} and the fourth equation in  \eqref{Eq_1} and letting $\textbf{v}=\boldsymbol{\vartheta}_{\textbf{u}}^{n}$, \rev{yield
\begin{align}\label{EST1b}
&\mathcal{M}_{2,h}(\dfrac{\boldsymbol{\vartheta}_{\textbf{u}}^{n}-\boldsymbol{\vartheta}_{\textbf{u}}^{n-1}}{\tau}, \boldsymbol{\vartheta}_{\textbf{u}}^{n})+\mathcal{K}_{h}(\boldsymbol{\vartheta}_{\textbf{u}}^{n}, \boldsymbol{\vartheta}_{\textbf{u}}^{n}) \nonumber \\[1mm]
&=\big[\mathcal{M}_{2,h}(\delta_{t}\textbf{u}_{h}^{n}, \boldsymbol{\vartheta}_{\textbf{u}}^{n})+\mathcal{K}_{h}(\textbf{u}_{h}^{n}, \boldsymbol{\vartheta}_{\textbf{u}}^{n})\big]-\big[\mathcal{M}_{2,h}(\delta_{t}\textbf{u}_{I}^{n}, \boldsymbol{\vartheta}_{\textbf{u}}^{n})+\mathcal{K}_{h}(\textbf{u}_{I}^{n}, \boldsymbol{\vartheta}_{\textbf{u}}^{n})\big]\nonumber\\[1mm]
&=\mathcal{M}_{2,h}(\delta_{t}(\textbf{u}^{n}-\textbf{u}_{I}^{n}), \boldsymbol{\vartheta}_{\textbf{u}}^{n})+\big[ \mathcal{M}_{2}(\partial_{t} \textbf{u}^{n}, \boldsymbol{\vartheta}_{\textbf{u}}^{n})-\mathcal{M}_{2,h}(\delta_{t}\textbf{u}^{n}, \boldsymbol{\vartheta}_{\textbf{u}}^{n})\big]+\mathcal{K}_{h}(\textbf{u}^{n}-\textbf{u}_{I}^{n}, \boldsymbol{\vartheta}_{\textbf{u}}^{n})\nonumber\\[1mm]
&\quad+\big[ \mathcal{K}( \textbf{u}^{n}, \boldsymbol{\vartheta}_{\textbf{u}}^{n})-\mathcal{K}_{h}(\textbf{u}^{n}, \boldsymbol{\vartheta}_{\textbf{u}}^{n})\big] +\big[\mathcal{E}(\textbf{u}^{n};\textbf{u}^{n},\boldsymbol{\vartheta}_{\textbf{u}}^{n})-\mathcal{E}_{h}(\textbf{u}_{h}^{n};\textbf{u}_{h}^{n},\boldsymbol{\vartheta}_{\textbf{u}}^{n}) \big]\nonumber\\[1mm]
&\quad+\big[\left((c_{1}^{n}-c_{2}^{n})\nabla \phi^{n}, \boldsymbol{\vartheta}_{\textbf{u}}^{n} \right)-\left((c_{1,h}^{n}-c_{2,h}^{n})\nabla \phi_{h}^{n}, \boldsymbol{\vartheta}_{\textbf{u}}^{n} \right)_{h} \big]\nonumber\\[1mm]
&:=\rev{T_{1}+T_{2}+T_{3}+T_{4}+T_{5}+T_{6}.}
\end{align}

\medskip 
\noindent \textbf{Step 2: bounding the error terms.} 
For the terms  $T_{1}$ and $T_{3}$ we can apply the continuity of $\mathcal{M}_{2,h}(\cdot,\cdot)$ and $\mathcal{K}_{h}(\cdot,\cdot)$ (cf. \eqref{BUND2}) and the approximation properties of the interpolator $\textbf{u}_{I}$ (cf. \eqref{eqIn}) to arrive at}  
\begin{equation*}
\begin{aligned}
\left|T_{1}\right| &=\rev{\left|\mathcal{M}_{2,h}(\delta_{t}(\textbf{u}^{n}-\textbf{u}_{I}^{n}), \boldsymbol{\vartheta}_{\textbf{u}}^{n})\right| \leq\left| \mathcal{M}_{2,h}\left(\delta_{t} \textbf{u}^{n}-\partial_{t} \textbf{u}^{n}, \boldsymbol{\vartheta}_{\textbf{u}}^{n}\right)+\mathcal{M}_{2,h}\left(\partial_{t} \textbf{u}^{n}-\delta_{t} \textbf{u}_{I}^{n}, \boldsymbol{\vartheta}_{\textbf{u}}^{n}\right) \right|}\\
& \leq \tilde{\alpha}_{1}\left\{2 \tau^{1 / 2}\left\|\partial_{tt} \textbf{u}\right\|_{L^{2}\left( L^{2}\right)}+\tau^{-1 / 2} h^{k+1}\left\|\partial_{t} \textbf{u}\right\|_{L^{2}\left( H^{k+1}\right)}\right\}\left\|\boldsymbol{\vartheta}_{\textbf{u}}^{n}\right\|_{0},
\end{aligned}
\end{equation*}
and similarly
\begin{equation*}
\rev{\left|T_{3}\right| =\left|\mathcal{K}_{h}(\textbf{u}^{n}-\textbf{u}_{I}^{n}, \boldsymbol{\vartheta}_{\textbf{u}}^{n})\right|\leq Ch^{k+1}\Vert u^{n}\Vert_{L^{\infty}(H^{k+1})}\left|\boldsymbol{\vartheta}_{\textbf{u}}^{n}\right|_{1}.}
\end{equation*}

For $T_{2}$ and $T_{4}$ we first notice that, \rev{by adding and subtracting $\mathcal{M}_{2,h}(\partial_{t} \textbf{u}^{n}, \boldsymbol{\vartheta}_{\textbf{u}}^{n})$,  we can write} 
\begin{equation}\label{EST2}
\rev{  T_{2}
  = \mathcal{M}_{2}(\partial_{t} \textbf{u}^{n}, \boldsymbol{\vartheta}_{\textbf{u}}^{n})-\mathcal{M}_{2,h}(\partial_{t} \textbf{u}^{n}, \boldsymbol{\vartheta}_{\textbf{u}}^{n})  
+\mathcal{M}_{2,h}(\partial_{t} \textbf{u}^{n}-\delta_{t} \textbf{u}^{n}, \boldsymbol{\vartheta}_{\textbf{u}}^{n}).}
\end{equation}
To determine upper bounds for the terms in the right-hand side of \eqref{EST2}, we use Cauchy--Schwarz's inequality, Lemma \ref{l_12}, and the continuity   of the $L^{2}$-projector $\Pi_{k}^{0}$. This gives 
\begin{equation}\label{EST3}
\rev{\big| \mathcal{M}_{2}(\partial_{t}\textbf{u}^{n}, \boldsymbol{\vartheta}_{\textbf{u}}^{n})-\mathcal{M}_{2,h}(\partial_{t}\textbf{u}^{n}, \boldsymbol{\vartheta}_{\textbf{u}}^{n})\big|\leq \tilde{\alpha}_{1}\Vert \partial_{t}\textbf{u}^{n}-\boldsymbol{\Pi}_{k}^{0}\partial_{t}\textbf{u}^{n}\Vert_{0}\Vert\boldsymbol{\vartheta}_{\textbf{u}}^{n}\Vert_{0}\leq Ch^{k+1}\Vert \partial_{t}\textbf{u}^{n}\Vert_{k+1}\Vert\boldsymbol{\vartheta}_{\textbf{u}}^{n}\Vert_{0},}
\end{equation}
and
\[
\rev{\mathcal{M}_{2,h}(\partial_{t}\textbf{u}^{n}-\delta_{t} \textbf{u}^{n}, \boldsymbol{\vartheta}_{\textbf{u}}^{n})\leq \tilde{\alpha}_{1}\tau^{1/2}\Vert \partial_{tt}\textbf{u}\Vert_{L^{2}(L^{2})}\Vert\boldsymbol{\vartheta}_{\textbf{u}}^{n}\Vert_{0}.}
\]
After combining this estimate with \eqref{EST3} and \eqref{EST2}, we can conclude that 
\begin{equation*}
\rev{|T_{2}| \leq\left\{C h^{k+1}\left\|\partial_{t}\textbf{u}^{n}\right\|_{k+1}+\tau^{1 / 2}\left\|\partial_{tt}\textbf{u}\right\|_{L^{2}\left( L^{2}\right)}\right\}\left\|\boldsymbol{\vartheta}_{\textbf{u}}^{n}\right\|_{0},}
\end{equation*}
and similarly
\begin{equation*}
\rev{|T_{4}|=\big| \mathcal{K}(\textbf{u}^{n}, \boldsymbol{\vartheta}_{\textbf{u}}^{n})-\mathcal{K}_{h}(\textbf{u}^{n}, \boldsymbol{\vartheta}_{\textbf{u}}^{n})\big|\leq \tilde{\alpha}_{2}\Vert \textbf{u}^{n}-\boldsymbol{\Pi}_{k}^{\nabla}\textbf{u}^{n}\Vert_{0}\Vert\boldsymbol{\vartheta}_{\textbf{u}}^{n}\Vert_{0}\leq Ch^{k}\Vert \textbf{u}^{n}\Vert_{k+1}\vert\boldsymbol{\vartheta}_{\textbf{u}}^{n}\vert_{1}.}
\end{equation*}
The term $T_{5}$  can be rewritten by adding and subtracting some suitable terms 
\begin{align}\label{Geq_1}
  T_{5}&
= \rev{\big[\mathcal{E}(\textbf{u}^{n};\textbf{u}^{n},\boldsymbol{\vartheta}_{\textbf{u}}^{n})-\mathcal{E}_{h}(\textbf{u}^{n};\textbf{u}^{n},\boldsymbol{\vartheta}_{\textbf{u}}^{n})\big]+ \big[\mathcal{E}_{h}(\textbf{u}^{n};\textbf{u}^{n},\boldsymbol{\vartheta}_{\textbf{u}}^{n})-\mathcal{E}_{h}(\textbf{u}_{h}^{n};\textbf{u}_{h}^{n},\boldsymbol{\vartheta}_{\textbf{u}}^{n})\big]}:=T_{5}^{(1)}+T_{5}^{(2)}.
\end{align}
The first term above is estimated using Lemma \ref{l_C}
\begin{equation*}
|T_{5}^{(1)}|\leq Ch^{k}\left(\Vert \textbf{u}^{n}\Vert_{k}+\Vert \textbf{u}^{n}\Vert_{\textbf{X}}+\Vert \textbf{u}^{n}\Vert_{k+1} \right)\Vert \textbf{u}^{n}\Vert_{k+1}\Vert \boldsymbol{\vartheta}_{\textbf{u}}^{n}\Vert_{\textbf{X}},
\end{equation*}
while for  the second term 
\rev{we use the skew-symmetry of $\mathcal{E}$ and $\mathcal{E}_{h}$},  and we  recall that $\boldsymbol{\vartheta}_{\textbf{u}}^{n}=\textbf{u}_{h}^{n}-\textbf{u}_{I}^{n}$
\begin{align*}
 T_{5}^{(2)}&
=\mathcal{E}_{h}(\textbf{u}^{n};\textbf{u}^{n}-\textbf{u}_{h}^{n}+\boldsymbol{\vartheta}_{\textbf{u}}^{n}, \boldsymbol{\vartheta}_{\textbf{u}}^{n})
+\mathcal{E}_{h}(\textbf{u}^{n}-\textbf{u}_{h}+\boldsymbol{\vartheta}_{\textbf{u}}^{n};\textbf{u}_{h}^{n}, \boldsymbol{\vartheta}_{\textbf{u}}^{n})-\mathcal{E}_{h}(\boldsymbol{\vartheta}_{\textbf{u}}^{n};\textbf{u}_{h}^{n},\boldsymbol{\vartheta}_{\textbf{u}}^{n})\\ 
&\leq \gamma_{2}\bigg(\Vert \textbf{u}^{n}-\textbf{u}_{h}^{n}+\boldsymbol{\vartheta}_{\textbf{u}}^{n}\Vert_{\textbf{X}}(\Vert \textbf{u}^{n}\Vert_{\textbf{X}}+\Vert \textbf{u}_{h}^{n}\Vert_{\textbf{X}})+\Vert \textbf{u}_{h}^{n}\Vert_{\textbf{X}}\Vert \boldsymbol{\vartheta}_{\textbf{u}}^{n}\Vert_{\textbf{X}} \bigg)\Vert \boldsymbol{\vartheta}_{\textbf{u}}^{n}\Vert_{\textbf{X}}\\ 
&\leq \gamma_{2}\bigg(h^{k}\Vert \textbf{u}^{n}\Vert_{k+1}(\Vert \textbf{u}^{n}\Vert_{\textbf{X}}+\Vert \textbf{u}_{h}^{n}\Vert_{\textbf{X}})+\Vert \textbf{u}_{h}^{n}\Vert_{\textbf{X}}\Vert \boldsymbol{\vartheta}_{\textbf{u}}^{n}\Vert_{\textbf{X}} \bigg)\Vert \boldsymbol{\vartheta}_{\textbf{u}}^{n}\Vert_{\textbf{X}}.
\end{align*}
Substituting these expressions \rev{back into \eqref{Geq_1} and rearranging terms, we arrive at}  
\begin{equation*}
\rev{|T_{5}|\leq \bigg(Ch^{k}\left(\Vert \textbf{u}^{n}\Vert_{k}+\Vert \textbf{u}^{n}\Vert_{\textbf{X}}+\Vert \textbf{u}^{n}\Vert_{k+1}+\Vert \textbf{u}_{h}^{n}\Vert_{\textbf{X}} \right)\Vert \textbf{u}^{n}\Vert_{k+1}+\Vert \textbf{u}_{h}^{n}\Vert_{\textbf{X}}\Vert \boldsymbol{\vartheta}_{\textbf{u}}^{n}\Vert_{\textbf{X}}\bigg) \Vert \boldsymbol{\vartheta}_{\textbf{u}}^{n}\Vert_{\textbf{X}}.}
\end{equation*}
\rev{
Finally, the term $T_{6}$  is rewritten as
\begin{align}\label{Eq_W0}
T_{6}&=\big[ \left((c_{1}^{n}-c_{2}^{n})\nabla \phi^{n}, \boldsymbol{\vartheta}_{\textbf{u}}^{n} \right)-\left((c_{1}^{n}-c_{2}^{n})\nabla \phi^{n}, \boldsymbol{\vartheta}_{\textbf{u}}^{n} \right)_{h}\big]+\big[\left((c_{1}^{n}-c_{2}^{n})\nabla \phi^{n}, \boldsymbol{\vartheta}_{\textbf{u}}^{n} \right)_{h}-\left((c_{1,h}^{n}-c_{2,h}^{n})\nabla \phi_{h}^{n}, \boldsymbol{\vartheta}_{\textbf{u}}^{n} \right)_{h}\big]\nonumber\\[1mm]
&:=T_{6}^{(1)}+T_{6}^{(2)}.
\end{align}
Recalling the definition of the discrete inner product $(\cdot;\cdot,\cdot)_{h}$ we add  and subtract  suitable terms to have 
\begin{align}\label{Eq_W1}
T_{6}^{(1)}&=\sum_{E\in\mathcal{T}_{h}}\int_{E}\left( ((c_{1}^{n}-c_{2}^{n})\nabla \phi^{n})\cdot \boldsymbol{\vartheta}_{\textbf{u}}^{n} -(\Pi_{k}^{0}(c_{1}^{n}-c_{2}^{n})\boldsymbol{\Pi}_{k-1}^{0}\nabla \phi^{n})\cdot \boldsymbol{\Pi}_{k}^{0}\boldsymbol{\vartheta}_{\textbf{u}}^{n} \right)dE\nonumber\\[1mm]
&=\sum_{E\in\mathcal{T}_{h}}\int_{E}\bigg(((c_{1}^{n}-c_{2}^{n})\nabla \phi^{n})\cdot (I-\boldsymbol{\Pi}_{k}^{0})\boldsymbol{\vartheta}_{\textbf{u}}^{n}+(I-\Pi_{k}^{0,E})(c_{1}^{n}-c_{2}^{n})\nabla \phi^{n} \cdot\boldsymbol{\Pi}_{k}^{0}\boldsymbol{\vartheta}_{\textbf{u}}^{n}\nonumber\\
&\quad + \Pi_{k}^{0,E}(c_{1}^{n}-c_{2}^{n})(I-\boldsymbol{\Pi}_{k-1}^{0})\nabla \phi^{n} \cdot\boldsymbol{\Pi}_{k}^{0}\boldsymbol{\vartheta}_{\textbf{u}}^{n}\bigg)dE.
\end{align}
Then, the definition of  $\boldsymbol{\Pi}_{k}^{0,E}$, estimate \eqref{eqPi}, the H\"older inequality and Sobolev embedding $H^{k}\subset W_{4}^{k-1}$ lead to 
\begin{align*}
\int_{E}\left(((c_{1}^{n}-c_{2}^{n})\nabla \phi^{n})\cdot (I-\boldsymbol{\Pi}_{k}^{0})\boldsymbol{\vartheta}_{\textbf{u}}^{n}\right)dE&=\int_{E}\left((I-\boldsymbol{\Pi}_{k-1}^{0})((c_{1}^{n}-c_{2}^{n})\nabla \phi^{n})\cdot (I-\boldsymbol{\Pi}_{k}^{0})\boldsymbol{\vartheta}_{\textbf{u}}^{n}\right)dE\nonumber\\[1mm]
&\leq \Vert (I-\boldsymbol{\Pi}_{k-1}^{0})((c_{1}^{n}-c_{2}^{n})\nabla \phi^{n})\Vert_{0,E}\Vert (I-\boldsymbol{\Pi}_{k}^{0})\boldsymbol{\vartheta}_{\textbf{u}}^{n}\Vert_{0,E}\nonumber\\[1mm]
&\leq Ch_{E}^{k}\Vert (c_{1}^{n}-c_{2}^{n})\nabla \phi^{n}\Vert_{k-1,E}\Vert \boldsymbol{\vartheta}_{\textbf{u}}^{n}\Vert_{1,E}\nonumber\\[1mm]
&\leq Ch_{E}^{k}\Vert (c_{1}^{n}-c_{2}^{n})\Vert_{k-1,4,E}\Vert \nabla\phi^{n}\Vert_{k-1,4,E}\Vert \boldsymbol{\vartheta}_{\textbf{u}}^{n}\Vert_{1,E}\nonumber\\[1mm]
&\leq Ch_{E}^{k}\Vert (c_{1}^{n}-c_{2}^{n})\Vert_{1,E}\Vert \nabla\phi^{n}\Vert_{1,E}\Vert \boldsymbol{\vartheta}_{\textbf{u}}^{n}\Vert_{1,E}.
\end{align*}
Also,   the H\"older inequality, \rev{again property   \eqref{eqPi}, and  the continuity of $\boldsymbol{\Pi}_{k}^{0,E}$, give}
\begin{align*}
\hspace{-.5cm}\int_{E}\left((I-\Pi_{k}^{0,E})(c_{1}^{n}-c_{2}^{n})\nabla \phi^{n} \cdot\boldsymbol{\Pi}_{k}^{0}\boldsymbol{\vartheta}_{\textbf{u}}^{n}\right)dE &\leq \Vert (I-\Pi_{k}^{0,E})(c_{1}^{n}-c_{2}^{n})\Vert_{0,4,E}\Vert\nabla \phi^{n}\Vert_{0,E}\Vert\boldsymbol{\Pi}_{k}^{0}\boldsymbol{\vartheta}_{\textbf{u}}^{n}\Vert_{0,4,E}\nonumber\\[1mm]
&\leq \bigg(\Vert (c_{1}^{n}-c_{2}^{n})-(c_{1,\pi}^{n}-c_{2,\pi}^{n})\Vert_{0,4,E}\nonumber\\[1mm]
&+\Vert \Pi_{k}^{0,E}((c_{1}^{n}-c_{2}^{n})-(c_{1,\pi}^{n}-c_{2,\pi}^{n}))\Vert_{0,4,E} \bigg)\Vert\nabla \phi^{n}\Vert_{0,E}\Vert\boldsymbol{\vartheta}_{\textbf{u}}^{n}\Vert_{0,4,E}\nonumber\\[1mm]
&\leq Ch_{E}^{k}(\Vert c_{1}^{n}\Vert_{k+1,E}+\Vert c_{2}^{n}\Vert_{k+1,E})\Vert\phi^{n}\Vert_{1,E}\Vert\boldsymbol{\vartheta}_{\textbf{u}}^{n}\Vert_{1,E},
\end{align*}
and
\begin{align*}
\int_{E}\left( \Pi_{k}^{0,E}(c_{1}^{n}-c_{2}^{n})(I-\boldsymbol{\Pi}_{k-1}^{0})\nabla \phi^{n} \cdot\boldsymbol{\Pi}_{k}^{0}\boldsymbol{\vartheta}_{\textbf{u}}^{n}\right)dE &\leq \Vert \Pi_{k}^{0,E}(c_{1}^{n}-c_{2}^{n})\Vert_{0,4,E}\Vert (I-\boldsymbol{\Pi}_{k-1}^{0})\nabla \phi^{n}\Vert_{0,E}\Vert \boldsymbol{\Pi}_{k}^{0}\boldsymbol{\vartheta}_{\textbf{u}}^{n}\Vert_{0,4,E}\nonumber\\[1mm]
&\leq Ch_{E}^{k}(\Vert c_{1}^{n}-c_{2}^{n}\Vert_{1,E}\Vert\phi^{n} \Vert_{k+1}\Vert\boldsymbol{\vartheta}_{\textbf{u}}^{n}\Vert_{1,E}.
\end{align*}
Substituting this expression into \eqref{Eq_W1} and rearranging terms, \rev{leads to}  
\begin{equation}\label{Eq_E1}
|T_{6}^{(1)}|\leq Ch^{k}\bigg(\Vert c_{1}^{n}-c_{2}^{n}\Vert_{Z}(\Vert\phi^{n}\Vert_{2}+\Vert\phi^{n}\Vert_{k+1})+(\Vert c_{1}^{n}\Vert_{k+1}+\Vert c_{2}^{n}\Vert_{k+1})\Vert\phi^{n}\Vert_{Y} \bigg)\Vert\boldsymbol{\vartheta}_{\textbf{u}}^{n}\Vert_{\textbf{X}}.
\end{equation}
The \rev{second term in \eqref{Eq_W0} can be estimated by the H\"older inequality, Sobolev embedding $H^{k}\subset W_{4}^{k-1}$ and the continuity of $\boldsymbol{\Pi}_{k}^{0,E}$ with respect to the $L^{4}$-norm as
\begin{align}\label{Eq_E2}
T_{6}^{(2)}&=\left((c_{1}^{n}-c_{2}^{n})\nabla \phi^{n}-(c_{1,h}^{n}-c_{2,h}^{n})\nabla \phi_{h}^{n}, \boldsymbol{\vartheta}_{\textbf{u}}^{n} \right)_{h}\nonumber\\[1mm]
&=\left((c_{1}^{n}-c_{2}^{n})\nabla(\phi^{n}-\phi_{h}^{n}), \boldsymbol{\vartheta}_{\textbf{u}}^{n} \right)_{h}
+\left((c_{1}^{n}-c_{2}^{n})-(c_{1,h}^{n}-c_{2,h}^{n}))\nabla(\phi^{n}-\phi_{h}^{n}), \boldsymbol{\vartheta}_{\textbf{u}}^{n} \right)_{h}
+\left((c_{1}^{n}-c_{2}^{n})-(c_{1,h}^{n}-c_{2,h}^{n}))\nabla\phi^{n}, \boldsymbol{\vartheta}_{\textbf{u}}^{n} \right)_{h}\nonumber\\[1mm]
&\leq \bigg((\Vert c_{1}^{n}- c_{2}^{n}\Vert_{Z}+\Vert c_{1,h}^{n}-c_{2,h}^{n}\Vert_{Z})\Vert \phi^{n}-\phi_{h}^{n}\Vert_{Y}+(\Vert c_{1}^{n}-c_{1,h}^{n}\Vert_{0}+\Vert c_{2}^{n}-c_{2,h}^{n}\Vert_{0})\Vert\phi^{n}\Vert_{2} \bigg)\Vert  \boldsymbol{\vartheta}_{\textbf{u}}^{n}\Vert_{\textbf{X}}.
\end{align}
Finally, it suffices to substitute \eqref{Eq_E1} and \eqref{Eq_E2} back into \eqref{Eq_W0}, to arrive at 
\begin{align*}
T_{6}&\leq Ch^{k}\bigg(\Vert c_{1}^{n}-c_{2}^{n}\Vert_{Z}(\Vert\phi^{n}\Vert_{2}+\Vert\phi^{n}\Vert_{k+1})+(\Vert c_{1}^{n}\Vert_{s+1}+\Vert c_{2}^{n}\Vert_{k+1})\Vert\phi^{n}\Vert_{Z} \bigg)\Vert\boldsymbol{\vartheta}_{\textbf{u}}^{n}\Vert_{\textbf{X}}\nonumber\\[1mm]
&\quad+\bigg((\Vert c_{1}^{n}-c_{2}^{n}\Vert_{Z}+\Vert c_{1,h}^{n}-c_{2,h}^{n}\Vert_{Z})\Vert \phi^{n}-\phi_{h}^{n}\Vert_{Z} + (\Vert c_{1}^{n}-c_{1,h}^{n}\Vert_{0}+\Vert c_{2}^{n}-c_{2,h}^{n}\Vert_{0})\Vert\phi^{n}\Vert_{2} \bigg)\Vert  \boldsymbol{\vartheta}_{\textbf{u}}^{n}\Vert_{\textbf{X}}.
\end{align*}
}

\medskip 
\noindent\textbf{Step 3:   error estimate at the $n$-th time step.} Inserting the bounds on $T_{1}$-$T_{6}$ into \eqref{EST1b}, \rev{yields
\begin{align*}
&\mathcal{M}_{2,h}(\dfrac{\boldsymbol{\vartheta}_{\textbf{u}}^{n}-\boldsymbol{\vartheta}_{\textbf{u}}^{n-1}}{\tau}, \boldsymbol{\vartheta}_{\textbf{u}}^{n})+\mathcal{K}_{h}(\boldsymbol{\vartheta}_{\textbf{u}}^{n}, \boldsymbol{\vartheta}_{\textbf{u}}^{n})
\leq \big[\rev{\widehat{\varpi}_{1}^{n}}+\Vert\textbf{u}_{h}^{n}\Vert_{\textbf{X}}\Vert\boldsymbol{\vartheta}_{\textbf{u}}^{n}\Vert_{\textbf{X}}+(\Vert c_{1}^{n}-c_{1,h}^{n}\Vert_{0}\\[1mm]
&\qquad  +\Vert c_{2}^{n}-c_{2,h}^{n}\Vert_{0})\Vert\phi^{n}\Vert_{2}+\left(\rev{\widehat{\varpi}_{2}^{n}}+\widehat{\varpi}_{3}^{n} \right)\Vert \phi^{n}-\phi_{h}^{n}\Vert_{Y}\big]\left\|\boldsymbol{\vartheta}_{\textbf{u}}^{n}\right\|_{\textbf{X}},
\end{align*}
with positive scalars
\begin{equation}\label{EQLI}
\widehat{\varpi}_{1}^{n}\leq \overline{C}_{1}h^{k+1}+\overline{C}_{2}h^{k}+\tau^{1/2}O_{1}^{n}+\tau^{-1/2}h^{k+1}O_{2}^{n},\quad \widehat{\varpi}_{2}^{n}\leq \Vert c_{1}^{n}- c_{2}^{n}\Vert_{Z},\quad \widehat{\varpi}_{3}^{n}\leq \Vert c_{1,h}^{n}- c_{2,h}^{n}\Vert_{Z},
\end{equation}
and where
\[
\overline{C}_{1}\leq \left\|\partial_{t}\textbf{u}^{n}\right\|_{L^{\infty}(H^{k+1})}+\Vert \textbf{u}^{n}\Vert_{L^{\infty}(H^{k+1})},\quad O_{1}\leq \left\|\partial_{tt}\textbf{u}\right\|_{L^{2}\left( L^{2}\right)},\quad O_{2}\leq \left\|\partial_{t}\textbf{u}\right\|_{L^{2}\left( H^{k+1}\right)},
\]
\begin{align*}
\overline{C}_{2}&\leq \left(\Vert \textbf{u}^{n}\Vert_{k}+\Vert \textbf{u}^{n}\Vert_{\textbf{X}}+\Vert \textbf{u}^{n}\Vert_{k+1}+\Vert \textbf{u}_{h}^{n}\Vert_{\textbf{X}} +1\right)\Vert \textbf{u}^{n}\Vert_{k+1}\nonumber\\[1mm]
&+(\Vert c_{1}^{n}\Vert_{Z}+\Vert c_{2}^{n}\Vert_{Z})(\Vert\phi^{n}\Vert_{2}+\Vert\phi^{n}\Vert_{k+1})+(\Vert c_{1}^{n}\Vert_{k+1}+\Vert c_{2}^{n}\Vert_{k+1})\Vert\phi^{n}\Vert_{Y}.
\end{align*}
And an application of Eqs. \eqref{RR1} and \eqref{RR2}, yields
\begin{align*}
 \left( \widehat{\varpi}_{2}^{n}+\widehat{\varpi}_{3}^{n}\right)
\leq \left(\Vert c_{1}^{n}\Vert_{Z}+\Vert c_{2}^{n}\Vert_{Z}+\Vert c_{1,h}^{n}\Vert_{Z}+\Vert c_{2,h}^{n}\Vert_{Z}\right)\leq \dfrac{\widehat{\alpha}}{6\gamma_{4}}+\dfrac{\widehat{\beta}}{6\tilde{\gamma}_{4}}:=\overline{C}_{3}
\end{align*}
Also, it} is not difficult to verify that
\begin{equation*}
\mathcal{M}_{2,h}(\dfrac{\boldsymbol{\vartheta}_{\textbf{u}}^{n}-\boldsymbol{\vartheta}_{\textbf{u}}^{n-1}}{\tau}, \boldsymbol{\vartheta}_{\textbf{u}}^{n})\geq \dfrac{1}{2\tau}\left(\tilde{\beta}_{1}\Vert \boldsymbol{\vartheta}_{\textbf{u}}^{n}\Vert_{0}^{2}-\tilde{\alpha}_{1}\Vert \boldsymbol{\vartheta}_{\textbf{u}}^{n-1}\Vert_{0}^{2}\right), \qquad 
\mathcal{K}_{h}(\boldsymbol{\vartheta}_{\textbf{u}}^{n}, \boldsymbol{\vartheta}_{\textbf{u}}^{n})\geq \tilde{\beta}_{2} \Vert \boldsymbol{\vartheta}_{\textbf{u}}^{n}\Vert_{\textbf{X}}^{2}.
\end{equation*}
And employing the inequalities above, gives
\begin{align}\label{EQAK1}
&\dfrac{1}{2\tau}\left(\Vert \boldsymbol{\vartheta}_{\textbf{u}}^{n}\Vert_{0}^{2}-\Vert \boldsymbol{\vartheta}_{\textbf{u}}^{n-1}\Vert_{0}^{2}\right)+\tilde{\beta}_{2} \Vert\boldsymbol{\vartheta}_{\textbf{u}}^{n}\Vert_{\textbf{X}}^{2}
\leq \Vert\textbf{u}_{h}^{n}\Vert_{\textbf{X}}\Vert\boldsymbol{\vartheta}_{\textbf{u}}^{n}\Vert_{\textbf{X}}^{2}+\big[\widehat{\varpi}_{1}^{n}+(\Vert c_{1}^{n}-c_{1,h}^{n}\Vert_{0}+\Vert c_{2}^{n}-c_{2,h}^{n}\Vert_{0})\Vert\phi^{n}\Vert_{2}\big]^{2}\nonumber\\[1mm]
&\qquad \qquad +\overline{C}_{3}^{2}\Vert \phi^{n}-\phi_{h}^{n}\Vert_{Y}^{2}+\dfrac{\tilde{\beta}_{2}}{4}\left\|\boldsymbol{\vartheta}_{\textbf{u}}^{n}\right\|_{\textbf{X}}^{2}.
\end{align}
\rev{Next, using  \eqref{E_Stab} and invoking the smallness assumptions \eqref{AS_T2}, we arrive at
\begin{align*}
\Vert \boldsymbol{\vartheta}_{\textbf{u}}^{n}\Vert_{0}^{2}+2\tau\tilde{\beta}_{2} \Vert \boldsymbol{\vartheta}_{\textbf{u}}^{n}\Vert_{\textbf{X}}^{2}&\leq \Vert \boldsymbol{\vartheta}_{\textbf{u}}^{n-1}\Vert_{0}^{2}+2\tau\big[\widehat{\varpi}_{1}^{n}+(\Vert c_{1}^{n}-c_{1,h}^{n}\Vert_{0}+\Vert c_{2}^{n}-c_{2,h}^{n}\Vert_{0})\Vert\phi^{n}\Vert_{2}+\overline{C}_{3}\Vert \phi^{n}-\phi_{h}^{n}\Vert_{Y}\big]^{2}.
\end{align*}
Then we proceed to sum} up \rev{the above inequality over $n$,
$1\leq n\leq N$}, \rev{which gives 
\begin{align}\label{EstPr_L2}
\Vert \boldsymbol{\vartheta}_{\textbf{u}}^{n}\Vert_{0}^{2}+2\tau\tilde{\beta}_{2}\sum_{j=1}^{n} \Vert \boldsymbol{\vartheta}_{\textbf{u}}^{j}\Vert_{\textbf{X}}^{2} &\leq \Vert \boldsymbol{\vartheta}_{\textbf{u}}^{0}\Vert_{0}^{2}+\tau\sum_{j=1}^{n}\big[\widehat{\varpi}_{1}^{j}+(\Vert c_{1}^{j}-c_{1,h}^{j}\Vert_{0}+\Vert c_{2}^{j}-c_{2,h}^{j}\Vert_{0})\Vert\phi^{j}\Vert_{2}+\overline{C}_{3}\Vert \phi^{j}-\phi_{h}^{j}\Vert_{Y}\big]^{2}.
\end{align}
Using the fact that $\sum_{j=1}^{n}\tau\leq t_{F}$ along  with the definition of $\widehat{\varpi}_{1}^{n}$ in \eqref{EQLI}}, we obtain
\begin{align*}
\tau \sum_{j=1}^{n} (\widehat{\varpi}_{1}^{j})^{2}&\leq \sum_{j=1}^{n}\tau \left(\overline{C}_{1}h^{k+1}+\overline{C}_{2}h^{k}+\tau^{1/2}O_{1}^{n}+\tau^{-1/2}h^{k+1}O_{2}^{n} \right)^{2}\nonumber\\
&\leq \big[ h^{2(k+1)}\overline{C}_{1}^{2}+h^{2(k)}\overline{C}_{2}^{2}\big]\sum_{j=1}^{n}\tau+\tau^{2}\sum_{j=1}^{n}(O_{1}^{n})^{2}+h^{2(k+1)}\sum_{j=1}^{n}(O_{2}^{n})^{2}\nonumber\\
&\leq C(h^{2k}+\tau^{2}),
\end{align*}
}
which, together with \eqref{EstPr_L2}, completes the proof.
\end{proof}

\subsection{Error bounds: concentrations and electrostatic potential}
\rev{Let us now consider} the following problem:
\begin{subequations}\label{Dis_PNP}
\begin{align}
\mathcal{M}_{1,h}(\delta_{t}c_{i,h}^{n}, z_{i,h})+\mathcal{A}_{i,h}(c_{i,h}^{n}, z_{i,h})+e_{i}\mathcal{C}_{h}(c_{i,h}^{n};\phi_{h}^{n}, z_{i,h})-\mathcal{D}_{h}(\textbf{u}_{h}^{n}; c_{i,h}^{n}, z_{i,h})&=0,\\[1mm]
\mathcal{A}_{3,h}(\phi_{h}^{n}, z_{3,h})-\mathcal{M}_{1,h}(c_{1,h}^{n}, z_{3,h})+\mathcal{M}_{1,h}(c_{2,h}^{n}, z_{3,h})&=0,\label{Dis_PNP:b}
\end{align}
\end{subequations}
where $\textbf{u}_{h}^{n}\in \widetilde{\textbf{X}}_{h}$ is the solution from \eqref{ful_3} for $n=1,\cdots, N$. Next we  define   discrete projection operators \rev{that will be instrumental in deriving} error estimates for concentrations and electrostatic potential. 
\subsubsection{Electrostatic potential}
We now derive an upper bound for $\Vert \phi^{n}-\phi_{h}^{n}\Vert_{1}$ in terms of the concentration errors for
$n=1,\cdots,N$. For any $t\in [0,t_{F}]$, we define the energy projection \rev{$\mathcal{P}_{h}: H^{1}(\Omega)\cap H^{k+1}(\Omega)\rightarrow Y_{h}$} as the solution of
\begin{equation}\label{Proj_1}
\mathcal{A}_{3,h}(\mathcal{P}_{h}\phi(t), z_{3})=\mathcal{A}_{3}(\phi(t), z_{3}),\quad\quad\forall z_{3}\in  Y_{h}.
\end{equation}
Using the interpolation  property \eqref{eqIn_A}, we recall the following approximation properties of $\mathcal{P}_{h}$.
\begin{lemma}[\cite{Vacca15}]
\label{l_proj1}
Assume that $z\in H^{k+1}(\Omega)\cap H^{1}(\Omega)$. Then, there exists a unique $\mathcal{P}_{h}z\in Y_{h}$ solution of \eqref{Proj_1} satisfying
\begin{equation}\label{err_proj1}
\Vert z-\mathcal{P}_{h}z\Vert_{0}+h\vert z-\mathcal{P}_{h}z\vert_{1}\leq Ch^{k+1}\Vert z\Vert_{k+1},
\end{equation}
and
\begin{equation*}
\Vert z-\mathcal{P}_{h}z\Vert_{1,\infty}\leq Ch^{k}\Vert z\Vert_{k+1,\infty}.
\end{equation*}
\end{lemma} 
In the next result we state an optimal error estimate for $\mathcal{P}_{h}\phi-\phi_{h}$.
\begin{lemma}\label{l_esti_phi}
Let \rev{$\{c_{1},c_{2},\phi\}$, $\{c_{1,h}^{n},c_{2,h}^{n},\phi_{h}^{n}\}$} be solutions to \eqref{Eq_1}, \eqref{Dis_PNP}, respectively. Then for $n=1,\cdots, N$ we have
\begin{equation*}
\rev{\Vert \phi^{n}-\phi_{h}^{n}\Vert_{0}\leq C\left(\Vert c_{1}^{n}-c_{1,h}^{n}\Vert_{0} +\Vert c_{2}^{n}-c_{2,h}^{n}\Vert_{0}+h^{k+1}(\Vert c_{1}^{n}\Vert_{k+1}+\Vert c_{2}^{n}\Vert_{k+1})\right).}
\end{equation*}
\end{lemma}
\begin{proof}
\rev{Setting $\vartheta^{n}=\phi_{h}^{n}-\mathcal{P}_{h}\phi^{n}$ implies that  $\vartheta^{n}\in Y_{h}$. Using \eqref{Dis_PNP:b},  \eqref{Proj_1}, and choosing $z_{3}=\vartheta^{n}$,  we get}
\begin{align*}
\beta_{4}\Vert \vartheta^{n}\Vert_{Y}^{2}\leq \mathcal{A}_{3,h}(\vartheta^{n},\vartheta^{n})&=\mathcal{A}_{3,h}(\phi_{h}^{n},\vartheta^{n})-\mathcal{A}_{3,h}(\mathcal{P}_{h}\phi^{n},\vartheta^{n})\nonumber\\[1mm]
&=\mathcal{M}_{1,h}(c_{1,h}^{n}-c_{1}^{n}, \vartheta^{n})-\mathcal{M}_{1,h}(c_{2,h}^{n}-c_{2}^{n}, \vartheta^{n})\nonumber\\[1mm]
&+\big[\mathcal{M}_{1,h}(c_{1}^{n}, \vartheta^{n})-\mathcal{M}(c_{1}^{n}, \vartheta^{n})\big]-\big[\mathcal{M}_{1,h}(c_{2}^{n}, z_{3})-\mathcal{M}(c_{2}^{n}, \vartheta^{n})\big].
\end{align*}
The continuity of $\mathcal{M}_{1,h}(\cdot,\cdot)$ given in \eqref{BUND1}, and Lemma \ref{l_12}, confirm that 
\begin{equation*}
\beta_{4}\Vert \vartheta^{n}\Vert_{Y}^{2}\leq \bigg(\alpha_{1}(\Vert c_{1}^{n}-c_{1,h}^{n}\Vert_{0} +\Vert c_{2}^{n}-c_{2,h}^{n}\Vert_{0})+Ch^{k+1}(\Vert c_{1}^{n}\Vert_{k+1}+\Vert c_{2}^{n}\Vert_{k+1}) \bigg)\Vert \vartheta^{n}\Vert_{Y},
\end{equation*}
which, by Poincar\'e inequality, implies that 
\begin{equation}\label{eq1_err_phi}
\rev{\Vert \vartheta^{n}\Vert_{Y}\leq C_{\mathtt{p}}\beta_{4}^{-1}\bigg(\alpha_{1}(\Vert c_{1}^{n}-c_{1,h}^{n}\Vert_{0} +\Vert c_{2}^{n}-c_{2,h}^{n}\Vert_{0})+Ch^{k+1}(\Vert c_{1}^{n}\Vert_{k+1}+\Vert c_{2}^{n}\Vert_{k+1}) \bigg).}
\end{equation}
Next,  a duality argument for nonlinear elliptic equations \cite{Cangiani20I}, gives the following $L^{2}$-error estimate 
\begin{equation}\label{eq2_err_phi}
\rev{\Vert \vartheta^{n}\Vert_{0}\leq Ch\Vert \vartheta^{n}\Vert_{Y}+C_{\mathtt{p}}\beta_{4}^{-1}\bigg(\alpha_{1}(\Vert c_{1}^{n}-c_{1,h}^{n}\Vert_{0} +\Vert c_{2}^{n}-c_{2,h}^{n}\Vert_{0})+Ch^{k+1}(\Vert c_{1}^{n}\Vert_{k+1}+\Vert c_{2}^{n}\Vert_{k+1}) \bigg).}
\end{equation}
Finally, combining \eqref{eq1_err_phi}, \eqref{eq2_err_phi}, the triangle inequality, and estimate   \eqref{err_proj1}, the desired result follows.
\end{proof}

\subsubsection{Concentrations} 
The aim of this part is to attain an upper bound for $\Vert c_{i}^{n}-c_{i,h}^{n}\Vert_{0}$. For this purpose, for a fixed $\textbf{u}(t)\in \textbf{X}$, $\phi(t)\in Z$ and $t\in J$, we define a discrete projection operator  $\mathcal{P}_{i,h}: Z\rightarrow Z_{h}$, as follows 
\begin{equation}\label{defiPN_Ph}
\mathcal{L}_{i,h}(\textbf{u}(t),\phi(t);\mathcal{P}_{i,h}c_{i}, z_{i})=\mathcal{L}_{i}(\textbf{u}(t),\phi(t);c_{i}, z_{i}),\quad\quad \forall z_{i}\in Z_{h},
\end{equation}
where
\begin{subequations}
\begin{align}
\mathcal{L}_{i,h}(\textbf{u}(t),\phi;c_{i}, z_{i})&=\mathcal{A}_{i,h}(c_{i}, z_{i})+e_{i}\mathcal{C}_{h}(c_{i};\phi, z_{i})-\mathcal{D}_{h}(\textbf{u}(t); c_{i}, z_{i})+(c_{i}, z_{i})_{h},\label{Pdef_1}\\[1mm]
\mathcal{L}_{i}(\textbf{u}(t),\phi;c_{i}, z_{i})&=\mathcal{A}_{i}(c_{i}, z_{i})+e_{i}\mathcal{C}(c_{i};\phi, z_{i})-\mathcal{D}(\textbf{u}(t); c_{i}, z_{i})+(c_{i}, z_{i})_{0}.\label{Pdef_2}
\end{align}\end{subequations}
\begin{lemma}\label{l_well_Ph}
Assume that   $\textbf{u}\in [L^{\infty}(\Omega)]^{2}$ and $\phi\in W^{1,\infty}(\Omega)$ for all $t\in (0,t_{F}]$. Then, the operator $\mathcal{P}_{i,h}: Z\rightarrow Z_{h}$ in \eqref{defiPN_Ph} is well-defined.
\end{lemma}
\begin{proof} We proceed by the Lax--Milgram lemma and the proof 
is divided into two steps. The first step establishes  that the bilinear form   on the left-hand side of \eqref{defiPN_Ph} is continuous and coercive on $Z_{h}\times Z_{h} $, whereas the second step proves that the right-hand side functional is bounded over $Z_{h}$.
Continuity of $\mathcal{L}_{i}$ is achieved by the continuity of the forms
$(\cdot,\cdot)_{0}$ and $\mathcal{A}_{i}(\cdot,\cdot)$, and Poincar\'e inequality with
\begin{equation*}
\mathcal{C}(c_{i};\phi,z_{i})=(c_{i}\nabla\phi, \nabla z_{i})_{0}\leq \Vert\phi\Vert_{1,\infty}\vert c_{i}\vert_{1} \vert z_{i}\vert_{1},
\end{equation*}
and
\begin{align*}
\mathcal{D}(\textbf{u}(t); c_{i}, z_{i})&\leq \dfrac{1}{2}\Vert \textbf{u}\Vert_{0,4}\left(\Vert c_{i}\Vert_{0,4}\Vert z_{i}\Vert_{1}+\Vert c_{i}\Vert_{1}\Vert z_{i}\Vert_{0,4} \right)\leq C \Vert \textbf{u}\Vert_{\textbf{X}}\Vert c_{i}\Vert_{Z}\Vert z_{i}\Vert_{Z}.
\end{align*}
The continuity of $\mathcal{L}_{i,h}$ can be handled using Lemmas \ref{l_N} and \ref{l_Con_C}. In turn, for the coercivity of $\mathcal{L}_{i,h}$ we have
\begin{align*}
\mathcal{D}_{h}^{E}(\textbf{u}; c_{i}, c_{i})&=\dfrac{1}{2}\big[(\boldsymbol{\Pi}_{k}^{0,E}\textbf{u}\cdot \Pi_{k}^{0,E} c_{i},  \Pi_{k-1}^{0,E}\nabla c_{i})_{0}-(\boldsymbol{\Pi}_{k}^{0,E}\textbf{u}\cdot  \Pi_{k-1}^{0,E}\nabla c_{i},  \Pi_{k}^{0,E}c_{i})_{0}  \big] 
=0.
\end{align*}
Then, combining this result with \eqref{CORC1} completes the proof.
\end{proof}
Now we derive the error estimates of $c_{1}^{n}-\mathcal{P}_{1,h}c_{1}^{n}$ and $c_{2}^{n}-\mathcal{P}_{2,h}c_{2}^{n}$ in the $L^{2}$-norm.
\begin{lemma}\label{l_D}
Assume that $\{c_{1},c_{2},\phi\}$ is the solution of \eqref{Eq_1} satisfing the regularity
assumptions
\begin{align*}
C(\Vert\textbf{u}^{n}\Vert_{\textbf{X}},\Vert\phi^{n}\Vert_{1,\infty})\Vert c_{1}^{n}\Vert_{k+1}&+\Vert c_{1}^{n}\nabla \phi^{n}\Vert_{k}+\Vert c_{1}^{n}\Vert_{\infty}\Vert \phi^{n}\Vert_{k+1}+\Vert \phi^{n}\Vert_{1,\infty}\Vert c_{1}^{n}\Vert_{k}\\
&+\Vert \textbf{u}\Vert_{k+1}(\Vert c_{1}^{n}\Vert_{k+1}+\Vert c_{1}^{n}\Vert_{1})+\Vert c_{1}^{n}\Vert_{k+1}(\Vert \textbf{u}^{n}\Vert_{k}+\Vert \textbf{u}^{n}\Vert_{1}) \leq C,\\[1mm]
C(\Vert\textbf{u}^{n}\Vert_{\textbf{X}},\Vert\phi^{n}\Vert_{1,\infty})\Vert c_{2}^{n}\Vert_{k+1}&+\Vert c_{2}^{n}\nabla \phi^{n}\Vert_{k}+\Vert c_{2}^{n}\Vert_{\infty}\Vert \phi^{n}\Vert_{k+1}+\Vert \phi^{n}\Vert_{1,\infty}\Vert c_{2}^{n}\Vert_{k}\\
&+\Vert \textbf{u}^{n}\Vert_{k+1}(\Vert c_{2}^{n}\Vert_{k+1}+\Vert c_{2}^{n}\Vert_{1})+\Vert c_{2}^{n}\Vert_{k+1}(\Vert \textbf{u}^{n}\Vert_{k}+\Vert \textbf{u}^{n}\Vert_{1}) \leq C,
\end{align*}
and $\mathcal{P}_{i,h}$ is defined as in \eqref{defiPN_Ph}. Then for $n=1,\cdots, N$ and $i=1,2$, we have the following error estimates
\begin{align*}
\rev{\Vert c_{i}^{n}-\mathcal{P}_{i,h}c_{i}^{n}\Vert_{0}+h\Vert c_{i}^{n}-\mathcal{P}_{i,h}c_{i}^{n}\Vert_{Z}\leq Ch^{k+1}.}
\end{align*}
\end{lemma}
\begin{proof}
We first bound the term $c_{i}^{n}-\mathcal{P}_{i,h}c_{i}^{n}$ in the $H^{1}$-norm for any $n=1,\cdots,N$. To this end, for   $\{c_{1}^{n},c_{2}^{n}\}\in H^{k+1}(\Omega)\times H^{k+1}(\Omega)$ we recall the estimate of its interpolant $\{c_{1,I}^{n},c_{2,I}^{n}\}$   given in \eqref{eqIn_A}. Let $\theta_{i}^{n}:=\mathcal{P}_{i,h}c_{i}^{n}-c_{i,I}^{n}$ be elements of $Z_{h}$. Employing the discrete coercivity of $\mathcal{L}_{i,h}$ (cf. proof of Lemma \ref{l_well_Ph}) and Eq. \eqref{defiPN_Ph}, yields
\rev{
\begin{align}\label{eqK1}
\hat{C}\vert \theta_{i}^{n}\vert_{1}^{2}&\leq \mathcal{L}_{i,h}(\textbf{u}^{n},\phi^{n};\theta_{i}^{n}, \theta_{i}^{n})=\mathcal{L}_{i,h}(\textbf{u}^{n},\phi^{n};\mathcal{P}_{h}c_{i}^{n}, \theta_{i}^{n})-\mathcal{L}_{i,h}(\textbf{u}^{n},\phi^{n};c_{i,I}^{n}, \theta_{i}^{n})\nonumber\\
&=\big[\mathcal{L}_{i}(\textbf{u}^{n},\phi^{n};c_{i}^{n}, \theta_{i}^{n})-\mathcal{L}_{i,h}(\textbf{u}^{n},\phi^{n};c_{i}^{n}, \theta_{i}^{n}) \big]+\mathcal{L}_{i,h}(\textbf{u}^{n},\phi^{n};c_{i}^{n}-c_{i,I}^{n}, \theta_{i}^{n})\nonumber\\
&:=L_{1}+L_{2}.
\end{align}
}
Using the definitions of $\mathcal{L}_{i}$ and $\mathcal{L}_{i,h}$ given in \eqref{Pdef_1} and \eqref{Pdef_2}, respectively, splits the term $L_{1}$ as follows:
\begin{align*}
L_{1}&=\big[\mathcal{A}_{i}(c_{i}^{n}, \theta_{i}^{n})-\mathcal{A}_{i,h}(c_{i}^{n}, \theta_{i}^{n}) \big]
+e_{i}\big[\mathcal{C}(c_{i}^{n};\phi^{n}, \theta_{i}^{n})-\mathcal{C}_{h}(c_{i}^{n};\phi, \theta_{i}^{n}) \big]\nonumber\\ 
&\quad +\big[\mathcal{D}(\textbf{u}^{n}; c_{i}^{n}, \theta_{i}^{n}) -\mathcal{D}_{h}(\textbf{u}^{n}; c_{i}^{n}, \theta_{i}^{n})\big]
+\big[ (c_{i}^{n}, \theta_{i}^{n})_{0}-(c_{i}^{n}, \theta_{i}^{n})_{h}\big]\nonumber\\ 
&:=L_{1}^{(1)}+L_{1}^{(2)}+L_{1}^{(3)}+L_{1}^{(4)}.
\end{align*}
Next, we will bound each of the terms $L_{1}^{(j)}$, with $j=1,2,3,4$ in the above decomposition. This is achieved by Lemmas \ref{l_12}, \ref{l_ComCPNP}, \ref{l_ComD} and \ref{l_12}, respectively, as
\begin{gather*}
L_{1}^{(1)}\leq Ch^{k}\Vert c_{i}^{n}\Vert_{k+1}\vert \theta_{i}^{n}\vert_{1},\quad L_{1}^{(2)}\leq C\bigg(h^{k}\Vert c_{i}^{n}\nabla \phi^{n}\Vert_{k}+\Vert c_{i}^{n}\Vert_{\infty}h^{k}\Vert \phi^{n}\Vert_{k+1}+\Vert \phi^{n}\Vert_{1,\infty}h^{k}\Vert c_{i}^{n}\Vert_{k} \bigg)\vert \theta_{i}^{n}\vert_{1},\\
L_{1}^{(3)} \leq Ch^{k}\left( \Vert \textbf{u}^{n}\Vert_{k+1}(\Vert c_{i}^{n}\Vert_{k+1}+\Vert c_{i}^{n}\Vert_{1})+\Vert c_{i}^{n}\Vert_{k+1}(\Vert \textbf{u}^{n}\Vert_{k}+\Vert \textbf{u}^{n}\Vert_{1}) \right)\Vert \theta_{i}^{n}\Vert_{1},\\
L_{1}^{(4)}\leq Ch^{k+1}\Vert c_{i}^{n}\Vert_{k+1}\Vert \theta_{i}^{n}\Vert_{0}.
\end{gather*}
Next, to estimate $L_{2}$, we apply the continuity of $\mathcal{L}_{i,h}$ (cf. Lemma \ref{l_well_Ph}) and the interpolation error estimate \eqref{eqIn}, to find that
\begin{equation*}
L_{2}=\mathcal{L}_{i,h}(\textbf{u}^{n},\phi^{n};c_{i}^{n}-c_{i,I}^{n}, \theta_{i}^{n})\leq C(\Vert\textbf{u}^{n}\Vert_{\textbf{X}},\Vert\phi^{n}\Vert_{1,\infty})h^{k}\Vert c_{i}^{n}\Vert_{k+1}\vert \theta_{i}^{n}\vert_{1}.
\end{equation*}
Thus, the  $H^{1}$-seminorm estimate is derived by inserting all bounds $L_{1}^{(j)}$ for $j=1,\cdots,4$ into $L_{1}$ and then substituting the obtained estimates for $L_{1}$ and $L_{2}$ in \eqref{eqK1}. Note that an estimate  in the $L^2$-norm is obtained by combining the  arguments in above with a standard duality approach. That is omitted here.
\end{proof}

Finally, we state an error estimate for $c_{1}^{n}-c_{1,h}^{n}$, $c_{2}^{n}-c_{2,h}^{n}$ and $\phi^{n}-\phi_{h}^{n}$ in the $L^{2}$-norm valid for the scheme \eqref{ful_3}.
\begin{theorem}\label{T3}
Let the assumption of Theorem \ref{T_1} be satisfied.
Also, assume that $\{c_{1}^{n},c_{2}^{n},\phi^{n}\}$ be the solution of \eqref{Eq_1} satisfying the regularity assumptions presented in Lemma \ref{l_D} and $\{c_{1,h}^{n},c_{2,h}^{n},\phi_{h}^{n}\}$ be the solution of \eqref{Dis_PNP}. 
Then, the following error estimation
holds for $n=1,\cdots, N$,
\begin{equation*}
\Vert c_{1}^{n}-c_{1,h}^{n}\Vert_{0}+\Vert c_{2}^{n}-c_{2,h}^{n}\Vert_{0}+\tau \sum_{j=1}^{n}\left(\vert c_{1}^{n}-c_{1,h}^{n}\vert_{1}+\vert c_{2}^{n}-c_{2,h}^{n}\vert_{1} \right)\leq C(\tau + h^{k}).
\end{equation*}
\end{theorem}
\begin{proof}
We divide the proof into three steps.

\medskip 
\noindent\textbf{Step 1:  discrete evolution equation for the error.} First, we split the
concentration errors 
as follows
\begin{equation*}
c_{i}^{n}-c_{i,h}^{n}=c_{i}^{n}-\mathcal{P}_{i,h}c_{i}^{n}+\mathcal{P}_{i,h}c_{i}^{n}-c_{i,h}^{n}:=\vartheta_{i}^{n}+\rho_{i}^{n},
\end{equation*}
where $\rho_{i}^{n}$ is estimated in Lemma \ref{l_D}. Now we estimate $\vartheta_{i}^{n}$. An application of Eqs. \eqref{eq6} and \eqref{Dis_PNP} with $z_{i}=\vartheta_{i}^{n}$ and the definition of the projector
$\mathcal{P}_{i,h}$ given in \eqref{defiPN_Ph} imply
\begin{align}\label{eq_D1}
&\mathcal{M}_{1,h}(\dfrac{\vartheta_{i}^{n}-\vartheta_{i}^{n}}{\tau}, \vartheta_{i}^{n})+\mathcal{A}_{i,h}(\vartheta_{i}^{n},\vartheta_{i}^{n})\nonumber \\
&=\big[\mathcal{M}_{1}(\partial_{t} c_{i}^{n}, \vartheta_{i}^{n}) -\mathcal{M}_{1,h}(\delta_{t}\mathcal{P}_{i,h}c_{i}^{n}, \vartheta_{i}^{n})\big] +e_{i}\big[\mathcal{C}_{h}(\mathcal{P}_{i,h}c_{i}^{n};\phi^{n}, \vartheta_{i}^{n})-\mathcal{C}_{h}(c_{i,h}^{n};\phi_{h}^{n}, \vartheta_{i}^{n}) \big]\nonumber\\ 
&+\big[ \mathcal{D}_{h}(\textbf{u}^{n};\mathcal{P}_{i,h}c_{i}^{n}, \vartheta_{i}^{n})-\mathcal{D}_{h}(\textbf{u}_{h}^{n};c_{i,h}^{n}, \vartheta_{i}^{n})\big]\nonumber +\big[(\mathcal{P}_{i,h}c_{i}^{n}, \vartheta_{i}^{n})_{h}-(c_{i}^{n}, \vartheta_{i}^{n})_{0} \big]\nonumber\\ 
&:=R_{1,i}+R_{2,i}+R_{3,i}+R_{4,i},
\end{align}
\rev{and owing} to the coercivity of $\mathcal{A}_{i,h}$, we have 
\begin{equation*}
\mathcal{A}_{i,h}(\vartheta_{i}^{n},\vartheta_{i}^{n})\geq \beta_{i+1}\vert \vartheta_{i}^{n}\vert_{1}^{2}.
\end{equation*}
\textbf{Step 2: bounding the error terms $R_{1,i}$-$R_{4,i}$.}  For the term $R_{1,i}$ we first notice that by adding zero in the form $\pm\mathcal{M}_{1,h}(\partial_{t} c_{i}^{n}, \vartheta_{i}^{n})$, we can obtain
\[
R_{1,i}= \mathcal{M}_{1}(\partial_{t} c_{i}^{n}, \vartheta_{i}^{n}) -\mathcal{M}_{1,h}(\delta_{t}\mathcal{P}_{i,h}c_{i}^{n}, \vartheta_{i}^{n})=\big[\mathcal{M}_{1}(\partial_{t} c_{i}^{n}, \vartheta_{i}^{n})-\mathcal{M}_{1,h}(\partial_{t} c_{i}^{n}, \vartheta_{i}^{n}) \big]
+\mathcal{M}_{1,h}(\partial_{t} c_{i}^{n}-\delta_{t}\mathcal{P}_{i,h}c_{i}^{n}, \vartheta_{i}^{n}).
\]
To determine upper bounds of the right-hand side terms above, we use Cauchy--Schwarz's inequality, Lemma \ref{l_12}, and the continuity  of the $L^{2}$-projector $\Pi_{k}^{0}$. This gives 
\begin{align*}
|R_{1,i}| & \leq \left( Ch^{k+1}\Vert \partial_{t} c_{i}^{n}\Vert_{k+1}+\alpha_{1}\tau^{1/2}\Vert \partial_{tt} c_{i}\Vert_{L^{2}(L^{2})}\right)\Vert \vartheta_{i}^{n}\Vert_{0}.
\end{align*}
For the term $R_{2,i}$, note that from the definition of $\mathcal{C}_{h}(\cdot;\cdot,\cdot)$ in \eqref{defiDisCPNP}, it holds
\begin{align*}
R_{2,i}&=\mathcal{C}_{h}(\mathcal{P}_{i,h}c_{i}^{n};\phi^{n}, \vartheta_{i}^{n})-\mathcal{C}_{h}(c_{i,h}^{n};\phi_{h}^{n}, \vartheta_{i}^{n}) =(\mathcal{P}_{i,h}c_{i}^{n}\nabla \phi^{n}, \nabla \vartheta_{i}^{n})_{h}-(c_{i,h}^{n}\nabla\phi_{h}^{n},  \nabla \vartheta_{i}^{n})_{h}.
\end{align*}
Note also that, after adding and subtracting some suitable terms, we can rewrite the above expression as 
\begin{align}\label{R2}
R_{2,i}&=\left((\mathcal{P}_{i,h}c_{i}^{n}-c_{i,h}^{n})\nabla\phi^{n},  \nabla \vartheta_{i}^{n}\right)_{h}
+\left( c_{i,h}^{n}\nabla (\phi^{n}-\phi_{h}^{n}), \nabla \vartheta_{i}^{n}\right)_{h}\nonumber\\
&:=R_{2,i}^{(1)}+R_{2,i}^{(2)}.
\end{align}
For $R_{2,i}^{(1)}$, using the H\"older inequality and the continuity of $\Pi_{k-1}^{0}$ and $\boldsymbol{\Pi}_{k-1}^{0,E}$ we can  write 
\begin{align}\label{R2_1}
R_{2,i}^{(1)}&=\big| \left((\mathcal{P}_{i,h}c_{i}^{n}-c_{i,h}^{n})\nabla\phi^{n},  \nabla \vartheta_{i}^{n}\right)_{h}\big|\leq \Vert \boldsymbol{\Pi}_{k-1}^{0}\nabla\phi^{n}\Vert_{\infty} \Vert\Pi_{k-1}^{0}\vartheta_{i}^{n}\Vert_{0}\Vert \boldsymbol{\Pi}_{k-1}^{0}\nabla\vartheta_{i}^{n}\Vert_{0}\nonumber\\
&\leq 
\Vert \phi^{n}\Vert_{1,\infty} \Vert\vartheta_{i}^{n}\Vert_{0}\vert\vartheta_{i}^{n}\vert_{1}.
\end{align}
Regarding the terms $R_{2,i}^{(2)}$, we have 
\begin{align}\label{R2_3}
R_{2,i}^{(2)}&= \bigg| \left( c_{i,h}^{n}\nabla (\phi^{n}-\phi_{h}^{n}), \nabla \vartheta_{i}^{n}\right)_{h}\bigg|\leq \Vert c_{i,h}^{n}\Vert_{\infty}\left( Ch^{k}+\Vert c_{1}^{n}-c_{1,h}^{n}\Vert_{0} +\Vert c_{2}^{n}-c_{2,h}^{n}\Vert_{0} \right)\vert \vartheta_{i}^{n}\vert_{1}.
\end{align}
Substituting Eqs. \eqref{R2_1}, \eqref{R2_3}, into \eqref{R2}, and using inequalities given in \eqref{RR1} and \eqref{RR2},  yield 
\begin{align*}
|R_{2,i}|&\leq \big[\Vert \phi^{n}\Vert_{1,\infty} \Vert\vartheta_{i}^{n}\Vert_{0}+\Vert c_{i,h}^{n}\Vert_{\infty}\left( Ch^{k}+\Vert c_{1}^{n}-c_{1,h}^{n}\Vert_{0} +\Vert c_{2}^{n}-c_{2,h}^{n}\Vert_{0} \right) \big]\vert \vartheta_{i}^{n}\vert_{1}\nonumber\\[1mm]
&\left(\frac{\widehat{\beta}\tilde{\alpha}_{1}}{6\tilde{\gamma}_{3}\tilde{\beta}_{4}}c_{\mathtt{p}} \Vert\vartheta_{i}^{n}\Vert_{0}+\dfrac{\widehat{\alpha}}{6\gamma_{3}}\left( Ch^{k}+\Vert \vartheta_{1}^{n}\Vert_{0} +\Vert \vartheta_{2}^{n}\Vert_{0} \right) \right)\vert \vartheta_{i}^{n}\vert_{1}
\end{align*} 
For the term $R_{3,i}$, we note that the definition of $\mathcal{D}_{h}(\cdot;\cdot,\cdot)$ implies 
\begin{align*}
R_{3,i}&=\mathcal{D}_{h}(\textbf{u}^{n};\mathcal{P}_{i,h}c_{i}^{n}, \vartheta_{i}^{n})-\mathcal{D}_{h}(\textbf{u}_{h}^{n};c_{i,h}^{n}, \vartheta_{i}^{n})\nonumber\\
&=\dfrac{1}{2}\big[\left(\textbf{u}^{n}\mathcal{P}_{i,h}c_{i}^{n},\nabla  \vartheta_{i}^{n}\right)_{h}-\left(\textbf{u}_{h}^{n}c_{i,h}^{n},\nabla  \vartheta_{i}^{n}\right)_{h} \big]-\dfrac{1}{2}\big[\left( \textbf{u}^{n}\cdot \nabla \mathcal{P}_{i,h}c_{i}^{n}, \vartheta_{i}^{n}\right)_{h} -\left( \textbf{u}_{h}^{n}\cdot \nabla c_{i,h}^{n}, \vartheta_{i}^{n}\right)_{h} \big].
\end{align*}
We note that the above equation, after adding zero as 
\begin{align*}
0&=\left(\textbf{u}_{h}^{n}\cdot \nabla \vartheta_{i}^{n}, \vartheta_{i}^{n} \right)_{h}-\left(\textbf{u}_{h}^{n}\cdot \nabla \vartheta_{i}^{n}, \vartheta_{i}^{n} \right)_{h}\\
&=\left(\textbf{u}_{h}^{n}\cdot \nabla \vartheta_{i}^{n}, c_{i,h}^{n} \right)_{h}-\left(\textbf{u}_{h}^{n}\cdot \nabla \vartheta_{i}^{n}, \mathcal{P}_{i,h}c_{i}^{n} \right)_{h}
-\left(\textbf{u}_{h}^{n}\nabla c_{i,h}^{n},  \vartheta_{i}^{n}\right)_{h}+\left(\textbf{u}_{h}^{n}\nabla \mathcal{P}_{i,h}c_{i}^{n},  \vartheta_{i}^{n}\right)_{h},
\end{align*}
can be bounded as follows
\begin{equation*}
|R_{3,i}|\leq \dfrac{1}{2}\big[ \left((\textbf{u}^{n}-\textbf{u}_{h}^{n})\mathcal{P}_{i,h}c_{i}^{n}, \nabla  \vartheta_{i}^{n}\right)_{h}- \left((\textbf{u}^{n}-\textbf{u}_{h}^{n})\cdot \nabla \mathcal{P}_{i,h}c_{i}^{n}, \vartheta_{i}^{n}\right)_{h} \big]:=R_{3,i}^{(1)}+R_{3,i}^{(2)}.
\end{equation*}
For $R_{3,i}^{(2)}$, applying the H\"older inequality and the continuity of the projectors $\boldsymbol{\Pi}_{k}^{0}$ with respect to the $L^{2}$ and $L^{4}$-norms we estimate
\begin{align*}
|R_{3,i}^{(2)}|= \bigg| \left((\textbf{u}^{n}-\textbf{u}_{h}^{n})\cdot \nabla \mathcal{P}_{i,h}c_{i}^{n}, \vartheta_{i}^{n}\right)_{h}\bigg| &\leq \Vert \boldsymbol{\Pi}_{k}^{0}(\textbf{u}^{n}-\textbf{u}_{h}^{n})\Vert_{0,4} \Vert\boldsymbol{\Pi}_{k-1}^{0}\nabla \mathcal{P}_{i,h}c_{i}^{n}\Vert_{0} \Vert \Pi_{k}^{0}\vartheta_{i}^{n}\Vert_{0,4}\nonumber\\
&\leq \Vert \textbf{u}^{n}-\textbf{u}_{h}^{n}\Vert_{0,4} \Vert \nabla \mathcal{P}_{h}c_{i}^{n}\Vert_{0} \Vert \vartheta_{i}^{n}\Vert_{0,4}.
\end{align*}
Using the triangle inequality and Lemma \ref{l_D}, we end up with the following upper bound for the second term on the right-hand side of the above inequality 
\begin{align*}
\Vert \nabla \mathcal{P}_{i,h}c_{i}^{n}\Vert_{0}\leq \Vert \nabla c_{i}^{n}\Vert_{0}+\Vert \nabla (\mathcal{P}_{i,h}c_{i}^{n}-c_{i}^{n})\Vert_{0}\leq \Vert c_{i}^{n}\Vert_{1},
\end{align*} 
which, together with the Sobolev embedding $H^{1}\subset L^{4}$, in turn implies
\begin{equation*}
|R_{3,i}^{(2)}|\leq \Vert \textbf{u}^{n}-\textbf{u}_{h}^{n}\Vert_{\textbf{X}} \Vert c_{i}^{n}\Vert_{1} \Vert \vartheta_{i}^{n}\Vert_{1}.
\end{equation*}
Bounding the term $R_{3,i}^{(1)}$ analogously to $R_{3,i}^{(2)}$, we can confirm that 
\begin{equation*}
R_{3,i}^{(1)}\leq \Vert \textbf{u}^{n}-\textbf{u}_{h}^{n}\Vert_{\textbf{X}} \Vert c_{i}^{n}\Vert_{1} \Vert \vartheta_{i}^{n}\Vert_{1}.
\end{equation*}
Thus using \eqref{RR2} we arrive at the bounds 
\begin{align*}
R_{3,i}&\leq \dfrac{\widehat{\beta}}{6\tilde{\gamma}_{4}} \Vert \textbf{u}^{n}-\textbf{u}_{h}^{n}\Vert_{\textbf{X}} \Vert \vartheta_{i}^{n}\Vert_{1}.
\end{align*}

\medskip 
\noindent\textbf{Step 3:  error estimate at  a generic $n$-th time step.}
We now insert the bounds on $R_{1,i}-R_{3,i}$ in \eqref{eq_D1}, yielding 
\begin{align}\label{St3_1}
\dfrac{1}{2\tau}\left(\Vert \vartheta_{i}^{n}\Vert_{0}^{2}-\Vert \vartheta_{i}^{n-1}\Vert_{0}^{2} \right)+\beta_{i+1}\vert \vartheta_{i}^{n}\vert_{1}^{2}&\leq \varpi_{1,i}\Vert \vartheta_{i}^{n}\Vert_{0}+\big[C_{1}\Vert \textbf{u}^{n}-\textbf{u}_{h}^{n}\Vert_{\textbf{X}}+C_{2}( \Vert \vartheta_{1}^{n}\Vert_{0}+\Vert \vartheta_{2}^{n}\Vert_{0})\big]\vert \vartheta_{i}^{n}\vert_{1}\nonumber\\
&\leq \dfrac{1}{2}\big[\varpi_{1,i}^{2}+\Vert \vartheta_{i}^{n}\Vert_{0}^{2}\big]+\epsilon\vert \vartheta_{i}^{n}\vert_{1}^{2}+\big[ \Vert \vartheta_{1}^{n}\Vert_{0}+\Vert \vartheta_{2}^{n}\Vert_{0}\big]^{2},
\end{align}
with positive scalars
\begin{align*}
\varpi_{1,i}&\leq Ch^{k+1}\Vert \partial_{t} c_{i}^{n}\Vert_{k+1}+\alpha_{1}\tau^{1/2}\Vert \partial_{tt} c_{i}\Vert_{L^{2}(L^{2})},\quad\quad\quad C_{1}\leq \dfrac{\widehat{\beta}}{6\tilde{\gamma}_{4}},\quad C_{2}\leq \max\{\frac{\widehat{\beta}\tilde{\alpha}_{1}}{6\tilde{\gamma}_{3}\tilde{\beta}_{4}}c_{\mathtt{p}},\dfrac{\widehat{\alpha}}{6\gamma_{3}}\}
\end{align*}
\rev{Summing on $n$ on both sides of \eqref{St3_1},} where $0\leq n\leq N$, allows us to obtain 
\rev{
\begin{align*}
\dfrac{1}{2\tau}\left(\Vert \vartheta_{i}^{n}\Vert_{0}^{2}-\Vert \vartheta_{i}^{0}\Vert_{0}^{2} \right)+\sum_{j=0}^{n}\vert \vartheta_{i}^{j}\vert_{1}^{2}&\leq \sum_{j=0}^{n}\varpi_{1,i}^{2}+\max\{C_{1},C_{2}\}\sum_{j=0}^{n}\left( \Vert \vartheta_{1}^{j}\Vert_{0}^{2}+\Vert \vartheta_{2}^{j}\Vert_{0}^{2}+\Vert \textbf{u}^{j}-\textbf{u}_{h}^{j}\Vert_{\textbf{X}}^{2} \right).
\end{align*}
}
Summing up these inequalities and employing Theorem \ref{T2} and Gronwall's inequality, it finally gives 
\begin{equation*}
\Vert \vartheta_{1}^{j}\Vert_{0}^{2}+\Vert \vartheta_{2}^{j}\Vert_{0}^{2}+\tau\sum_{n=0}^{j}\left(\vert \vartheta_{1}^{n}\vert_{1}^{2}+\vert \vartheta_{2}^{n}\vert_{1}^{2}\right)\leq \tau\sum_{n=0}^{j}\left(\varpi_{1,1}^{2}+\varpi_{1,2}^{2} \right).
\end{equation*}
\rev{The sought result follows from a similar procedure as in Theorem \ref{T2} and employing Lemma \ref{l_D}.} 
\end{proof}

\section{Numerical Results}\label{s7}
In this section, we provide numerical experiments to show the performance of the proposed VEM for coupled PNP/NS equations. In all examples, we use the virtual spaces $(\textbf{Z}_{h},Y_{h})$ for concentrations and electrostatic potential and the pair  ($\textbf{X}_{h}$, $Q_{h}$) for velocity and pressure, specified by the polynomial degree $k=2$, unless otherwise stated. \rev{The nonlinear fully-discrete system is linearized using a Picard algorithm and the fixed-point iterations are terminated when the $\ell^2$-norm of the global incremental discrete solutions drop below a fixed tolerance of 1e-08.} 

\subsection{Example 1: Accuracy assessment}
\rev{First we apply the fully discrete VEM to validate all theoretical convergence results shown in Theorems \ref{T2} and \ref{T3}. For this we consider the following closed-form} solutions to the coupled PNP/NS problem 
\begin{equation}\label{eq:manuf}
\left\{\begin{array}{l}
c_{1}(x,y,t) = \sin(2\pi x)\sin(2\pi y)\sin(t),\qquad 
c_{2}(x,y,t) = \sin(3\pi x)\sin(3\pi y)\sin(2t),\\ 
\phi(x,y,t) = \sin(\pi x)\sin(\pi y)(1-\exp(-t)),\\[.14ex] 
\textbf{u}(x,y,t) = \begin{pmatrix}  -0.5\exp(t)\cos(x)^{2}\cos(y)\sin(y)\\0.5\exp(t)\cos(y)^{2}\cos(x)\sin(x)  \end{pmatrix},\quad 
p(x,y,t) = \exp(t)(\sin(x)-\sin(y)),
\end{array}\right.
\end{equation}
 defined over the computational domain $\Omega =(0,1)^{2}$ and the time interval $[0,0.5]$. The exact velocity is divergence-free and the problem is modified including non-homogeneous forcing and source terms on the momentum and concentration equations constructed using the manufactured solutions \eqref{eq:manuf}. The model parameters  are taken as $\kappa_{1},\kappa_{2},\epsilon = 1$.
 \begin{figure}[t!]
\begin{center}	
\includegraphics[height=0.45\textwidth]{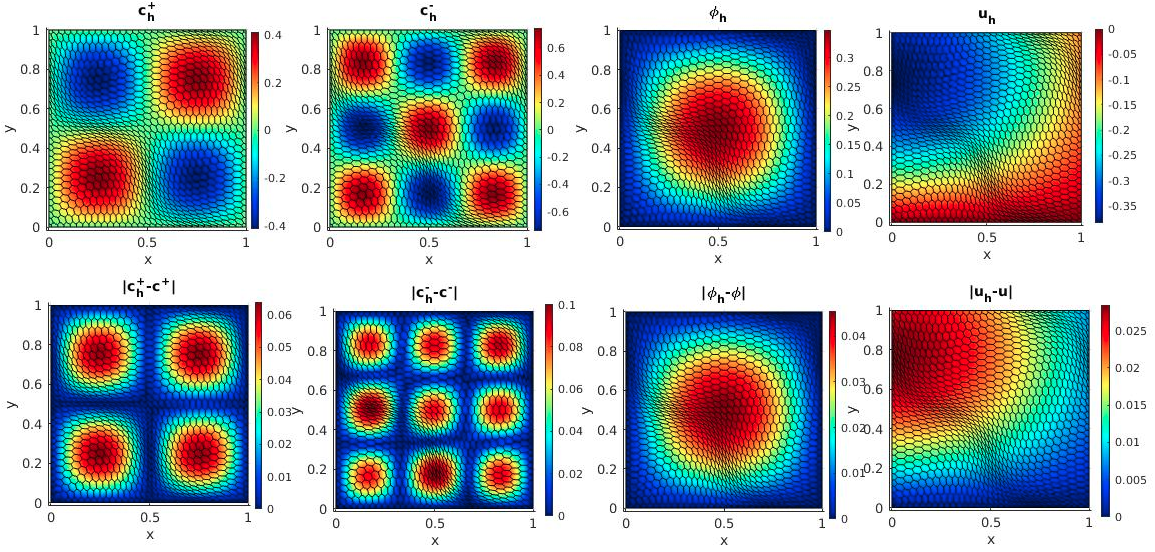}\\
\includegraphics[height=0.325\textwidth]{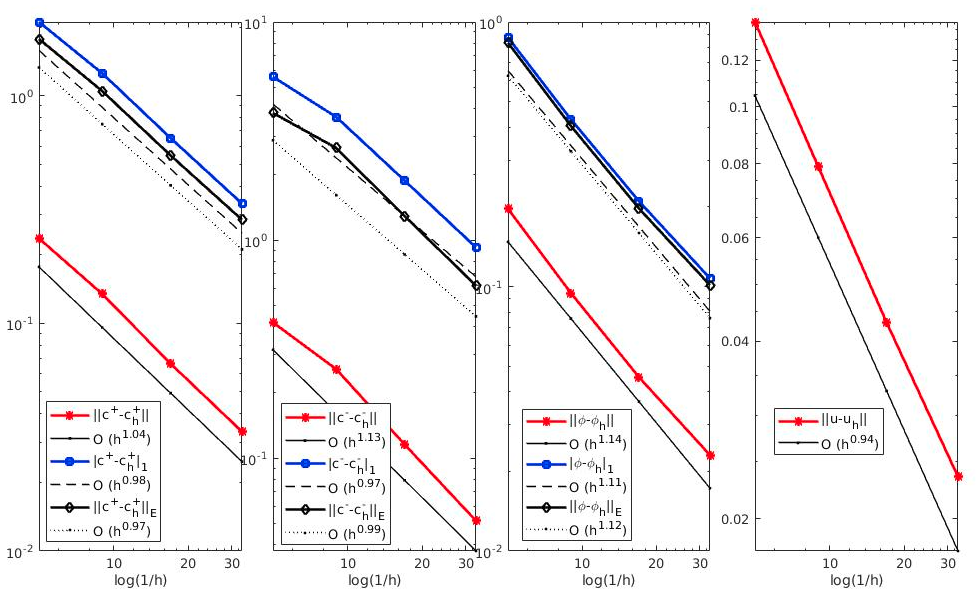}\\
\includegraphics[height=0.325\textwidth]{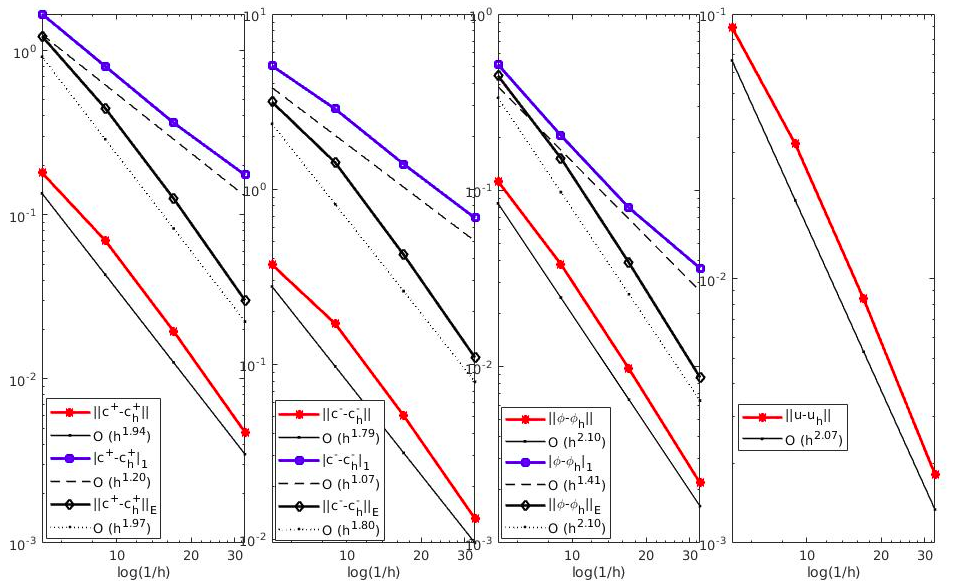}
\end{center}
	\vspace{-4mm}
	\caption{Example 1. Snapshots of numerical solutions $\{c_{1,h}^{n},c_{2,h}^{n},\phi_{h}^{n},\textbf{u}_{h}^{n}\}$ and its absolute error (top) and error history for the verification of convergence \rev{with $\tau =h$ (middle row) and
          $\tau = h^2$ (bottom)}.}\label{fig0}
\end{figure}

Approximate errors (computed with the aid of suitable projections) and the associated convergence rates generated on a sequence of successively refined grids (uniform hexagon meshes) are displayed in \rev{Fig. \ref{fig0}} by setting $\tau =h$ and $\tau =h^{2}$. One can see the second-order convergence for the total errors of all individual variables in the $L^{2}$-norm, and the first-order convergence for errors of concentrations and potential in the $H^{1}$-seminorm, which are in agreement with the theoretical analysis. \rev{The top panels of Fig. \ref{fig0} show samples of coarse-mesh approximate solutions together with absolute errors.}

\subsection{Example 2: Dynamics of the PNP/NS equations with initial discontinuous concentrations}
Now we investigate the dynamics of the system on the unit square with an initial value as follows (see \cite{Andreas09,Huadong17,Huadong18})
\begin{equation*}
c_{1,0}=\left\{\begin{array}{lll}
{1}&&{(0,1)^{2}\backslash\{(0,0.75)\times(0,1)\cup (0.75,1)\times (0,\frac{11}{20})\}},\\
\text{1e{-06}}&&\text{otherwise},
\end{array}\right.
\end{equation*}
\begin{equation*}
c_{2,0}=\left\{\begin{array}{lll}
{1}&&{(0,1)^{2}\backslash\{(0,0.75)\times(0,1)\cup (0.75,1)\times (\frac{9}{20},1)\}},\\
\text{1e{-06}}&&\text{otherwise}.
\end{array}\right.
\end{equation*}
and $\textbf{u}_{0}=\textbf{0}$. The  discontinuity of the initial concentrations represents an interface between the the electrolyte and the solid surfaces where  electrosmosis (transport of ions  from the electrolyte towards the solid surface) is expected to occur.  We consider a fixed time step of $\tau=$1e-03 and a coarse polygonal mesh with mesh size $h=1/64$. We show snapshots of the numerical solutions (concentrations and electrostatic potential)  at times $t_{F}=$2e-03, $t_{F}=$2e-02 and $t_{F}=$0.1 in Fig \ref{fig2}. All plots  confirm that the obtained results qualitatively match with those obtained in, e.g.,  \cite{Andreas09,Huadong17,Huadong18} (which use similar decoupling schemes).  Moreover, Fig. \ref{fig2a} shows that the total discrete energy is decreasing and the numerical solution is mass preserving during the evolution, which verifies numerically our findings from Theorems \ref{T:cons} and \ref{T:En}.

\begin{figure}[t!]
	\begin{center}
    \includegraphics[width=0.35\textwidth]{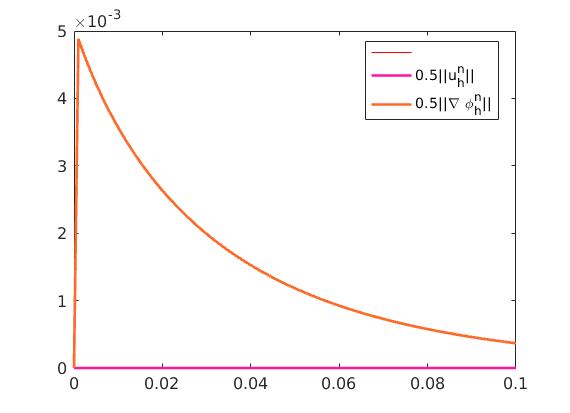}
     \includegraphics[width=0.35\textwidth]{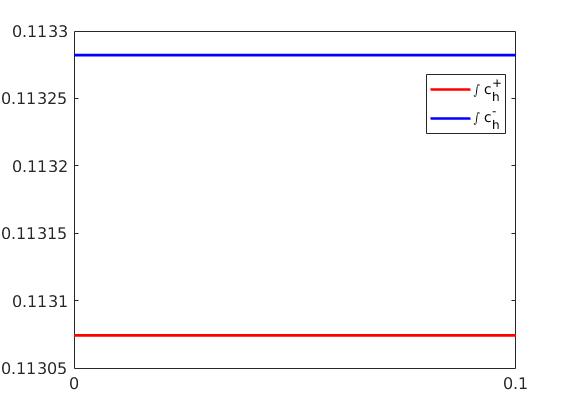}
    \end{center}

   \vspace{-4mm}
	\caption{Example 2.  Evolution of electric (and kinetic) energy (left) and global masses (right) with $\tau =1$e-3. }\label{fig2a}
\end{figure}

\begin{figure}[t!]
	\begin{center}
    \includegraphics[width=0.19\textwidth]{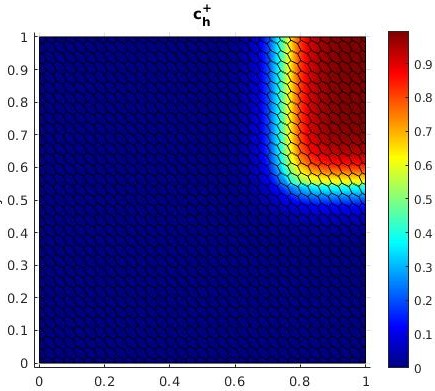}
    \includegraphics[width=0.19\textwidth]{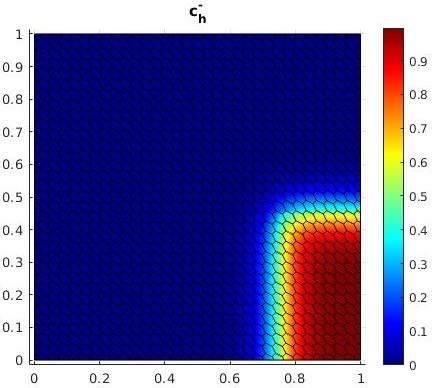}
    \includegraphics[width=0.19\textwidth]{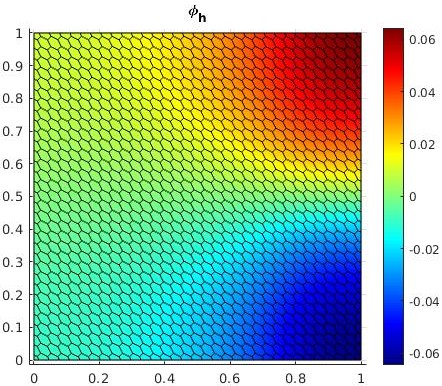}
   \includegraphics[width=0.19\textwidth]{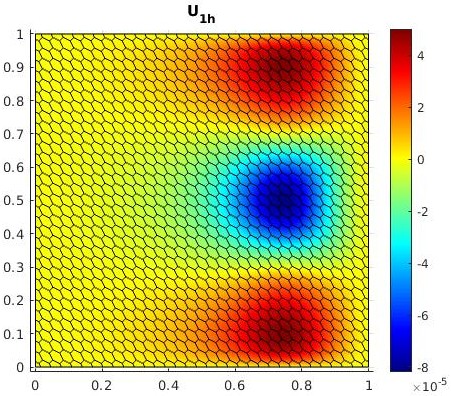}
    \includegraphics[width=0.19\textwidth]{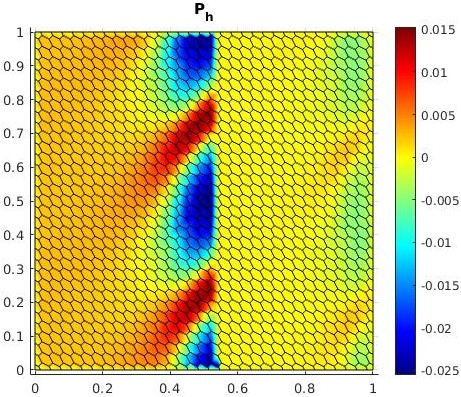}\\
   \includegraphics[width=0.19\textwidth]{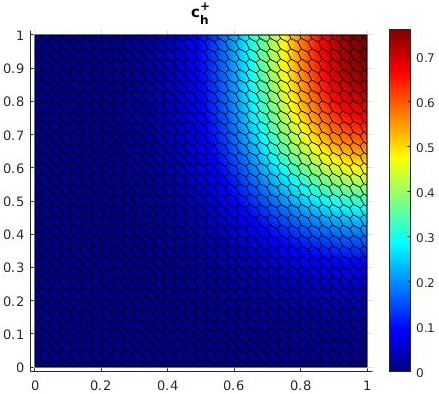}
    \includegraphics[width=0.19\textwidth]{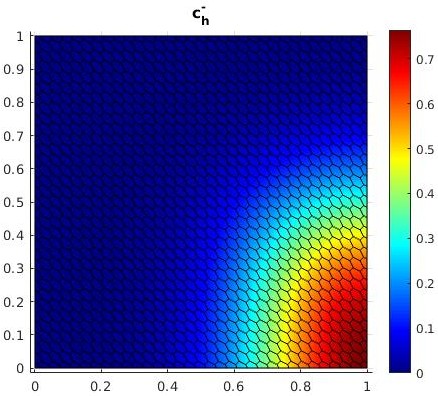}
    \includegraphics[width=0.19\textwidth]{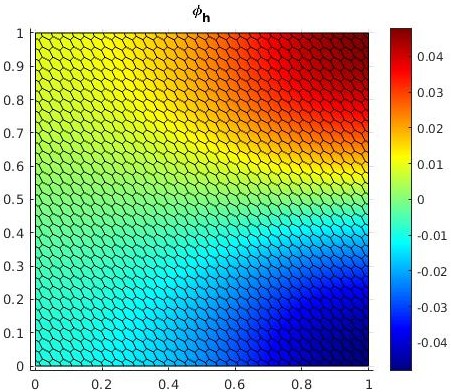}
   \includegraphics[width=0.19\textwidth]{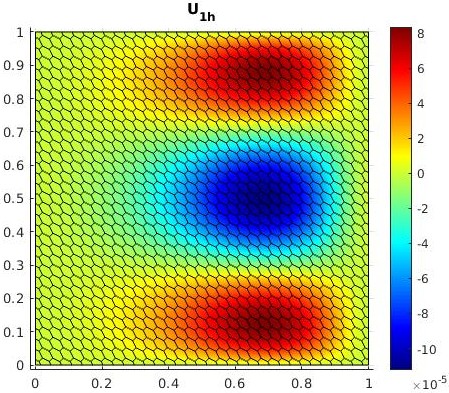}
   \includegraphics[width=0.19\textwidth]{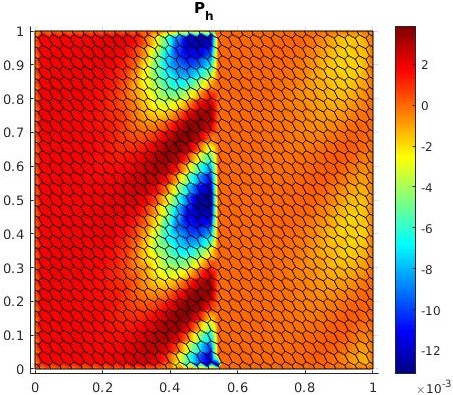}\\
    \includegraphics[width=0.19\textwidth]{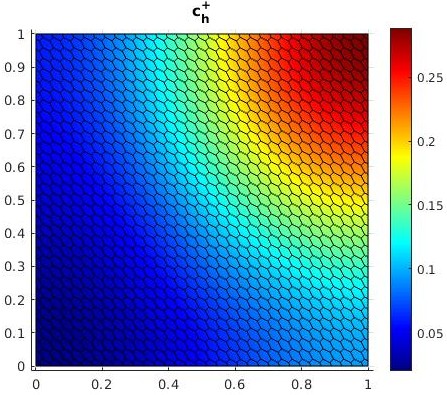}
   \includegraphics[width=0.19\textwidth]{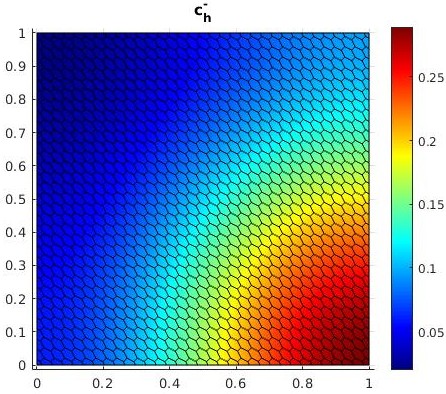}
    \includegraphics[width=0.19\textwidth]{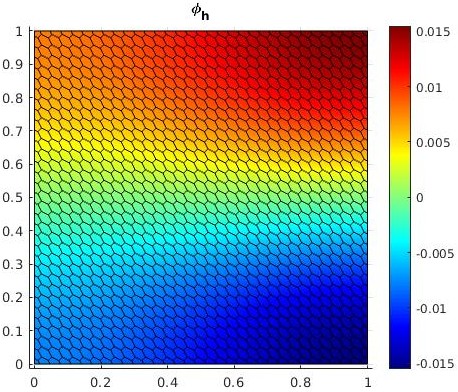}
   \includegraphics[width=0.19\textwidth]{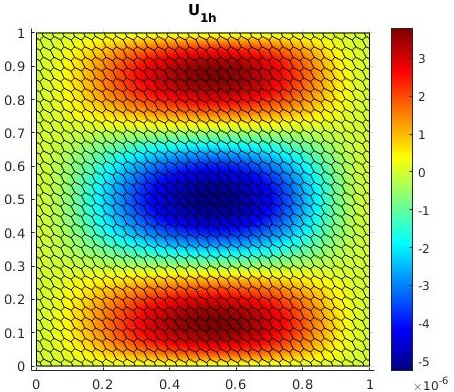}
    \includegraphics[width=0.19\textwidth]{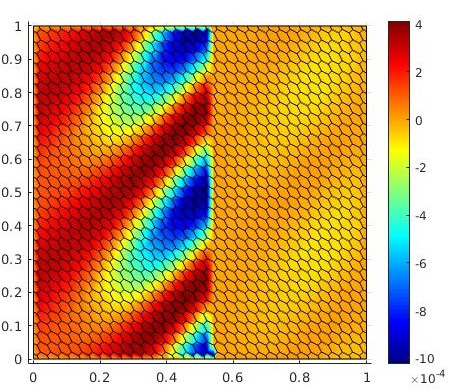}
    \end{center}
	\caption{Example 2. Snapshots of the approximate solutions $\{c_{1,h}^{n},c_{2,h}^{n},\phi_{h}^{n}, \textbf{u}_{h}^{n},p_{h}^{n}\}$ obtained with the proposed VEM, and shown at times $t_{F}=$2e-03 (top row), $t_{F}=$2e-02 (middle) and $t_{F}=$1e-01  (bottom). }\label{fig2}
\end{figure}

\subsection{Example 3: Application to water desalination}
The desalination of alternative waters, such as brackish and seawater, municipal and industrial wastewater, has become an increasingly important strategy for addressing water shortages and expanding traditional water supplies. Electrodialysis (ED) is a membrane desalination technology that uses semi-permeable ion-exchange membranes (IEMs) to selectively separate salt ions in water under the influence of an electric field \cite{Xu13}. An ED structure consists of pairs of cation-exchange membranes (CEMs) and anion-exchange membranes (ARMs), alternately arranged between a cathode and an anode (Figure \ref{fig00-01}, left). The driving force of ion transfer in the electrodialysis process is the electrical potential difference's applied between an anode and a cathode which causes ions to be transferred out of the aquatic environment and water purification. When an electric field is applied by the electrodes, the appearing charge at the anode surface becomes positive (and at the cathode surface becomes negative). The applied electric field causes positive ions (cations) to migrate to the cathode and negative ions (anions) to the anode. During the migration process, anions pass through anion-selective membranes but are returned by cation-selective membranes. A similar process occurs for cations in the presence of cationic and anionic membranes. As a result of these events, the ion concentration in different parts intermittently decreases and increases. Finally, an ion-free dilute solution and a concentrated solution as saline or concentrated water are out of the system.
In what follows we investigate the effects of the applied voltage and salt concentration on electrokinetic instability appearing  in ED processes. For this purpose, simulations of a binary electrolyte solution near a CEM are conducted. Since CEMs and AEMs have similar hydrodynamics and ion transport, the present findings can be applied to AEMs.

\begin{figure}[t!]
	\centering
\includegraphics[height=0.32\textwidth]{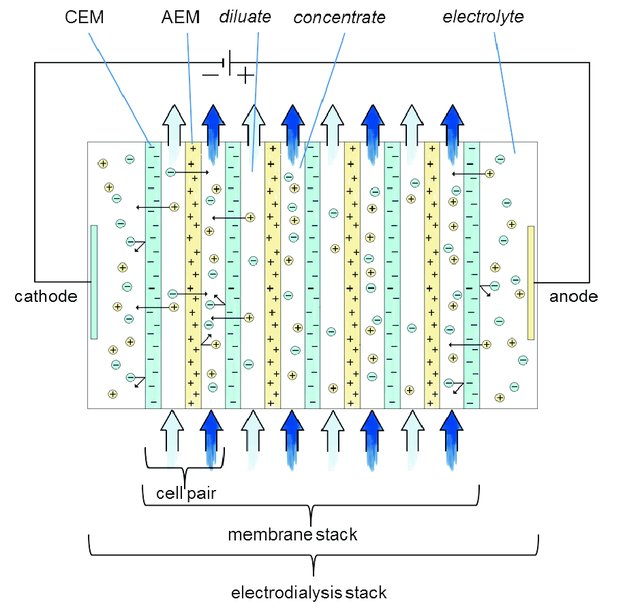}
\includegraphics[height=0.32\textwidth]{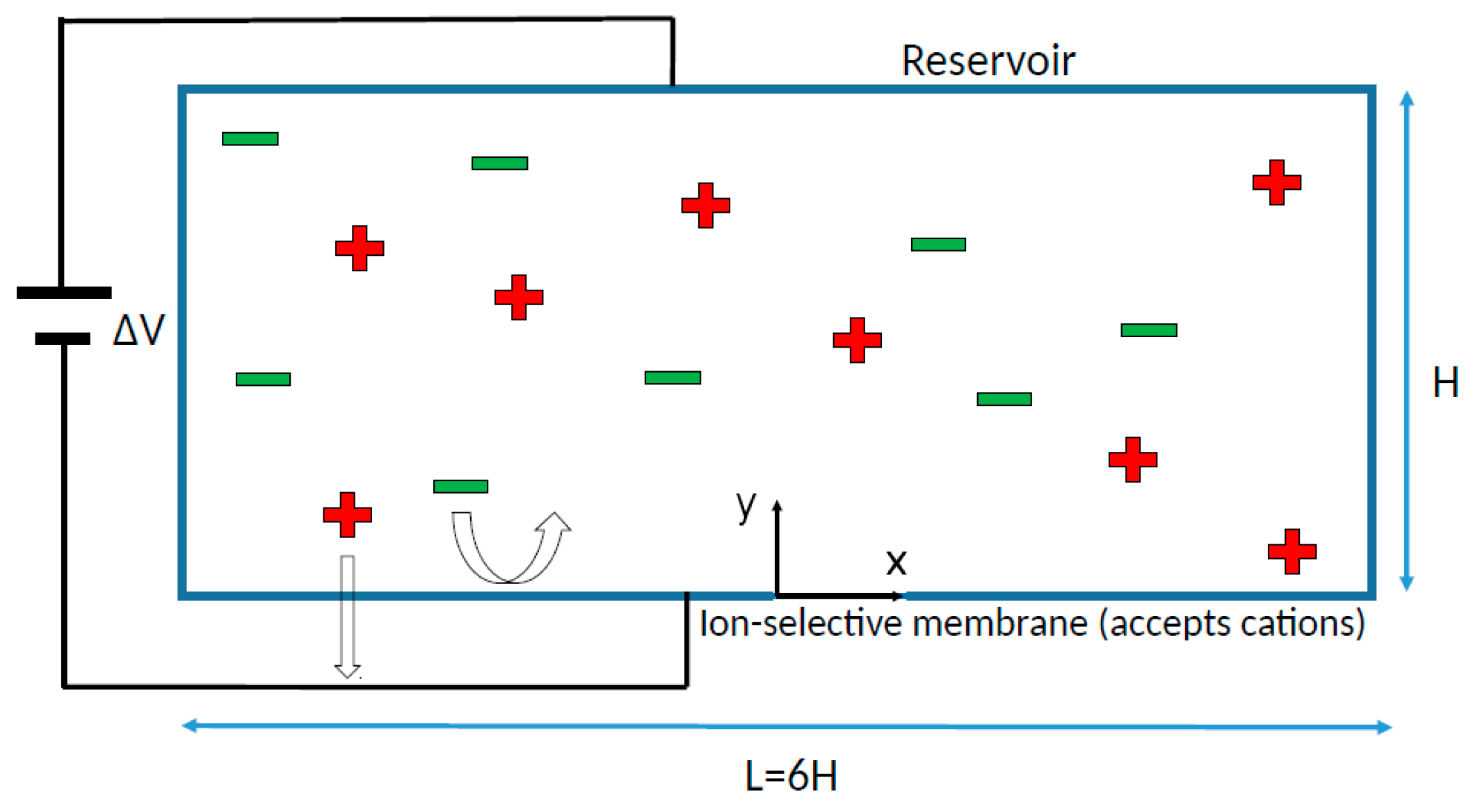}
	\caption{Example 3. Schematic of an electrodialysis stack \cite{Galama} (left) and simplified configuration of a 2D problem with ion-selective membrane from \cite{Mani13} (right). }\label{fig00-01}
\end{figure}

The simulations presented here are based on the 2D configuration used in  \cite{Mani13} (see also \cite{Karatay15,wang17}), consisting of a reservoir on top and a CEM at the bottom that allows cationic species to pass-through (Fig. \ref{fig00-01}, right). An electric field, i.e., $E=\frac{\Delta V}{H}$, is applied in the orientation perpendicular to the membrane and the reservoir. Here, we set $\Omega =[0,4]\times [0,1]$ and consider the  NS momentum balance equation using the following non-dimensionalization 
\[
\dfrac{1}{S_{c}}\left(\textbf{u}_{t} + (\textbf{u}\cdot\nabla)\textbf{u}\right)-\Delta \textbf{u}+\nabla p+\dfrac{\kappa}{\epsilon}(c_{1}-c_{2})\nabla \phi =\textbf{0}.
\]
The model parameters common to all considered cases are the Schmidt number 
$S_{c}=$1e-03, the rescaled Debye length $\epsilon =$2e-03, and  the electrodynamics coupling constant $ \kappa =0.5$. 
The initial velocity is zero, and the initial concentrations are determined by the randomly perturbed fields, that is:
\[
c_{1}(x,y,0)=\alpha\text{ rand}(x,y)(2-y),\quad\quad c_{2}(x,y,0)=\alpha\text{ rand}(x,y)y,
\]
where $\text{rand}(x,y)$ is a uniform random perturbation between $0.98$ and $1$. Mixed boundary conditions are set at the top $\partial\Omega_{top}$ and bottom $\Omega_{bot}$ segments of the boundary, and periodic boundary conditions on the vertical walls  $\partial\Omega_{lr}$ 
\begin{equation*}
\left\{\begin{array}{ll}
c_{1}=\alpha,\quad c_{2}=\alpha,\quad \phi =\beta,\quad \textbf{u}=\textbf{0},&\text{on}~\partial\Omega_{top},\\[1mm]
c_{1}=2\alpha,\quad \nabla c_{2}\cdot \textbf{n}=0,\quad \phi =0,\quad \textbf{u}=\textbf{0},& \text{on}~\partial\Omega_{bot},\\[1mm]
u(4,y,t)=u(0,y,t),\quad \forall u\in \{c_{1},c_{2},\phi,\textbf{u} \},& \text{on}~\partial \Omega_{lr}.
\end{array}\right.
\end{equation*}
where $\alpha$ and $\beta$ assume different values in the different simulation cases (see Table \ref{Tab1}). We utilize triangular meshes which are sufficiently refined towards the ion-selective membrane (i.e., $y=0$). The number of cells and the computational time step are listed in Table \ref{Tab1}, right columns.

\begin{table}
\centering
	{\footnotesize\begin{tabular}{l|ll|ll}
		\hline\hline
		{$\text{Case}$}&\hspace{.5cm}{$\alpha$}&\hspace{.5cm}{$\beta$}&\hspace{.5cm}{$\text{Number of
elements}$}&\hspace{.5cm}{$\text{Time step}$}\\
		\hline\hline
		 {3A$: \text{Baseline}$}&\hspace{.5cm}{$1$}&\hspace{.5cm}{$30,40,120$}&\hspace{.5cm}\footnotesize{$32\times 32$}&\hspace{.5cm}\footnotesize{1e-06}\\
		 {3B$: \text{Low}$}&\hspace{.5cm}{$10$}&\hspace{.5cm}{$120$}&\hspace{.5cm}\footnotesize{$400\times 100$}&\hspace{.5cm}\footnotesize{1e-07}\\
		 {3C$: \text{Medium}$}&\hspace{.5cm}{$100$}&\hspace{.5cm}{$120$}&\hspace{.5cm}\footnotesize{$400\times 100$}&\hspace{.5cm}\footnotesize{1e-08}\\
		\hline\hline
	\end{tabular}}
	\caption{Example 3. Model and discretization parameters to be varied according to each simulated case.}\label{Tab1}
\end{table}

\medskip
\noindent\textbf{Example 3A: Effect of the applied voltage.}  
Figs. \ref{fig3A1} and \ref{fig3A2} show images of the anion concentration, velocity, and electric potential for $V=30$ and $V=40$, representative of the 2D  baseline simulation. One can see, in the beginning, at times $t=$3e-03 for $V=30$ (and $t=$8e-04 for $V=40$), the solutions are still quite similar to the initial condition. As time progresses, electrokinetic instabilities (EKI) appear near the surface of the membrane. As a consequence of the EKI, the contours of vertical velocity show that disturbances are increasing. Higher voltages cause the instability to set in earlier. A periodic structure above the membrane can be observed after the disturbance amplitudes are high enough. Structures are seen at more anion concentrations than electrical potentials. The disturbances at times $t=$2e-02 ($V=30$) and 7e-03 ($v=40$) are strong enough, which cause a significant distortion in the electrical potential. The merging of neighboring structures leads to the formation of larger structures, as evidenced in the snapshot at 5e-02 for $V=40$.

\begin{figure}[t!]
	\begin{center}
   \includegraphics[width=0.325\textwidth]{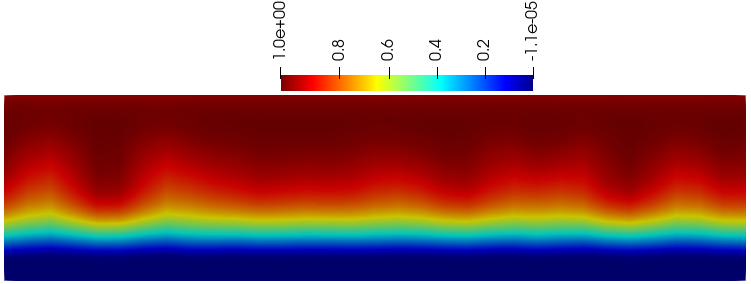}
     \includegraphics[width=0.325\textwidth]{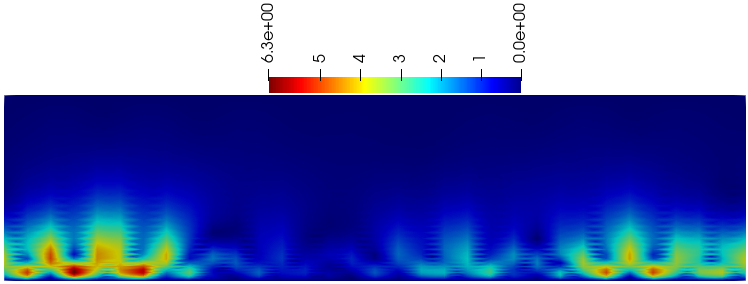}
     \includegraphics[width=0.325\textwidth]{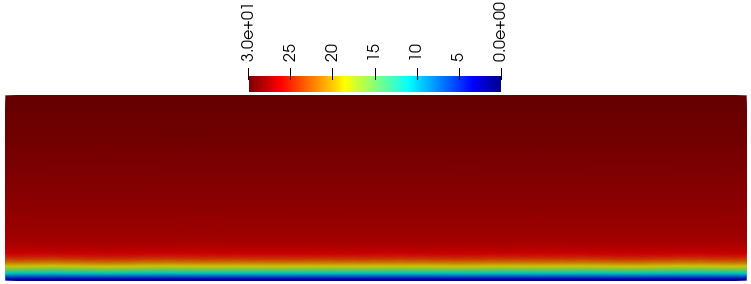}\\
   \includegraphics[width=0.325\textwidth]{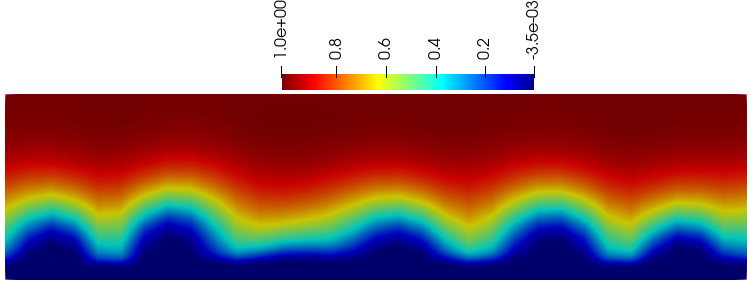}
     \includegraphics[width=0.325\textwidth]{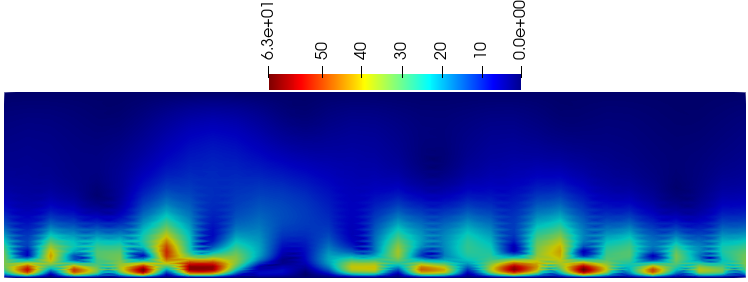}
     \includegraphics[width=0.325\textwidth]{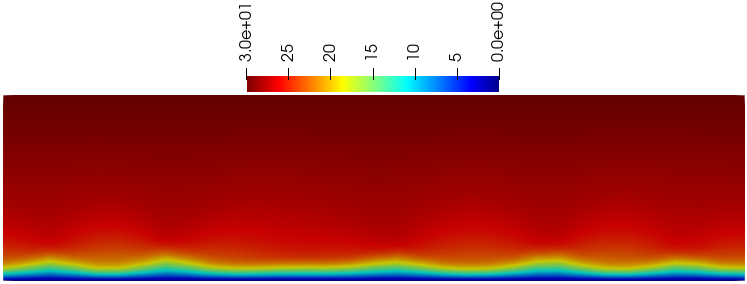}\\
    \includegraphics[width=0.325\textwidth]{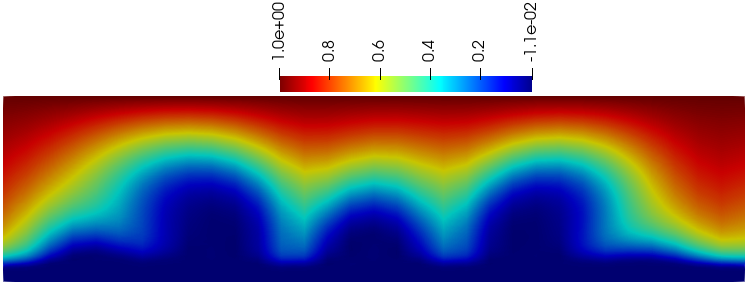}
   \includegraphics[width=0.325\textwidth]{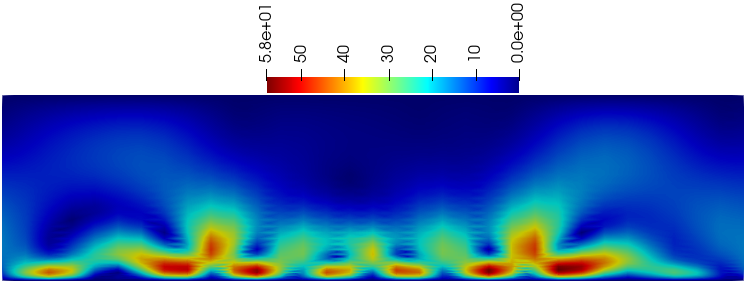}
     \includegraphics[width=0.325\textwidth]{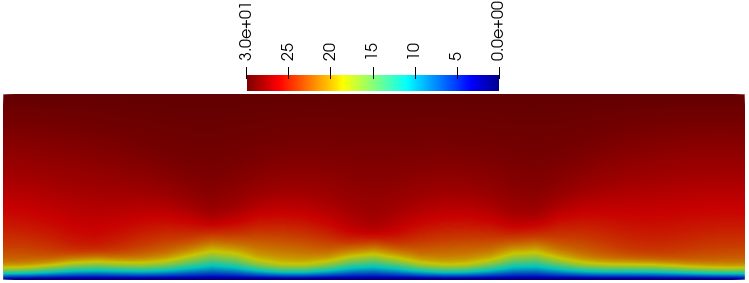}
\end{center}	

\vspace{-4mm}
\caption{Example 3A. Snapshots of numerical solutions $c_{2,h}$ (left), $\textbf{u}_{h}$ (middle) and $\phi_{h}$ (right) using the  \rev{proposed} VEM at times $t_{F}=$3e-03, $t_{F}=$2e-02, 
  and $t_{F}=$8e-02  with voltage $V=30$. }\label{fig3A1}
\end{figure}
\begin{figure}[t!]
	\begin{center}
\includegraphics[width=0.325\textwidth]{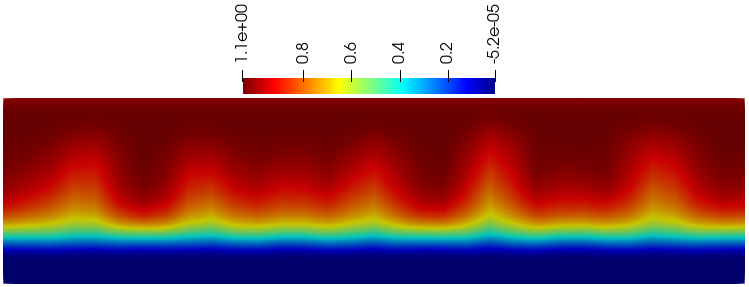}
     \includegraphics[width=0.325\textwidth]{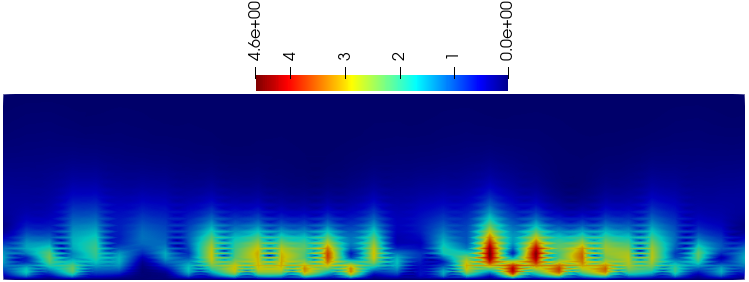}
     \includegraphics[width=0.325\textwidth]{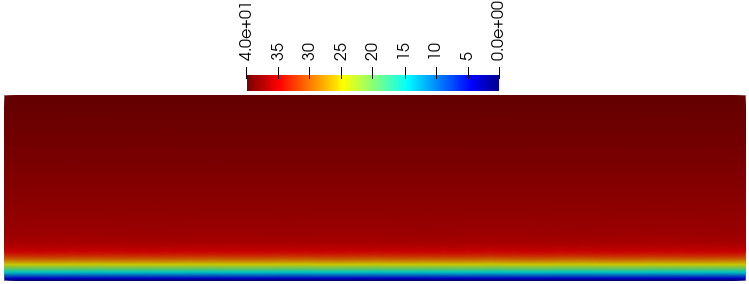}\\
\includegraphics[width=0.325\textwidth]{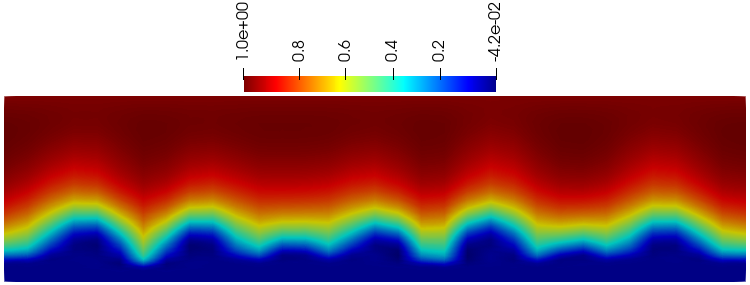}
     \includegraphics[width=0.325\textwidth]{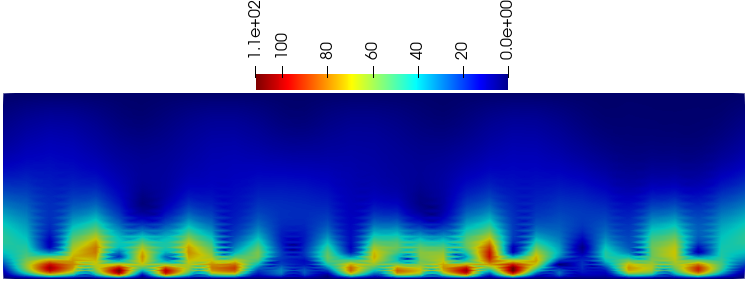}
     \includegraphics[width=0.325\textwidth]{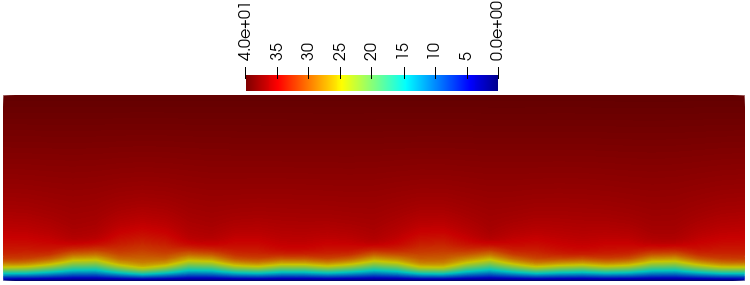}\\
\includegraphics[width=0.325\textwidth]{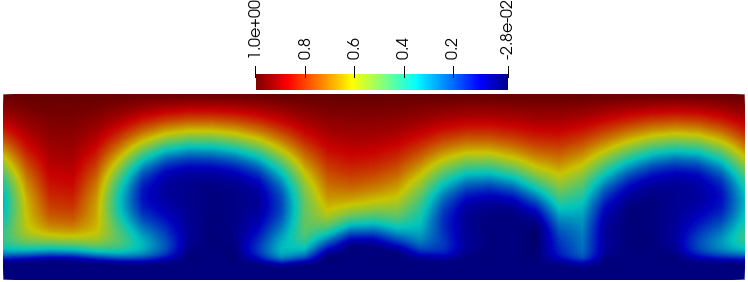}
\includegraphics[width=0.325\textwidth]{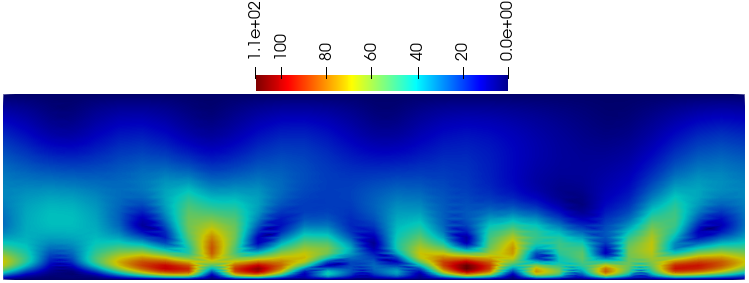}
\includegraphics[width=0.325\textwidth]{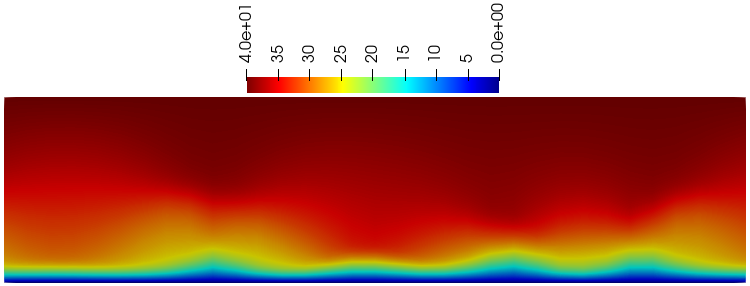}
\end{center}
	
\vspace{-4mm}
\caption{Example 3A. Snapshots of numerical solutions $c_{2,h}$ (left), $\textbf{u}_{h}$ (middle) and $\phi_{h}$ (right) using the  \rev{proposed} VEM at times $t_{F}=$3e-03, $t_{F}=$2e-02, 
  and $t_{F}=$8e-02  with voltage $V=40$. }\label{fig3A2}
\end{figure}
\begin{figure}[t!]
	\begin{center}
\includegraphics[width=0.325\textwidth]{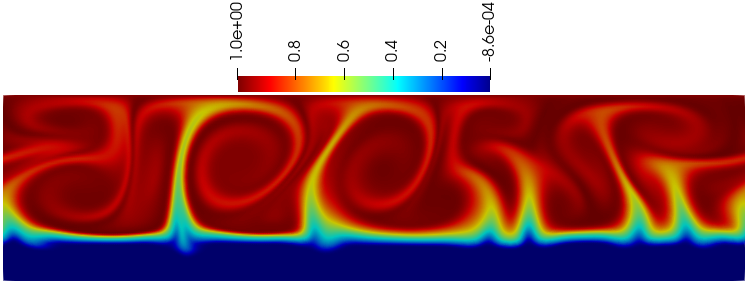}
     \includegraphics[width=0.325\textwidth]{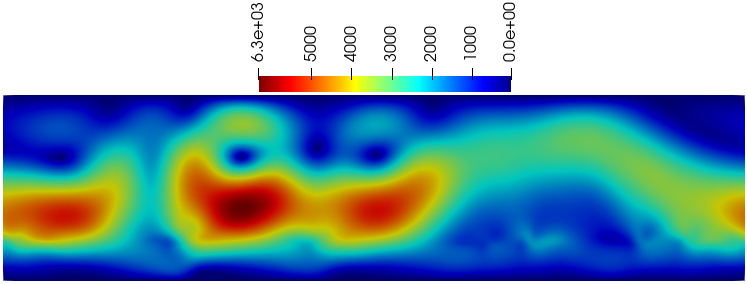}
     \includegraphics[width=0.325\textwidth]{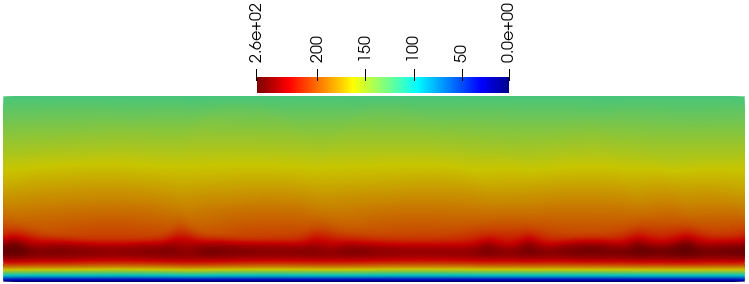}\\
\includegraphics[width=0.325\textwidth]{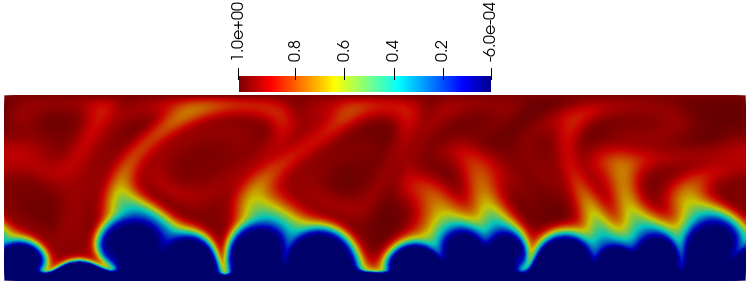}
     \includegraphics[width=0.325\textwidth]{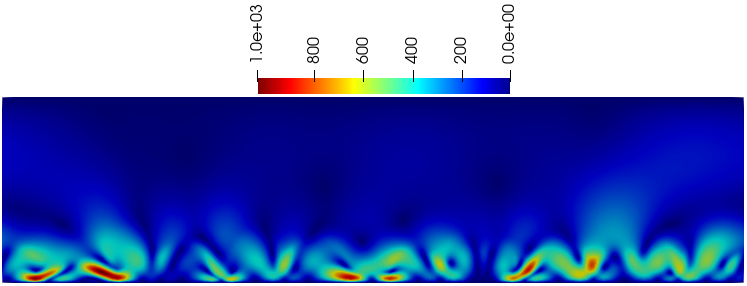}
     \includegraphics[width=0.325\textwidth]{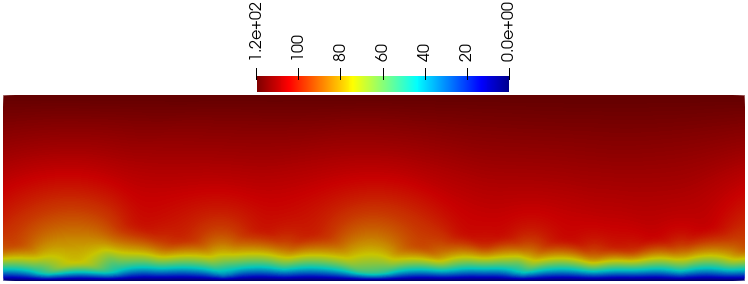}\\
\includegraphics[width=0.325\textwidth]{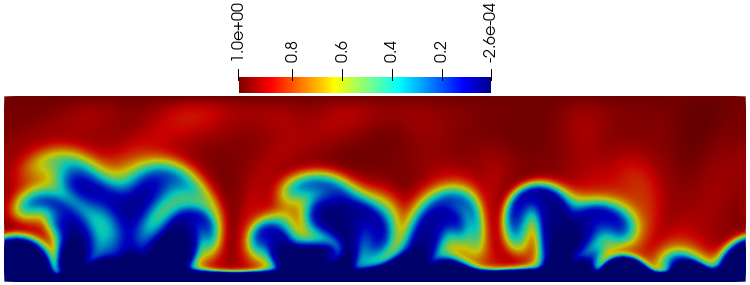}
   \includegraphics[width=0.325\textwidth]{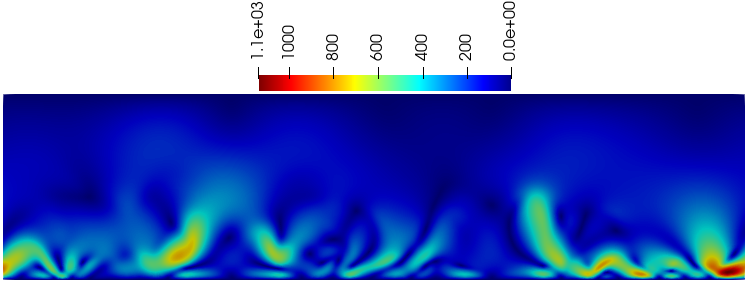}
     \includegraphics[width=0.325\textwidth]{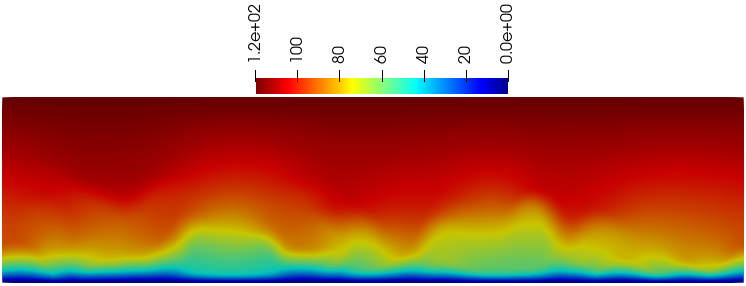}
	\end{center}
	
	\vspace{-4mm}
	\caption{Example 3A. Snapshots of numerical solutions $c_{2,h}$ (left), $\textbf{u}_{h}$ (middle) and $\phi_{h}$ (right) using the \rev{proposed} VEM at times $t_{F}=$3e-03, $t_{F}=$2e-02, 
          and $t_{F}=$8e-02  with voltage $V=120$. }\label{fig3A3}
\end{figure}

As it can be seen from Fig. \ref{fig3A3}, by increasing the voltage to $V=120$ the instability becomes stronger, the disturbances grow faster, and the structures appear earlier. Smaller structures have coalesced into bigger ones at time $t=$3.3e-03. Such a behavior is consistent with the results in \cite{Karatay15,Kim21} and it is fact similar to the encountered in fluid mechanics vortex fusion.

\medskip
\noindent\textbf{Example  3B: Effect of salt concentration.} 
Finally, we considered a fixed applied voltage  of $V=120$. A NaCl concentration of 10  was simulated for slightly brackish water and also increasing that concentration to 100   for moderately brackish water. By increasing the concentration, the structures reveal themselves earlier and their size decreases (see Fig. \ref{fig3B}, left). In the second case, structures appeared much sooner (and were much smaller). \rev{For the case of concentration 100, similar findings can be obtained (see Fig. \ref{fig3B}, right column). Based on this, it can be concluded that, in addition to voltage, the start of the instability depends also on the ion concentration.} 

\begin{figure}[t!]
	\begin{center}
\includegraphics[width=0.495\textwidth]{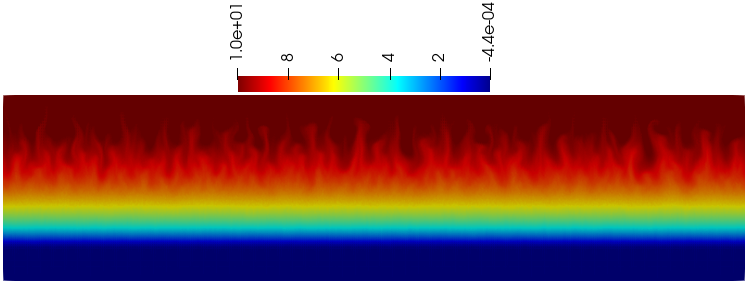}
     \includegraphics[width=0.495\textwidth]{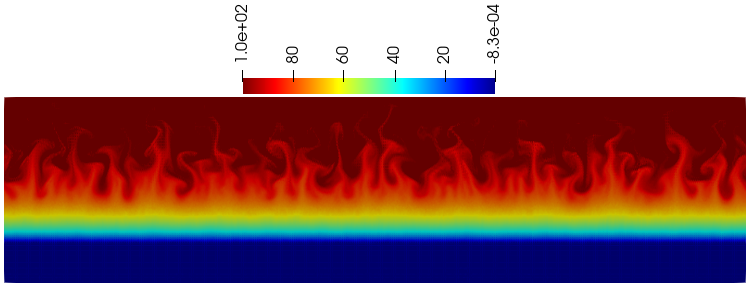}\\
\includegraphics[width=0.495\textwidth]{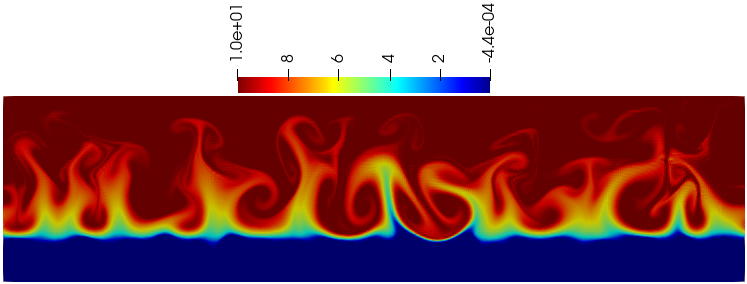}
     \includegraphics[width=0.495\textwidth]{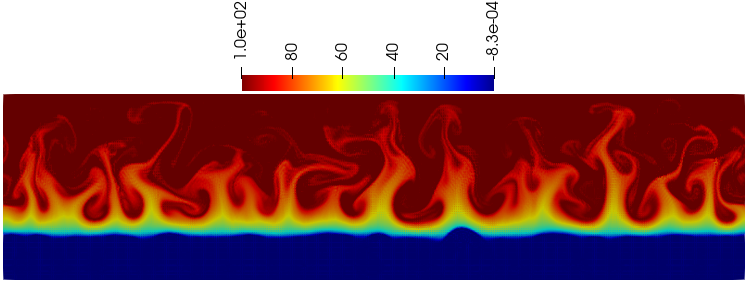}\\
\includegraphics[width=0.495\textwidth]{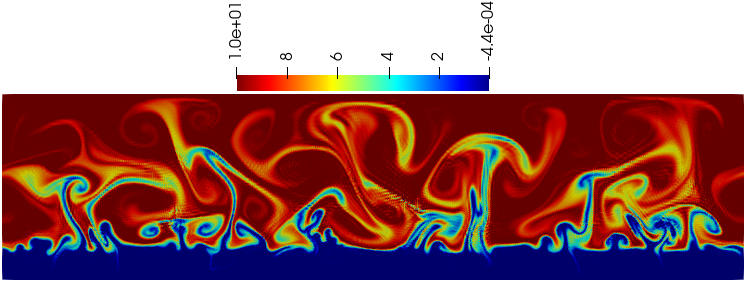}
     \includegraphics[width=0.495\textwidth]{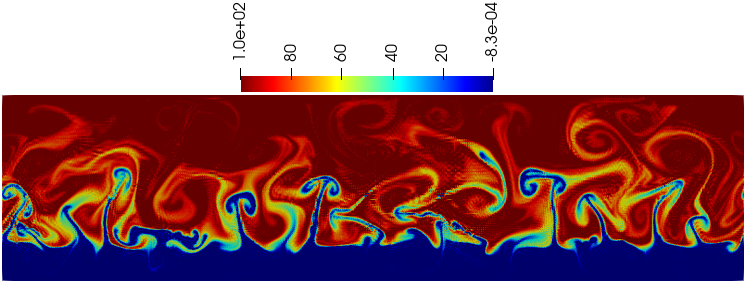}
	\end{center}
	
	\vspace{-4mm}
	\caption{Example 3B. Snapshots of numerical solutions $c_{2,h}$ with voltage $V=120$, for NaCl=10 
	  at times $t_{F}=$5e-07, $t_{F}=$2e-06, 
          and $t_{F}=$5e-06  (left); and for 
          NaCl = 100  	$t_{F}=$5e-07, $t_{F}=$1e-06, 
          and $t_{F}=$2e-06  (right). }\label{fig3B}
\end{figure}



\begin{thebibliography}{99}
\small 
\bibitem{Ahmad}
B. Ahmad, A. Alsaedi, F. Brezzi, L.D. Marini, A. Russo,  Equivalent projectors for virtual element methods, Comput. Math. Appl., 66 (3) (2013), 376-391.

\bibitem{Brezzi13}
L.~{Beir\~ao da Veiga}, F.~{Brezzi}, A.~{Cangiani}, G.~{Manzini}, L.~D. {Marini}  and A.~{Russo},  {Basic principles of virtual element methods}, {Math. Models Methods Appl. Sci.}, 23 (2013), 199--214.

\bibitem{Brezzi14}
L.~{Beir\~ao da Veiga}, F.~{Brezzi}, L.~D.~{Marini} and A.~{Russo},  {The Hitchhiker's guide to the virtual element method}, {Math. Models Methods Appl. Sci.}, 24~(8) (2014), 1541--1573.

\bibitem{Lovadina17}
L.~{Beir\~ao da Veiga}, C.~{Lovadina} and A.~{Russo},  {Stability analysis for the virtual element method}, {Math. Models Methods Appl. Sci.}, 27~(13) (2017), 2557--2594.

\bibitem{Beir17}
L.~{Beir\~ao da Veiga}, C.~{Lovadina} and G.~{Vacca},  {Divergence free virtual elements for the Stokes problem on polygonal meshes}, {ESAIM, Math. Model. Numer. Anal.}, 51~(2) (2017), 509--535.

\bibitem{Beir18N}
L. Beir\~ao {da Veiga}, C. {Lovadina} and G. {Vacca},  {Virtual elements for the Navier-Stokes problem on polygonal meshes}, {SIAM J. Numer. Anal.}, 56~(3) (2018), 1210--1242.

\bibitem{Beir19}
L.~{Beir\~ao da Veiga}, D. Mora and G.~{Vacca}, The Stokes complex for virtual elements with application to Navier-Stokes flows. J. Sci. Comput. 81(2) (201), 990--1018. 


\bibitem{Brenner17}
S.~C.~{Brenner}, Q.~{Guan} and L.~Y.~{Sung},  {Some estimates for virtual element methods}, {Comput. Methods Appl. Math.}, 17~(4) (2017), 553--574.


\bibitem{Marini14}
F.~{Brezzi}, R.~ S. ~{Falk} and L.~D.~{Marini},  {Basic principles of mixed virtual element methods}, {ESAIM, Math. Model. Numer. Anal.}, 48~(4) (2014), 1227--1240.


\bibitem{Gatica17}
E. {C\'aceres} and G. N. {Gatica},  {A mixed virtual element method for the pseudostress-velocity formulation of the Stokes problem}, {IMA J. Numer. Anal.}, 37~(1) (2017), 296--331.

\bibitem{Cangiani20I}
A. {Cangiani}, P. {Chatzipantelidis}, G. {Diwan} and E. H. {Georgoulis},  {Virtual element method for quasilinear elliptic problems}, {IMA J. Numer. Anal.}, 40~(4) (2020), 2450--2472.

\bibitem{Cangiani16}
A.~{Cangiani}, V.~{Gyrya} and G.~{Manzini},  {The nonconforming virtual element method for the Stokes equations}, {SIAM J. Numer. Anal.}, 54~(6) (2016), 3411--3435.

\bibitem{Cangiani17}
A.~{Cangiani}, G.~{Manzini} and O.~J.~{Sutton},  {Conforming and nonconforming virtual element methods for elliptic problems}, {IMA J. Numer. Anal.}, 37~(3) (2017), 1317--1354.


\bibitem{Chen8}
L.~{Chen} and J.~{Huang},  {Some error analysis on virtual element methods}, {Calcolo}, 55~(1) (2018), 23.


\bibitem{Choi05}
H. {Choi} and M. {Paraschivoiu}, {Advanced hybrid-flux approach for output bounds of electroosmotic flows: adaptive refinement and direct equilibrating strategies}, {Microfluid. Nanofluid.}, 2~(2)~(2005) 154--170.

\bibitem{Cioffi06}
M. {Cioffi}, F. {Boschetti}, M. T. {Raimondi} and G. {Dubini}, {Modeling evaluation of the fluidynamic microenvironment in tissue-engineered constructs: A micro-CT based model}, {Biotechnol. Bioeng.}, 93~(3) (2006) 500--510.

\bibitem{Gharibi21}
M.~{Dehghan} and Z.~{Gharibi},  {Virtual element method for solving an inhomogeneous Brusselator model with and without cross-diffusion in pattern formation}, {J. Sci. Comput.}, 89~(1) (2021), 16.

\bibitem{Dreyer13} W. Dreyer, C. Guhlke and R. M\"uller, Overcoming the shortcomings of the Nernst--Planck model, Phys. Chem. Chem. Phys., 15 (19) (2013),  7075--7086. 


\bibitem{Mani13}
C. {Druzgalski}, M. {Andersen} and A. {Mani}, {Direct numerical simulation of electroconvective instability and hydrodynamic chaos near an ion-selective surface}, {Phys. Fluids.}, 25 (2013), 110804.

\bibitem{Huadong17}
H.~{Gao} and D.~{He}, {Linearized conservative finite element
  methods for the Nernst-Planck-Poisson equations}, {J. Sci. Comput.}, 72
  (2017), 1269--1289.

\bibitem{Huadong18}
H.~{Gao} and P.~{Sun},  {A linearized local conservative mixed finite
  element method for Poisson-Nernst-Planck equations}, {J. Sci. Comput.}, 77
  (2018), 793--817.

\bibitem{Gatica18}
G. N. {Gatica}, M. {Munar} and F. {Sequeira},  {A mixed virtual element method for the Navier-Stokes equations}, {Math. Models Methods Appl. Sci.}, 28~(14) (2018), 2719--2762.

\bibitem{Gross19}
A.~{Gross}, A.~{Morvezen}, P.~{Castillo}, X.~{Xu} and P.~{Xu}, {Numerical Investigation of the Effect of Two-Dimensional Surface Waviness on the Current Density of Ion-Selective Membranes for Electrodialysis}, {Water}, 11~(7) (2019), 1397.

\bibitem{Galama}
O. {Galama}, {Ion exchange membranes in seawater applications Processes and Characteristics} Ph.D Thesis, 2015.

\bibitem{Gharibi20}
Z. {Gharibi}, M. {Dehghan}, M. {Abbaszadeh}, {Numerical analysis of locally conservative weak Galerkin dual-mixed finite element method for the time-dependent Poisson--Nernst--Planck system}, {Comput. Math. Appl.}, 92 (2021) 88--108.

\bibitem{he_numpde17} {M. He and P. Sun}, {Error analysis of mixed finite element method for Poisson- Nernst-Planck system}, Numer. Methods Partial Diff. Eqns., 33 (2017), 1924--1948.

\bibitem{he_jcam18} {M. He and P. Sun}, {Mixed finite element analysis for the Poisson-Nernst-Planck/Stokes coupling}, J. Comput. Appl. Math., 341 (2018),  61--79.


\bibitem{Hu05}
Y. {Hu}, J. S. {Lee}, C. {Werner} and D. {Li}, {Electrokinetically controlled concentration gradients
in micro-chambers in microfluidic systems}, {Microfluid. Nanofluid.}, 2~(2) (2005) 141--153.



\bibitem{Jerome02}
J. W. {Jerome}, {Analytical approaches to charge transport in a moving medium}, {Transp. Theory Stat. Phys.}, 31 (2002) 333--366.

\bibitem{Jerome85} 
J. W. {Jerome}, {Consistency of semiconductor modeling: an existence/stability analysis
for the stationary Van Boosbroeck system}, {SIAM J. Appl. Math.}, 45 (1985) 565--590.

\bibitem{Jerome11}
J. W. {Jerome}, {The steady boundary value problem for charged incompressible fluids:
PNP/Navier-Stokes systems}, {Nonlinear Anal.}, 74 (2011) 7486--7498.

\bibitem{Jerome08}
J. W. {Jerome}, B. {Chini}, M. {Longaretti} and R. {Sacco}, {Computational modeling and simulation of complex systems in bio-electronics}, {J. Comput. Electron.}, 7~(1) (2008) 10--13.

\bibitem{Karatay15} E. Karatay, C. L. Druzgalski and A. Mani, Simulation of chaotic electrokinetic transport: Performance of commercial software versus custom-built direct numerical simulation codes. J. Colloid Interf. Sci., 446 (2015), 67--76. 

\bibitem{Kim21} S. Kim, M. A. Khanwalea, R. K. Anand  and B. Ganapathysubramanian, Computational framework for resolving boundary layers in electrochemical systems using weak imposition of Dirichlet boundary conditions. Submitted preprint (2021). 

\bibitem{Linga20} G. Linga, A. Bolet  and J. Mathiesen, Transient electrohydrodynamic flow with concentration-dependent fluid properties: Modelling and energy-stable numerical schemes. J. Comput. Phys., 412 (2020), e109430. 


\bibitem{liu19} X. Liu and Z. Chen, The nonconforming virtual element method for the Navier-Stokes equations. Adv.
Comput. Math. 45(1), (2019), 51--74. 

\bibitem{Liu21}
Y.~{Liu}, S.~{Shu}, H.~{Wei} and Y.~{Yang},  {A virtual element method for the steady-state Poisson-Nernst-Planck equations on polygonal meshes}, {Comput. Math. Appl.}, 102 (2021), 95--112.

\bibitem{Lu10} B. Lu, M. Holst, J. McCammon and Y. Zhou, Poisson--Nernst--Planck equations for simulating biomolecular diffusion--reaction processes I: finite element solutions. J. Comput. Phys., 229 (2010), 6979--6994. 

\bibitem{Mauri15} A. Mauri, A. Bortolossi, G. Novielli and R. Sacco, 3D finite element modeling and simulation of industrial semiconductor devices including impact ionization. J. Math. Ind., 5 (2015), e18.

 
 \bibitem{Park97} J.-H. Park and J. W. Jerome, Qualitative properties of steady-state Poisson-Nernst-Planck systems: mathematical study. SIAM J. Appl. Math., 57(3) (1997),   609--630.
 
 

\bibitem{Andreas09}
A.~{Prohl} and M.~{Schmuck}, {Convergent discretizations for the
  Nernst-Planck-Poisson system}, {Numer. Math.}, 111 (2009), 591--630.
  
  \bibitem{prohl10} {A. Prohl and M. Schmuck}, {Convergent finite element discretizations of the Navier-Stokes-Nernst-Planck-Poisson system}, ESAIM Math. Model. Numer. Anal., 44 (2010), 531--571.


\bibitem{Ryham06} 
R. J. {Ryham}, {An energetic variational approach to mathematical modeling of charged
fluids: Charge phases}, simulation and well posedness, {Doctoral dissertation}, The Pennsylvania State University (2006).

\bibitem{Schmuck09}
M. {Schmuck}, {Analysis of the Navier-Stokes-Nernst-Planck-Poisson system}, {Math. Models Methods Appl. Sci.}, 19~(6) (2009) 993--1015.  


\bibitem{Vacca15}
\rev{G. {Vacca} and L. Beir\~ao {da Veiga},  {Virtual element methods for parabolic problems on polygonal meshes}, {Numer. Methods Partial Differ. Equations.}, 31~(6) (2015), 2110--2134.}

\bibitem{verma21} N. Verma  and S. Kumar, Virtual element approximations for non-stationary Navier-Stokes equations on polygonal meshes, Submitted preprint (2021). 

\bibitem{wang17} C. Wang, J. Bao, W. Pan and X. Sun, Modeling electrokinetics in ionic liquids. Electrophoresis 00 (2017) 1--13. 

\bibitem{wang21} G. Wang, F. Wang  and Y. He, A divergence-free weak virtual element method for the Navier-Stokes equation on polygonal meshes. Adv. Comput. Math., 47 (2021), e83. 

\bibitem{Wei21}
H.~{Wei}, X.~{Huang} and A.~{Li},  {Piecewise divergence-free nonconforming virtual elements for Stokes problem in any dimensions}, {SIAM J. Numer. Anal.}, 59~(3) (2021), 1835--1856.
%

\bibitem{Xu13}
P.~{Xu}, M.~{Capito} and T. Y.~{Cath}, {Selective removal of arsenic and monovalent ions from brackish water reverse osmosis concentrate}, {J. Hazard. Mater.}, 260~ (2013), 885--891.

\bibitem{xie20} D. Xie and B. Lu, An effective finite element iterative solver for a Poisson--Nernst--Planck ion channel model with periodic boundary conditions, SIAM J. Sci. Comput., 42(6) (2020), B1490--B1516. 

\end{thebibliography}
\end{document}